\documentclass[11pt]{article}

\title{Superposition of COGARCH processes}

\author{Anita Behme\thanks{Center for Mathematical Sciences, Technische Universit\"at M\"unchen,  85748 Garching, Boltzmannstrasse 3, Germany, e-mail: behme@ma.tum.de, carsten.chong@tum.de and cklu@ma.tum.de},
Carsten Chong$^\ast$\hspace{-5pt}, and
Claudia Kl\"uppelberg$^\ast$
}
\usepackage{amssymb}
\usepackage{amsfonts}
\usepackage{amsmath}
\usepackage{amsthm}
\usepackage{graphicx}
\usepackage{caption}
\usepackage{subcaption}
\usepackage[usenames, dvipsnames]{xcolor}
\usepackage{verbatim}
\usepackage{dsfont}
\usepackage{color}

\numberwithin{equation}{section}

\newtheorem{theorem}{Theorem}[section]
\newtheorem{lemma}[theorem]{Lemma}
\newtheorem{remark}[theorem]{Remark}
\newtheorem{example}[theorem]{Example}
\newtheorem{proposition}[theorem]{Proposition}
\newtheorem{definition}[theorem]{Definition}

\newtheorem{corollary}[theorem]{Corollary}
\newtheorem{fig}[theorem]{Figure}

\newcommand{\bthe}{\begin{theorem}}
\newcommand{\ethe}{\end{theorem}}

\newcommand{\ben}{\begin{enumerate}}
\newcommand{\een}{\end{enumerate}}

\newcommand{\bit}{\begin{itemize}}
\newcommand{\eit}{\end{itemize}}

\newcommand{\beq}{\begin{equation}}
\newcommand{\eeq}{\end{equation}}

\newcommand{\ble}{\begin{lemma}}
\newcommand{\ele}{\end{lemma}}

\newcommand{\bde}{\begin{definition}\rm}
\newcommand{\ede}{\halmos\end{definition}}

\newcommand{\bco}{\begin{corollary}}
\newcommand{\eco}{\end{corollary}}

\newcommand{\bpr}{\begin{proposition}}
\newcommand{\epr}{\end{proposition}}

\newcommand{\brem}{\begin{remark}\rm}
\newcommand{\erem}{\end{remark}}

\newcommand{\bproof}{\begin{proof}}
\newcommand{\eproof}{\end{proof}}

\newcommand{\bexam}{\begin{example}\rm}
\newcommand{\eexam}{\end{example}}

\newcommand{\bfi}{\begin{fig}}
\newcommand{\efi}{\end{fig}}

\newcommand{\btab}{\begin{tab}}
\newcommand{\etab}{\end{tab}}

\newcommand{\beao}{\begin{eqnarray*}}
\newcommand{\eeao}{\end{eqnarray*}\noindent}

\newcommand{\beam}{\begin{eqnarray}}
\newcommand{\eeam}{\end{eqnarray}\noindent}

\newcommand{\barr}{\begin{array}}
\newcommand{\earr}{\end{array}}

\newcommand{\bdis}{\begin{displaymath}}
\newcommand{\edis}{\end{displaymath}\noindent}

\def\N{{\mathbb N}}

\bigskip

\def\E{{\mathbb E}}

\def\RR{{\mathbb R}}

\newcommand{\NN}{\mathbb{N}}
\newcommand{\FF}{\mathbb{F}}
\newcommand{\GG}{\mathbb{G}}

\newcommand{\PP}{\mathbb{P}}

\def\calb{{\mathcal{B}}}

\def\calf{{\mathcal{F}}}

\def\calg{{\mathcal{G}}}
\def\calp{{\mathcal{P}}}

\def\calo{{\mathcal{O}}}

\def\cF{{\mathcal{F}}}
\def\cG{{\mathcal{G}}}

\newcommand{\bb}{\mathrm{b}}
\newcommand{\ee}{\mathrm{e}}
\newcommand{\ii}{\mathrm{i}}
\newcommand{\dd}{\mathrm{d}}
\newcommand{\tvp}{{\tilde\vp}}

\newcommand{\bbn}{\mathbb{N}}

\newcommand{\bbr}{\mathbb{R}}

\newcommand{\bbe}{\mathbb{E}}

\newcommand{\bbf}{\mathbb{F}}
\newcommand{\bbg}{\mathbb{G}}

\def\bone{{\mathds 1}}

\newcommand{\eqd}{\stackrel{\mathrm{d}}{=}}

\newcommand{\la}{{\lambda}}
\newcommand{\La}{{\Lambda}}
\newcommand{\ga}{{\gamma}}

\newcommand{\si}{{\sigma}}
\newcommand{\vp}{\varphi}
\newcommand{\eps}{\varepsilon}
\newcommand{\om}{\omega}
\newcommand{\Om}{\Omega}

\newcommand{\var}{{\rm Var}}
\newcommand{\cov}{{\rm Cov}}

\newcommand{\wt}{\widetilde}

\newcommand{\halmos}{\quad\hfill\mbox{$\Box$}}  

\allowdisplaybreaks[4]

\begin{document}


\maketitle

\begin{abstract}
We suggest three superpositions of COGARCH (supCOGARCH) volatility processes driven by L\'evy processes or L\'evy bases. 
We investigate second-order properties, jump behaviour, and prove that they exhibit Pareto-like tails. Corresponding price processes are defined and studied. 
We find that the supCOGARCH models allow for more flexible autocovariance structures than the COGARCH. 
Moreover, in contrast to most financial volatility models, the supCOGARCH processes do not exhibit a deterministic relationship between price and volatility jumps. 
Furthermore, in one supCOGARCH model not all volatility jumps entail a price jump, while in another supCOGARCH model not all price jumps necessarily lead to volatility jumps. 
\end{abstract}

\noindent
\begin{tabbing}
\> secondary: \,\,\,60G51, 60G57, 60H05
\end{tabbing}

\vspace{0.4cm}



\section{Introduction}\label{s1}

GARCH models have been used throughout the last decades to model returns sampled at regular intervals on stocks, currencies and other assets. They capture many of the stylized features of such data; e.g. heavy tails, volatility clustering and dependence without correlation. Also because of their interesting probabilistic properties as solutions to stochastic recurrence equations, they have attracted research by probabilists and statisticians; e.g. \cite{FZ}. Various attempts have been made to capture the stylized features of financial time series using continuous-time models. The interest in continuous-time models originates in the current wide-spread availability of irregularly spaced and high-frequency data.
There was a long debate, whether price and volatility fluctuations are caused by jumps or not. 
This question was answered convincingly in previous years by Jacod and collaborators, who developed sophisticated statistical tools to extract jumps of price and volatility out of high-frequency data (cf. \cite{Jacod10, JKMc,AJ} and references therein).

A prominent continuous-time model is the stochastic volatility model of Barndorff-Nielsen and Shephard \cite{Barn:Shep:2001}, in which the volatility process $V$ and the martingale part of the log asset price $G$ satisfy the equations 
\beam\label{eq1.1}
\dd V_t &=& -\la V_t\,\dd t + \dd L_{\la t},\\
\dd G_t &=&  \sqrt{V_t} \,\dd W_t + \rho \,\dd\wt L_{\la t},  \nonumber
\eeam
where $\la>0$, $\rho\le 0$, $L=(L_t)_{t\ge0}$ is a non-decreasing L\'evy process with compensated version $\wt L$ and $W=(W_t)_{t\ge0}$ is a standard Brownian motion independent of $L$. The volatility process $V$ is taken to be the stationary solution of \eqref{eq1.1}, in other words, a stationary L\'evy-driven Ornstein-Uhlenbeck (OU) process.
In this model, price jumps are modelled by (scaled) upwards jumps in the volatility.

It was noticed early on that the exponential autocovariance function of the OU process may be too restrictive.
Two suggestions have been made to allow for more flexibility in the autocovariance function: Barndorff-Nielsen 
\cite{BNsuper} suggested to replace $V$ by a superposition of such processes (called supOU process), which yields more flexible monotone autocovariance functions. It is defined as
\begin{equation}\label{supOU}
V_t=\int_{(-\infty,t]} \int_{(0,\infty)} \ee^{-\la (t-s)}\, \Lambda(\dd s,\dd\la),\quad t\in\bbr,
\end{equation}
where $\Lambda$ is an independently scattered infinitely divisible random measure, also called L\'evy basis. Superpositions of CARMA processes can be defined analogously; cf. \cite{BN-S:2011,Chong13}. As shown in e.g. \cite[Prop.~2.6]{FK}, supOU models can also model long range dependence for specific superposition measures.

On the other hand, \cite{brockwell5,todorov:tauchen:2006} suggested higher order L\'evy-driven CARMA models, which also allow for non-monotone autocovariance functions. 
The drawback of both model classes is their linearity and its consequences towards the stylized features of financial data. 
For instance, linear models inherit their distributions from that of the L\'evy increments in a linear way.
As a consequence, only when the driving L\'evy process has heavy-tailed  (regularly varying) increments, they model high level volatility clusters; cf. \cite[Prop.~5]{Fasen}.
Moreover, in contrast to empirical findings (cf.~\cite{JKMc}), these models allow only for negative price jumps coupled to the jumps in the volatility.

A continuous-time GARCH (COGARCH) model has been introduced in \cite{KLM:2004} with volatility process $V$ and martingale part of the log asset price given by 
\beam\label{eq1.2}
\dd V_t &=& (\beta -\eta V_t)\, \dd t + V_{t-}\vp\, \dd[L,L]_t,\\
\dd G_t &=& \sqrt{V_{t-}} \,\dd L_t,\nonumber
\eeam
where $\beta, \eta, \vp>0$ and $L$ is an arbitrary mean zero L\'evy process. 
The volatility process $V$ is taken to be the stationary solution of \eqref{eq1.2}.
This model satisfies all stylized features of financial prices, exactly as the GARCH model for low frequency data. 
The drawback of an exponentially decreasing covariance function has been taken care of by higher order models; cf.~\cite{BCL:2006}, like generalizing from OU to CARMA.

All models mentioned above have price jumps exactly at the times when the volatility jumps, since their prices are driven by the same L\'evy process.
Moreover, with the exception of the supOU/supCARMA process, all jump sizes in volatility and price exhibit a fixed deterministic relationship; cf.~\cite{JKMc}. As this is not very realistic, multi-factor models are needed. In this paper we want to contruct such a multi-factor model, based on the COGARCH.

In contrast to the OU or CARMA models, the COGARCH model is defined as a stochastic integral with stochastic integrand.
But also in this framework there is a canonical way to construct a superposition.

Starting by the fact that the ratio of volatility jumps and squared price jumps is always equal to $\vp$ in the COGARCH model, we randomize this scale parameter $\vp$.
There are various ways how to do this in a meaningful way, and we present three different possibilities, all leading to multi-factor COGARCH models.
Our three models have different qualitative behaviour. For instance, the first supCOGARCH allows for jumps in the volatility, which do not necessarily lead to jumps in the price process.
On the other hand, for certain choices of the distribution of the random parameter $\vp$, the third supCOGARCH model allows for jumps in the price without having a jump in the volatility. More properties will be reported.

An interesting feature is that some of the presented new supCOGARCH volatility processes can be written in terms of a so-called {\sl ambit process}, which has been introduced in \cite{BN-S:2004} in the context of turbulence modelling.
In our context the ambit process has a stochastic integrand, which is not independent of the integrator.
This implies that we are no longer in the framework of \cite{Rajput:1989}. Moreover, since COGARCH models are heavy-tailed, having possibly not even a second finite moment, the theory presented in~\cite{Walsh} is also not applicable. 
Instead we need the concept presented in \cite{Chong13}, which allows to integrate stochastic processes with respect to a L\'evy basis in the generality needed for our supCOGARCH models.

Our paper is organized as follows. In Section \ref{s2}, we recall the COGARCH model and give a short summary of L\'evy bases. In Section \ref{s3}, we present three different superpositions of COGARCH volatility processes. For each of the three models we give necessary and sufficient conditions for strict stationarity and derive the second order structure of the stationary process. 
The superpositions allow for more flexible autocorrelation structures than the COGARCH model
(Propositions~\ref{propapproach1moments}, \ref{propapproach2moments} and \ref{secord-supcog3}).
However, the stationary distributions of the supCOGARCH processes preserve the Pareto-like tails of the COGARCH process (Propositions~\ref{tail-supCOG1}, \ref{prop-tails-sup2} and \ref{tail-supCOG3}).
Section \ref{s5} is devoted to the corresponding price processes and the second-order properties of their stationary increments. Again, main characteristics of the COGARCH are preserved like the uncorrelated increments but positively correlated squared increments (Theorems \ref{price-supCOG1}, \ref{prop-price2} and \ref{prop-price3}). Nevertheless, each of the supCOGARCH models has its specific characteristics as highlighted in Section \ref{s4}. 
Furthermore, for all three models there is no longer a deterministic relationship between the jump sizes in volatility and price. Although in this paper we concentrate on the probabilistic properties of our new models, statistical issues are shortly addressed here. 
Finally, Section \ref{s6} contains the proofs of our results.

\section{Notation and Preliminaries}\label{s2}

By the L\'evy-Khintchine formula (e.g. \cite[Thm. 8.1]{sato}) the {\sl characteristic exponent} of a real-valued L\'evy process
$X=(X_t)_{t\geq 0}$ is given by
\[
 \psi_X(u) := \log \bbe\left[\ee^{\ii u X_1 } \right]
 = \ii \gamma_X u- \frac{1}{2}  \sigma_X^2 u^2  + \int_{\RR} (\ee^{\ii u y} -1 -\ii u y  \mathds{1}_{\{|y|\leq 1\}}) \,\nu_X(\dd y),\quad u\in\bbr,
\]
where $(\gamma_X, \sigma_X^2, \nu_X)$ is the {\sl characteristic triplet} of $X$ with {\sl L\'evy measure} $\nu_X$ satisfying  $\nu_X(\{0\})=0$ and $\int_\bbr 1\wedge|y|^2\,\nu_X(\dd y)<\infty$.
If additionally $\int_{|y|\leq 1} |y|\,\nu_X(\dd y)<\infty$, we may also write the characteristic exponent in the form
\[ \psi_X(u)=\ii\gamma^0_X u - \frac{1}{2} \sigma_X^2 u^2  + \int_{\RR} (\ee^{\ii u y}  -1 )\, \nu_X(\dd y),\quad u\in\bbr,\]
and call $\gamma_X^0$ the {\sl drift} of $X$.
This is in particular the case for subordinators, i.e. L\'evy processes with increasing sample paths.
We also recall that the {\sl quadratic variation process} of the L\'evy process $X$ is given by
\[ [X,X]_t := \si_X^2 t + [X,X]^\dd_t := \si_X^2 t + \sum_{0<s\leq t} (\Delta X_s)^2,\quad t\geq0,\]
where $[X,X]^\dd$ is called the {\sl pure-jump part} of $[X,X]$.

Every L\'evy process $(X_t)_{t\geq 0}$ can be extended to a {\sl two-sided L\'evy process}
$(X_t)_{t\in\RR}$ by setting $X_{t}=-X'_{-t-}$, $t<0$, for some i.i.d.
copy $X'$ of $X$. We say that $(X_t)_{t\in\RR}$ has characteristic triplet $(\gamma_X, \sigma_X^2, \nu_X)$ if $(X_t)_{t\geq 0}$
has characteristic triplet $(\gamma_X, \sigma_X^2, \nu_X)$.


Throughout we use the notation $\RR_+=(0,\infty)$, $\bbr_-=(-\infty,0)$ and $\NN_0= \NN \cup \{0\}$.

\subsection{The COGARCH model}\label{s21}

Let $(L_t)_{t\geq 0}$ be a L\'evy process with characteristic triplet $(\ga_L,\si_L^2,\nu_L)$ and define
\begin{equation} \label{eq-def-S}
 S_t:=[L,L]_t^{\dd}=\sum_{0<s\leq t} (\Delta L_s)^2, \quad t\geq 0.
\end{equation}
Then $(S_t)_{t\ge0}$ is a subordinator without drift and its L\'evy measure $\nu_S$ is the image measure of $\nu_L$ under the transformation $y\mapsto y^2$.
For $\eta>0$ and $\vp\geq0$ define another L\'evy process by
\begin{equation}\label{eq-defX}
X^\vp_t = \eta t - \sum_{0<s\le t} \log (1+\vp \Delta S_s),\quad t\ge 0,
\end{equation}
which is completely determined by $S$ (and hence by $L$). Then $X^\vp$ has characteristic triplet $(\eta, 0,\nu_{X^\vp})$, where
 $\nu_{X^\vp}$ is the image measure of $\nu_S$ under the mapping $y\mapsto-\log(1+\vp y)$, and is therefore a spectrally negative L\'evy process, i.e. it only has negative jumps. For $t\geq0$ we have
\beq\label{eq-defpsi} \E [\ee^{-u X^\vp_t}] = \ee^{t \Psi(u,\vp)} \quad\text{with}\quad \Psi(u,\vp)= -\eta u + \int_{\bbr_+} ((1+\vp y)^{u} -1) \,\nu_S(\dd y),\eeq
where, whenever  $\vp>0$, we have $\E [\ee^{-u X^\vp_t}]<\infty$ for $u>0$ for some $t>0$ or,
equivalently, for all $t>0$, if and only if
$\E[S_1^{u}]<\infty$ \cite[Lemma 4.1]{KLM:2004}. In particular, if $\E[S_1]<\infty$ or $\E[S_1^{2}]<\infty$, respectively, we have from \cite[Ex. 25.12]{sato}
\begin{equation} \label{eq-psi-explicit}
 \Psi(1,\vp)=\vp \E[S_1]-\eta \quad \mbox{and} \quad \Psi(2,\vp)= 2\vp \E[S_1]+\vp^2\var[S_1] -2\eta.
\end{equation}

Recall from \cite{KLM:2004} that the {\sl COGARCH (volatility) process} driven by the L\'evy process $L$ (or the subordinator $S$) with parameter $\vp$ is given by
\begin{equation}\label{cog}
V^\vp_t = \ee^{-X^\vp_t}\left(V^\vp_0 + \beta\int_{(0,t]} \ee^{X_{s}^\vp} \,\dd s \right),\quad t\ge0,
\end{equation}
where $\beta>0$ is a constant and $V_0^\vp$ is a nonnegative random variable, independent of $(S_t)_{t\ge0}$. 

Moreover, the COGARCH volatility process $V^\vp$ is a special case of a generalized Ornstein-Uhlenbeck process (cf. \cite{behme_diss,lindner:maller:2005}) and is the solution of the SDE
\begin{equation}\label{cog-sde}
\dd V^\vp_t = (\beta - \eta V^\vp_{t}) \,\dd t + V^\vp_{t-} \vp \,\dd S_t
= V^\vp_{t-}(\vp \dd S_t - \eta \,\dd t)+\beta\, \dd t,\quad t\ge0.
\end{equation}
It admits the integral representation
\begin{equation}\label{phisol}
V^\vp_t = V^\vp_0 + \beta t - \eta \int_{(0,t]} V^\vp_s \,\dd s +  \sum_{0<s\le t} V_{s-}^\vp \vp \Delta S_s,\quad t\ge0.
\end{equation}
The corresponding {\sl price process} or {\sl integrated COGARCH process} is then defined as
\begin{equation} \label{eq-pricecog} G_t=\int_0^t \sqrt{V_{s-}^\vp}\,\dd L_s, \quad t\geq 0.\end{equation}

\subsection{Stationary COGARCH processes}\label{s22}

By \cite[Thm. 3.1]{KLM:2004}, the process defined in \eqref{cog} or equivalently in \eqref{phisol} has a strictly stationary distribution if and only if
\beq\label{cogstatcond} \int_{\bbr_+} \log(1+\vp y)\,\nu_S(\dd y) =\int_{\bbr} \log(1+\vp y^2)\,\nu_L(\dd y) < \eta. \eeq
In this case, the stationary distribution of the COGARCH process is given by the distribution of $V_\infty^\vp:=\beta \int_{\bbr_+} \ee^{-X^\vp_{s}} \,\dd s$. Note that for $\vp=0$, the stationary COGARCH reduces to $V^0_t = \beta/\eta$ for all $t\geq0$.

In the sequel we denote by the set $\Phi_L$ all $\vp\geq0$ where \eqref{cogstatcond} is satisfied.
By monotone convergence, the left-hand side of \eqref{cogstatcond} is continuous in $\vp$ and converges to $+\infty$ as $\vp\to\infty$,
which means that $\vp_{\mathrm{max}}:=\sup \Phi_L$ is finite and hence $\Phi_L=[0,\vp_{\mathrm{max}})$.

Let us recall the moment structure of $V^\vp$ in the stationary case.
It follows by direct computation from \cite[Thm. 3.1]{behme} that, if $\kappa>0$ is a constant, then
\begin{equation} \label{momentcondcogarch}
 \bbe [S_1^{\max \{ \kappa , 1 \} }]<\infty \quad \mbox{and} \quad
\log\bbe  \left[ \ee^{-\kappa X_1^\vp}\right] = \Psi(\kappa, \vp)<0
\end{equation}
imply $\bbe[(V^\vp_0)^\kappa]<\infty$.
If \eqref{momentcondcogarch} holds for $\kappa=1$ or $\kappa=2$, respectively,
for every $t\geq 0$, $h\geq 0$ the first two moments of the stationary process $V^\vp$ are given by (\cite[Cor. 4.1]{KLM:2004})
\begin{align}
\bbe[V^\vp_t]&=-\frac{\beta}{\Psi(1,\vp)}=\frac{\beta}{\eta - \vp \bbe[S_1]},\label{mean-cog}\\
\bbe[(V^\vp_t)^2]&= \beta^2 \frac{2}{\Psi(1,\vp)\Psi(2,\vp)}\quad \mbox{and} \label{squaremean-cog} \\
\cov[V^\vp_t,V^\vp_{t+h}]
&=\ee^{h\,\Psi(1,\vp)} \var[V_0^\vp] =\ee^{h\,\Psi(1,\vp)}\beta^2\left(\frac{2}{\Psi(1,\vp)\Psi(2,\vp)}-\frac{1}{\Psi(1, \vp)^2}\right)\label{acf-cog}\\
&=\ee^{h\,(\vp \E[S_1]-\eta)}\frac{\beta^2\vp^2\var[S_1]}{(\vp \E[S_1]-\eta)^2(2\eta-2\vp \E[S_1]-\vp^2 \var[S_1])}.\nonumber
\end{align}

From \eqref{momentcondcogarch} we have the clear picture that, although a stationary $V^\vp$ exists for all $\vp\in\Phi_L=[0,\vp_{\max})$,
moments only exist on some subinterval, which shrinks with the increasing order of the moment. Moreover, it is known that no COGARCH process has
moments of all orders \cite[Prop. 4.3]{KLM:2004}. For later reference we set
\begin{equation}\label{eq-def-Phimoment}
 \Phi_L^{(\kappa)}:=[0,\vp_{\max}^{(\kappa)})\quad \mbox{with} \quad \vp_{\max}^{(\kappa)}=\sup\{\vp: \E[(V_0^\vp)^\kappa]<\infty \}.
\end{equation}
We have $0<\vp_{\max}^{(\kappa_2)}\le \vp_{\max}^{(\kappa_1)}< \vp_{\max}<\infty$ whenever $0<\kappa_1\leq\kappa_2<\infty$, i.e. $\Phi_L^{(\kappa_2)}\subset \Phi_L^{(\kappa_1)}\subset \Phi_L$.

In \cite{klm:2006} the tail behaviour of the COGARCH process is studied. In particular, it is shown that under rather weak assumptions
the distribution of $V_0^\vp$ has Pareto-like tails \cite[Thm. 6]{klm:2006}.

Regarding the price process $G^\vp$ in the stationary case, it is known from \cite[Prop. 5.1]{KLM:2004} that $G^\vp$ has stationary increments that are uncorrelated on disjoint intervals while the squared increments are, under some technical assumptions, positively correlated, an effect which is typical for financial time series.

For later reference we extend the stationary COGARCH volatility process \eqref{cog} to a two-sided
process in the following way.
For a two-sided L\'evy process $(L_t)_{t\in\bbr}$ we obtain a two-sided subordinator $(S_t)_{t\in \RR}$ by setting
\begin{equation}\label{eq-def-S2sided}
 S_t:=\sum_{0<s\leq t} (\Delta L_s)^2,\quad t\geq 0\quad \mbox{and}\quad S_t:= - \sum_{t< s \leq 0} (\Delta L_s)^2, \quad t\leq 0.
\end{equation}
Now we automatically obtain for every $\vp$ another two-sided L\'evy process $(X^\vp_t)_{t\in\RR}$
given by
\begin{equation}\label{eq-defX2sided}
 X^\vp_t = \eta t - \sum_{0<s\le t} \log (1+\vp \Delta S_s),
\quad t\ge 0,\quad X^\vp_t=\eta t + \sum_{t<s\le 0} \log (1+\vp \Delta S_s),\quad t<0.
\end{equation}
The two-sided COGARCH process $(V^\vp_t)_{t\in\RR}$ is then given by
\beq\label{cog-2sided} V^\vp_t:=\beta\int_{(-\infty,t]} \ee^{-(X^\vp_t-X^\vp_{s})}\,\dd s, \quad t\in\bbr, \eeq
and it is well-defined for every $\vp\in\Phi_L$.
Obviously, the restriction of this process to $t\geq0$ equals the process given in \eqref{cog}
with starting random variable $V^\vp_0:=\beta\int_{(-\infty,0]} \ee^{X_s^\vp}\,\dd s$. Hence the two-sided COGARCH is always stationary  with the same finite-dimensional distributions as the one-sided stationary COGARCH.

\subsection{L\'evy bases}\label{s23}

Let $(\Om,\calf,\bbf=(\calf_t)_{t\in\bbr},\PP)$ be a filtered probability space satisfying
the usual assumptions of completeness and right-continuity. Denote the space of all $\PP$-a.s. finite random variables by $L^0$, the optional (resp. predictable) $\si$-field by $\calo$ (resp. $\calp$) and set $\tilde\calp:=\calp\otimes\calb(\bbr^d)$, where $\calb(\bbr^d)$ is the Borel-$\si$-field on $\bbr^d$. Now let $(E_k)_{k\in\bbn}$ be a sequence of measurable subsets increasing to $\bbr^d$ and define
$\tilde\calp_\bb$ as the collection of all $\tilde\calp$-measurable subsets of $\Om\times(-k,k]\times E_k$ for $k\in\bbn$. Similarly, set $\calb_\bb:=\bigcup_{k=1}^\infty \calb((-k,k]\times E_k)$.

In this set-up, we use the term L\'evy basis as follows: 

\bde\label{def:Levybasis}
A {\sl L\'evy basis} on $\bbr\times\bbr^d$ is a mapping $\La\colon \tilde\calp_\bb \to L^0$ satisfying:
\ben
    \item $\La(\emptyset)=0$ a.s.
    \item If $(A_n)_{n\in\bbn}$ are pairwise disjoint sets in $\tilde\calp_\bb$ whose union again lies in $\tilde\calp_\bb$, then
    \[ \La\Big(\bigcup_{n=1}^\infty A_n\Big) = \sum_{n=1}^\infty \La(A_n)\quad \mbox{a.s.} \]
    \item If $(B_n)_{n\in\bbn}$ are pairwise disjoint sets in $\calb_\bb$, then $(\La(\Om\times B_n))_{n\in\bbn}$ is a sequence of independent random variables with each of them having an infinitely divisible distribution.
    \item If $A\in\tilde\calp_\bb$ is a subset of $\Om\times (-\infty,t]\times\bbr^d$ for some $t\in\bbr$, then $\La(A)$ is $\calf_t$-measurable.
    \item If $A\in\tilde\calp_\bb$, $t\in\bbr$ and $F\in\calf_t$, then $\La\big(A\cap(F\times(t,\infty)\times \bbr^d)\big) = \bone_F \La\big(A\cap(\Om\times(t,\infty)\times \bbr^d)\big)$.
    \item For all $t\in\bbr$ and measurable $U\subset E_k$ for some $k\in \NN$, we have $\La(\Om\times\{t\}\times U)=0$ a.s. 
\een
In the following, we often write $\La(B)=\La(\Om\times B)$ for a set $B\in\calb_\bb$.
\ede

A natural choice for $\bbf$ is certainly the {\sl augmented natural filtration} $\bbg=(\calg_t)_{t\in\bbr}$ of the
L\'evy basis $\La$, which means that for $t\in\bbr$, $\calg_t$ is the completion of the $\si$-field generated by the collection of all $\La(B)$ with $B\in\calb_\bb, B\subseteq (-\infty,t]\times\bbr^d$.

The first three points of Definition \ref{def:Levybasis} are similar to the notion of infinitely divisible independently scattered random measures in \cite{Rajput:1989}. 
Further we have added condition (f) because this ensures that $\La$ induces a jump measure $\mu^\La$ by
\beq\label{jm} \mu^\La(\om,\dd t, \dd x, \dd y) := \sum_{s\in\bbr} \sum_{\xi\in\bbr^d} \bone_{\{\La(\{s\}\times\{\xi\})(\om)\neq0\}}
\delta_{(s,\xi,\La(\{s\}\times\{\xi\})(\om))}(\dd t,\dd x,\dd y),\quad\om\in\Om, \eeq
where $\delta$ stands for the Dirac measure. We will follow the usual convention of suppressing $\om$ in the sequel. 
Thanks to (d) and (e), $\mu^\La$ is an optional $\tilde\calp$-$\si$-finite random measure in the sense of \cite[Theorem II.1.8]{JS2}. Therefore, the predictable compensator $\Pi$ of $\mu^\La$ is well-defined.  

In this paper, we will only consider L\'evy bases $\La$, which are of the form
\beq\label{subordbasis} \La(\dd s,\dd x)=\int_\bbr y\,\mu^\La(\dd s,\dd x,\dd y). \eeq
In addition, the predictable compensator of $\mu^\La$ in the augmented natural filtration $\bbg$ will always be given by $\Pi(\dd s,\dd x,\dd y)=\dd s\,\pi(\dd x)\,\nu(\dd s)$, where $\pi$ is some probability measure on $\bbr^d$ and $\nu$ the L\'evy measure of a subordinator. In this particular case, if we write
\[ W(s,x,y)\ast\mu^\La_t := W\ast\mu^\La_t :=
\begin{cases}
\displaystyle\int_{(0,t]\times\bbr^d\times\bbr} W(s,x,y)\,\mu^\La(\dd s,\dd x,\dd y), & \text{if } t\geq0,\\
 \displaystyle\int_{(t,0]\times\bbr^d\times\bbr} W(s,x,y)\,\mu^\La(\dd s,\dd x,\dd y), & \text{if } t<0,
\end{cases}  \]
for some $\calo\otimes\calb(\bbr^d)\otimes\calb(\bbr)$-measurable function $W$ which is integrable w.r.t. $\mu^\La$ ($\om$-wise as a Lebesgue integral), then we have
\beq\label{expect} \bbe[W\ast\mu^\La_t] = \bbe[W\ast\Pi_t] = \int_{(0,t]\times\bbr^d\times\bbr} \bbe[W(s,x,y)]\,\Pi(\dd s,\dd x,\dd y),\quad t\geq0,  \eeq
for all integrable functions $W$ (and similarly for $t<0$), see \cite[Theorem II.1.8]{JS2}. Moreover, when taking stochastic integrals with respect to $\La$, these can be expressed in terms of $\mu^\La$:
\[ \int_{(0,t]\times \bbr^d} H(s,x)\,\La(\dd s,\dd x) = \int_{(0,t]\times\bbr^d\times\bbr} H(s,x)y\,\mu^\La(\dd s,\dd x,\dd y),\quad t\geq0,\]
for all $H$ which are integrable w.r.t. $\La$ on $(0,t]$ (similarly for $t<0$); see \cite{Chong13} for integrability conditions and further details on L\'evy bases.

For later reference, we also introduce the pure-jump part of the quadratic variation measure of $\La$ defined as
\beq\label{pjqvm} [\La, \La]^\dd(A):=\int_{\bbr\times\bbr^d\times\bbr} \bone_A(t,x)y^2\,\mu^\La(\dd t,\dd x,\dd y),\quad A\in\tilde\calp_\bb. \eeq

\section{Superposition of COGARCH (supCOGARCH) processes}\label{s3}

In the following three subsections we propose different approaches to construct a superposition of COGARCH processes. As seen in Eq.~\eqref{cog-sde}, the parameters $\beta$ and $\eta$ only influence the continuous part of the COGARCH process, whereas $\vp$ scales its jump sizes. Since our goal is to find a model which shares the  basic features of the COGARCH model but has a more flexible jump structure, we let $\beta$ and $\eta$ be fixed in the following three approaches and only allow the parameter $\vp$ to vary.

\subsection{The supCOGARCH 1 volatility process}\label{s31}

The obvious idea of defining a supCOGARCH process as a weighted integral of independent COGARCH processes with different parameters $\vp$ yields to
consider
\beq\label{v1-idea} \bar{V}^{(1)}_t:=\int_{[0,\infty)} V_t^\vp \,\pi(\dd \vp),\quad t\geq0, \eeq
for some probability measure $\pi$ on $[0,\infty)$, where each COGARCH process $V^\vp$
is driven by $S^{\vp}=[L^\vp,L^\vp]^\dd$ and $(L^\vp)_{\vp\in[0,\infty)}$ are i.i.d. copies of a canonical L\'evy process $L$ which, together with $S=[L,L]^\dd$, we only use for notational convenience. As a consequence, $(V^\vp)_{\vp\in[0,\infty)}$ is a family of independent COGARCH processes such that the integral in \eqref{v1-idea} is only well-defined if $\pi$ has countable support. This leads to the {\sl supCOGARCH 1 volatility process}
\begin{equation}\label{v1}
\bar{V}^{(1)}_t=\int_{[0,\infty)} V_t^\vp \,\pi(\dd\vp) = \sum_{i=1}^\infty p_i V_t^{\vp_i}, \quad t\geq 0,
\end{equation}
where $\pi=\sum_{i=1}^\infty p_i \delta_{\vp_i}$ for nonnegative weights $(p_i)_{i\in\bbn}$ with $\sum_{i=1}^\infty p_i = 1$.

To avoid degenerate cases we will assume throughout that
\beq\label{conv-bed-1} \bar{V}^{(1)}_0 = \sum_{i=1}^\infty p_i V^{\vp_i}_0 < \infty \quad\text{a.s.} \eeq
Note that this does not automatically imply finiteness of the supCOGARCH process at all times unless we are in the stationary case (see below).

\brem The supCOGARCH 1 process can also be written in terms of a L\'evy basis. First, define a L\'evy basis on $\bbr_+\times{[0,\infty)}$ by
\[ \La^L((0,t]\times \{\vp_i\}):=\sqrt{p_i}L_t^{\vp_i},\quad t\geq0,i\in\bbn, \]
and $\La^L(\bbr\times ([0,\infty)\setminus \bigcup_{i=1}^\infty \{\vp_i\})):=0$. Now with $\La^S=[\La^L,\La^L]^\dd$ being the pure-jump quadratic variation measure of $\La^L$ (in particular, $\La^S((0,t]\times\{\vp_i\})=p_iS^{\vp_i}_t$) and inserting \eqref{phisol} in \eqref{v1}, we see that
\begin{align} \label{defapproach1}
 \bar{V}^{(1)}_t
&= \sum_{i=1}^\infty p_i V_0^{\vp_i} + \beta t - \eta \sum_{i=1}^\infty p_i \int_{(0,t]} V_{s}^{\vp_i} \,\dd s
+ \sum_{i=1}^\infty  \int_{(0,t]}  p_i \vp_i V_{s-}^{\vp_i} \,\dd S_s^{\vp_i} \\
&= \bar{V}^{(1)}_0  + \beta t - \eta \int_{(0,t]}\bar{V}^{(1)}_s \,\dd s +  \int_{(0,t]} \int_{[0,\infty)} \vp V^\vp_{s-}\,\La^S(\dd s,\dd\vp),
 \quad t\geq 0.\nonumber
\end{align}
Note that for each $i\in\bbn$, $V^{\vp_i}$ is driven by $S^{\vp_i}$.
\erem

It follows directly from \eqref{defapproach1} that the jumps of the supCOGARCH 1 process are given by
\begin{equation} \label{approach1jump}
 \Delta \bar{V}^{(1)}_t=\sum_{i=1}^\infty p_i \Delta V_t^{\vp_i} = \sum_{i=1}^\infty p_i V_{t-}^{\vp_i} \vp_i \Delta S^{\vp_i}_t = \int_{[0,\infty)} \vp V^\vp_{t-}\,\La^S(\{t\}\times \dd \vp),\quad t\geq0.
\end{equation}
Since the independent subordinators a.s. jump at different times, a.s. only one summand in \eqref{approach1jump} is nonzero at each jump time.

The following example for a probability measure $\pi$ with two-point support will be carried through the three different
supCOGARCH processes in this section to clearify their definitions.
\bexam \label{example-vergleich}
Let $\pi=p_1\delta_{\vp_1}+ p_2\delta_{\vp_2}$ with $p_1+p_2=1$ and $\vp_1,\vp_2\in\bbr_+$. Then the supCOGARCH~1 process is the weighted sum of two independent COGARCH processes. More precisely, we have $\bar{V}^{(1)}_t= p_1 V_t^{\vp_1} + p_2 V^{\vp_2}_t$ for $t\geq0$,
where $V^{\vp_1}$ and $V^{\vp_2}$ are driven by \emph{independent} copies of the canonical L\'evy process $L$. From Figure \ref{figsup1}, we clearly see that the supCOGARCH~1 process inherits both the jumps of $V^{\vp_1}$ and $V^{\vp_2}$, scaled with $p_1$ or $p_2$, respectively.
\eexam
\setlength{\unitlength}{0.8cm}
\begin{figure}[!ht]
\begin{picture}(0,5)
\put(0,0){\includegraphics{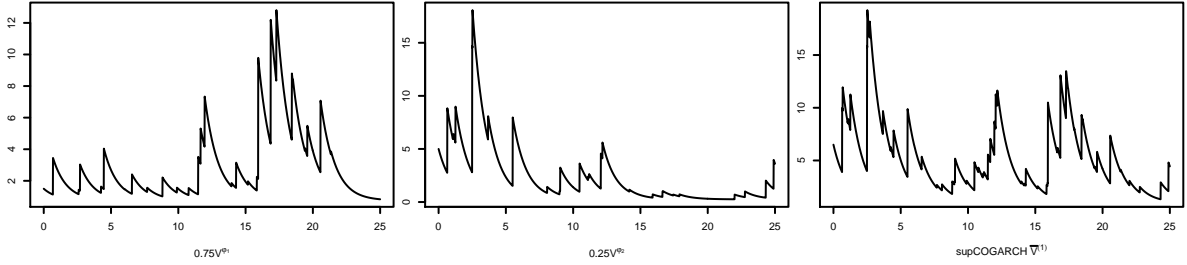}}
\end{picture}
\caption{\scriptsize Sample paths of two independent COGARCH processes with different values for $\vp$, scaled with the corresponding $p_i$, and the resulting supCOGARCH 1 process. The driving L\'evy processes are independent compound Poisson processes with rate $1$ and standard normal jumps. The parameters are: $\beta=1$, $\eta=1$, $\vp_1=0.5$, $\vp_2=0.95$ and $\pi=0.75\delta_{\vp_1}+0.25\delta{\vp_2}$, starting value was the respective mean.}\label{figsup1}
\end{figure}

Stationarity and second-order properties of the supCOGARCH 1 process are given in the following three results. Proofs are postponed to Section~\ref{s61}.

\begin{theorem}\label{thm-app1-stat}
Let $\pi=\sum_{i=1}^\infty p_i \delta_{\vp_i}$ be a probability measure on $[0,\infty)$, $\{L^{\vp_i}\colon i \in \NN \}$ a family of i.i.d. L\'evy processes, $\{S^{\vp_i}\colon i \in \NN \}$ the corresponding family of subordinators
and $\{V^{\vp_i}\colon i\in\NN \}$ the corresponding
family of COGARCH processes.
Assuming that \eqref{conv-bed-1} holds, 
a finite random variable $\bar{V_0}^{(1)}$ can be chosen
such that $\bar{V}^{(1)}$ is strictly stationary if and only if
\begin{equation} \label{approach1statiff}
\pi(\Phi_L)=1.
\end{equation}
In the case that a stationary distribution exists, it is uniquely determined by the law of
\begin{equation} \label{approach1statsol}
 \bar{V}_\infty^{(1)} := \int_{\Phi_L} V^\vp_\infty \,\pi(\dd \vp)
=  \beta \int_{\Phi_L} \int_{\bbr_+} \ee^{-X_{t}^\vp}  \,\dd t\, \pi(\dd \vp)=
\beta \sum_{i=1}^\infty p_i \int_{\bbr_+} \ee^{-X_{t}^{\vp_i}} \,\dd t.
\end{equation}
\end{theorem}

\begin{proposition} \label{propapproach1moments}
 Assume we are in the setting of Theorem \ref{thm-app1-stat} and let $\bar{V}^{(1)}$ be a strictly stationary solution of \eqref{defapproach1}. Recall the notation $\Phi_L^{(\kappa)}$ from Eq.~\eqref{eq-def-Phimoment}.
\begin{enumerate}
 \item
Suppose that $\pi(\Phi_L^{(1)})=1$. Then for every $t\geq0$,
 \begin{equation}\label{mean-supcog1}
   \E[\bar{V}^{(1)}_t] = \int_{\Phi_L} \E[V_0^\vp]\,\pi(\dd \vp)= \beta \sum_{i=1}^\infty \frac{p_i}{\eta-\vp_i \bbe[S_1]}.
  \end{equation}
\item
Suppose that $\pi(\Phi_L^{(2)})=1$. Then for every $t\geq 0$, $h\geq 0$ we have
 \begin{align}
   \var[\bar{V}^{(1)}_t] &= \sum_{i=1}^\infty p_i^2 \var[V_0^{\vp_i}] \quad \mbox{and} \label{var-supcog1}\\
  \cov[\bar{V}^{(1)}_t, \bar{V}^{(1)}_{t+h}] &= \sum_{i=1}^\infty p_i^2 \cov[V_0^{\vp_i},V_{h}^{\vp_i} ], \label{acf-supcog1}
 \end{align}
with $\var[V_0^{\vp_i}]$ and $\cov[V_0^{\vp_i},V_{h}^{\vp_i}]$ as given in \eqref{squaremean-cog} and \eqref{acf-cog}.
\end{enumerate}
Note that the quantities in \eqref{mean-supcog1}, \eqref{var-supcog1} and \eqref{acf-supcog1} may be infinite.
\end{proposition}

\begin{proposition} \label{tail-supCOG1} Assume we are in the setting of Theorem \ref{thm-app1-stat} and let $\bar{V}^{(1)}$ be a strictly stationary solution of \eqref{defapproach1}.
Set $\bar \vp:=\inf\{\vp>0\colon \pi((\vp,\infty))=0\}\leq \vp_\mathrm{max}<\infty$ and assume that there
 exists $\bar\kappa>0$ with
\beq\label{tailcond} \bbe[S^{\bar\kappa}_1\log^+(S_1)] <\infty \quad\text{and}\quad\Psi(\bar\kappa,\bar\vp)=0. \eeq
Then we have for $\kappa>0$
\[ \lim_{x\to\infty} x^\kappa \PP[\bar V^{(1)}_0 > x]=\begin{cases} 0 &\text{if } \kappa < \bar\kappa,\\ \infty &\text{if } \kappa>\bar\kappa, \end{cases} \]
while for $\kappa = \bar\kappa$ there exists a constant $C>0$ such that
\[ \lim_{x\to\infty} x^{\bar{\kappa}} \PP[\bar V^{(1)}_0 > x]=\begin{cases} C &\text{if } \pi(\{\bar\vp\})=\bar p>0, \\ 0 &\text{if } \pi(\{\bar\vp\})=0. \end{cases} \]
\end{proposition}

\brem Recall from \cite[Thm.~5]{klm:2006} that the stationary distribution of the COGARCH $V^\vp$ is self-decomposable, i.e. for all $b\in (0,1)$
there exists a random variable $Y_b$ such that $V_\infty^\vp \eqd b (V_\infty^\vp)' + Y_b$ where $(V_\infty^\vp)'$
is an independent copy of $V_\infty^\vp$. Due to the fact that self-decomposability is preserved
under scaling, convolution and taking limits, see e.g. \cite[Prop. V.2.2]{steutel}, it follows directly from \eqref{approach1statsol}
that the stationary distribution of the supCOGARCH 1 process $\bar{V}^{(1)}$ is self-decomposable, too.
\erem

\brem Unless we are in the degenerate case $\pi=\delta_\vp$ and the supCOGARCH is in fact just the COGARCH with parameter $\vp$, the supCOGARCH process $\bar{V}^{(1)}$ is no longer a Markov process with
respect to its augmented natural filtration,  i.e. the smallest filtration such that $\bar{V}^{(1)}$ is adapted
and which satisfies the usual hypotheses of right-continuity and completeness.
But it follows directly from \eqref{defapproach1} that,
letting $\FF^{(1)}=(\cF^{(1)}_t)_{t\geq 0}$ be the augmented natural filtration of $((V_t^{\vp_i})_{i\in\NN})_{t\geq 0}$, we have for every measurable function $f\colon\RR_+\to\RR$ and every $t\geq 0$
\[\E\Big[ f\big(\bar{V}^{(1)}_t\big)\big| \cF^{(1)}_t\Big] = \E\Big[f\big(\bar{V}^{(1)}_t\big)\big| (V_t^{\vp_i})_{i\in\NN} \Big].\]
\end{remark}

\brem
In the representation $\bar{V}^{(1)} = \sum_{i=1}^\infty p_i V^{\vp_i}$
a priori the $\vp_i$ do not have to be pairwise different and still the results of this section remain valid (apart from some obvious notational changes).
\erem

\subsection{The supCOGARCH 2 volatility process} \label{s32}

In order to deal with uncountable superpositions, one possibility is to drop the assumption of independence,
which led to the supCOGARCH 1.
Hence we fix a L\'evy process $L$, define the subordinator $(S_t)_{t\geq 0}$ by \eqref{eq-def-S} and define the superposition as a weighted integral
of COGARCH processes $V^\vp$ as given in \eqref{phisol}
with different parameters $\vp$, but all driven by the single L\'evy process $L$, i.e. we set
\[\bar{V}^{(2)}_t:=\int_{\Phi_L} V^\vp_t \,\pi(\dd \vp),\quad t\geq 0,\]
for some probability measure $\pi$ on the parameter space $\Phi_L$.
To ensure that $\vp\mapsto V^\vp_t$ is measurable at all times and in particular at time $t=0$, we will use two-sided COGARCH processes as in \eqref{cog-2sided} and define the {\sl supCOGARCH 2 volatility process} 
\begin{equation}
 \bar{V}^{(2)}_t
:= \int_{\Phi_L} V_t^{\vp} \,\pi(\dd \vp)
= \beta \int_{\Phi_L} \int_{(-\infty,t]} \ee^{-(X_t^\vp-X_{s}^{\vp})} \,\dd s \, \pi(\dd\vp),\quad t\in\RR, \label{defapproach2}
\end{equation}
for $(X_t^\vp)_{t\in\RR}$ as given in \eqref{eq-defX2sided}. As a consequence, we have for $t\geq 0$
\begin{align}
\bar{V}^{(2)}_t
&= \int_{\Phi_L} V^\vp_0 \,\pi(\dd\vp) +\beta t- \eta \int_{\Phi_L} \int_{(0,t]} V^\vp_{s} \,\dd s \,\pi(\dd \vp) +  \int_{\Phi_L} \int_{(0,t]} \vp V^\vp_{s-} \,\dd S_s\, \pi(\dd\vp) \nonumber \\
&= \bar{V}^{(2)}_0 + \beta t - \eta \int_{(0,t]} \bar{V}^{(2)}_{s} \,\dd s +  \int_{(0,t]} \int_{\Phi_L} \vp V^\vp_{s-} \,\pi(\dd\vp) \,\dd S_s. \label{app2explicit}
\end{align}
In order to ensure that \eqref{defapproach2} is finite, we always assume
\beq\label{conv-bed-2} \int_{\Phi_L} V_0^\vp \,\pi(\dd\vp)<\infty. \eeq
If $\pi=\sum_{i=1}^\infty p_i\delta_{\vp_i}$, we obviously have
$\bar{V}^{(2)}= \sum_{i=1}^\infty p_iV^{\vp_i}$ with dependent summands.

Observe that in this setting all single COGARCH processes jump at the same times and thus we have
\begin{equation}\label{approach2jump}
 \Delta \bar{V}^{(2)}_t=\int_{\Phi_L} \vp V^\vp_{t-}\,\pi(\dd \vp) \Delta S_t,\quad t\geq0.
\end{equation}

\bexam[Ex. \ref{example-vergleich} continued]\rm\label{example-vergleich-2}
Let $\pi=p_1\delta_{\vp_1}+ p_2\delta_{\vp_2}$ with $p_1+p_2=1$ and $\vp_1,\vp_2\in\Phi_L$. Then the supCOGARCH 2 process
is the weighted sum of two COGARCH processes with parameters $\vp_1$ and $\vp_2$, i.e. $\bar{V}^{(2)}_t= p_1 V_t^{\vp_1} + p_2 V^{\vp_2}_t$.
 In contrast to the supCOGARCH 1 process in Example \ref{example-vergleich},
$V^{\vp_1}$ and $V^{\vp_2}$ are driven by the same subordinator, say $S$, of the form \eqref{eq-def-S}.
In Figure \ref{figsup2} we illustrate the typical relationship between the original COGARCH processes and the resulting supCOGARCH 2
process. 
We observe that $V^{\vp_1}$, $V^{\vp_2}$ and $\bar V^{(2)}$ all jump at the same times, with the jump sizes of the supCOGARCH being the weighted average jump sizes of the two COGARCH processes.  \eexam
\setlength{\unitlength}{0.8cm}
\begin{figure}[!ht]
\begin{picture}(0,5)
\put(0,0){\includegraphics{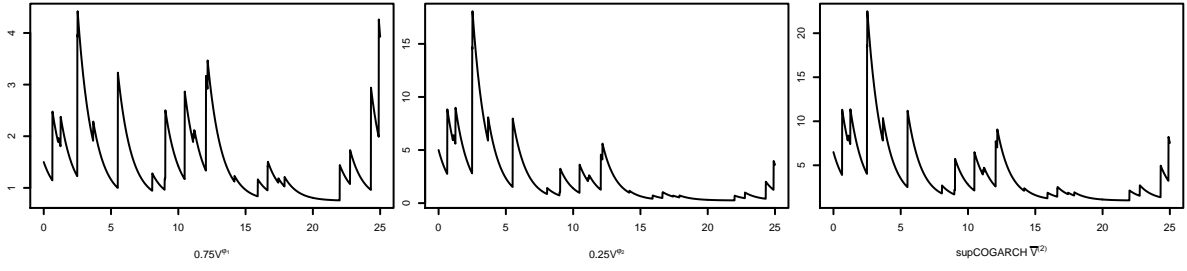}}
\end{picture}
\caption{\scriptsize Sample paths of two COGARCH processes $V^{\vp_1}$ and $V^{\vp_2}$ with different parameters, driven by the same L\'evy process $L$, scaled with the corresponding $p_i$, and the resulting supCOGARCH $\bar V^{(2)}$. The driving L\'evy process $L$ is a compound Poisson process with rate $1$ and standard normal jumps. The parameters are the same as in Figure \ref{figsup1}.}
\label{figsup2}
\end{figure}

In the following we present stationarity and second-order properties of the supCOGARCH process $\bar{V}^{(2)}$. 
Proofs are given in Section~\ref{s62}.

\begin{theorem} \label{thm-app2-stat}
Assume that \eqref{conv-bed-2} holds. Then $(\bar{V}^{(2)}_t)_{t\in\RR}$ as defined in \eqref{defapproach2} is strictly stationary.
\end{theorem}

Before we can calculate the moments of the stationary supCOGARCH process $\bar{V}^{(2)}$ in Proposition \ref{propapproach2moments} we need to establish
covariances between single COGARCH processes with different
parameters in the following proposition.

\begin{proposition} \label{prop-corrcogarches}
  Let $(S_t)_{t\in \RR}$ be a subordinator without drift, let $\vp, \tilde{\vp}\in \Phi_L$ be fixed and
define the stationary two-sided COGARCH processes
$(V^\vp_t)_{t\in\RR}$, $(V^{\tilde{\vp}}_t)_{t\in\RR}$ according to \eqref{cog-2sided}. If
$$\E [S_1^2 ]<\infty,\quad \Psi(2, \vp)<0 \quad \mbox{and} \quad
\Psi(2, \tvp)<0, $$
then $\bbe[V^\vp_t V^\tvp_{t+h}]<\infty$ for all $t\in \RR$ and $h\geq 0$. In this case, we have for all $t\in\RR$ that
\begin{align} \label{eq-crosscor1}
 \bbe[V^\vp_t V^\tvp_t] &= \frac{\beta^2((\vp+\tvp)\E[S_1]-2\eta)}{(\vp\E[S_1]-\eta)(\tvp\E[S_1]-\eta)((\vp+\tvp)\E[S_1]+ \vp\tvp\var[S_1]-2\eta)},\\
\cov[V^\vp_t, V^\tvp_t]&=\frac{\beta^2 \vp\tvp \var[S_1]}{(\vp\E[S_1]-\eta)(\tvp\E[S_1]-\eta)(2\eta -(\vp+\tvp)\E[S_1]- \vp\tvp\var[S_1])},
\label{eq-crosscor1b}
\end{align}
while for all $t\in\RR$ and $h\geq 0$
\begin{equation} \label{eq-crosscor2}
 \cov[V_t^\vp, V_{t+h}^\tvp] = \ee^{h \Psi(1, \tvp)} \cov[V^\vp_0, V^\tvp_0].
\end{equation}
Both covariances in \eqref{eq-crosscor1b} and \eqref{eq-crosscor2} are nonnegative.
\end{proposition}

Now we can describe the covariance structure of the supCOGARCH process $\bar{V}^{(2)}$.

\begin{proposition} \label{propapproach2moments}
Let $\bar{V}^{(2)}$ be the strictly stationary supCOGARCH 2 process as defined in \eqref{defapproach2}.
Recall the notation $\Phi_L^{(\kappa)}$ from Eq.~\eqref{eq-def-Phimoment}.
\begin{enumerate}
 \item Suppose that $\pi(\Phi_L^{(1)})=1$. Then we have for all $t\geq0$
\begin{equation}\label{mean-supcog2}
   \E[\bar{V}^{(2)}_t]
= \int_{\Phi_L} \E[V_0^\vp]\,\pi(\dd \vp) = \beta \int_{\Phi_L} \frac{1}{\eta-\vp \bbe[S_1]} \,\pi(\dd \vp).
     \end{equation}
\item
Suppose that $\pi(\Phi_L^{(2)})=1$.
Then for $t\in\RR$ and $h\geq 0$ we have
 \begin{align}
\E[(\bar{V}^{(2)}_t)^2] &= \int_{\Phi_L}\int_{\Phi_L}\E[ V_0^{\vp}V_0^{\tvp}]\,\pi(\dd \vp) \,\pi(\dd \tvp), \label{mean2-supcog2}\\
\var[\bar{V}^{(2)}_t] &= \int_{\Phi_L}\int_{\Phi_L} \cov[V_0^{\vp},V_0^{\tvp}]\,\pi(\dd \vp) \,\pi(\dd \tvp), \label{var-supcog2}\\
\cov[\bar{V}^{(2)}_t, \bar{V}^{(2)}_{t+h}]
 &= \int_{\Phi_L}\int_{\Phi_L} \cov[V_0^\vp, V_{h}^{\tvp}] \,\pi(\dd \vp) \,\pi(\dd \tvp),\label{acf-supcog2}
 \end{align}
with $\E[ V_0^{\vp}V_0^\tvp]$ and $\cov[V_0^\vp, V_{h}^\tvp]$ as given in Proposition \ref{prop-corrcogarches}.
\end{enumerate}
Note that the quantities in \eqref{mean-supcog2}, \eqref{mean2-supcog2}, \eqref{var-supcog2} and \eqref{acf-supcog2} may be infinite.
 \end{proposition}

The tail behaviour of $\bar V^{(2)}$ is similar to the tail behaviour of the supCOGARCH~1 process.
\begin{proposition}\label{prop-tails-sup2}
Let $\bar{V}^{(2)}$ be the strictly stationary supCOGARCH 2 process as defined in \eqref{defapproach2}.
Set $\bar \vp:=\inf\{\vp>0\colon \pi((\vp,\infty))=0\}\leq\vp_{\mathrm{max}}<\infty$ and assume that there exists $\bar\kappa>0$ such that \eqref{tailcond} holds. Then we have for $\kappa>0$
\[ \lim_{x\to\infty} x^\kappa \PP[\bar V^{(2)}_0 > x]=\begin{cases} 0 &\text{if } \kappa < \bar\kappa,\\ \infty &\text{if } \kappa>\bar\kappa, \end{cases} \]
while for $\kappa = \bar\kappa$ there exists a constant $C>0$ such that
\[ \lim_{x\to\infty} x^{\bar{\kappa}} \PP[\bar V^{(2)}_0 > x]=\begin{cases} C &\text{if } \pi(\{\bar\vp\})=\bar p>0, \\ 0 &\text{if } \pi(\{\bar\vp\})=0. \end{cases}
\]
\end{proposition}

\brem Similarly to $\bar{V}^{(1)}$, the process $\bar{V}^{(2)}$ is no Markov process with respect to its augmented natural filtration (unless in the degenerate case $\pi=\delta_\vp$), but again we have a Markov property in a wide sense. More precisely, for $\FF^{(2)}=(\cF^{(2)}_t)_{t\geq 0}$
being the augmented natural filtration of $((V_t^{\vp})_{\vp\in\Phi_L})_{t\geq 0}$, we obtain for every measurable function
$f\colon\RR_+\to\RR$ and every $t\geq 0$
\[\E\Big[ f\big(\bar{V}^{(2)}_t\big)\big| \cF^{(2)}_t\Big]
= \E\left[f\big(\bar{V}^{(2)}_t\big)\big| (V_t^{\vp})_{\vp\in\Phi_L} \right].\]
\end{remark}

\subsection{The supCOGARCH 3 volatility process}\label{s33}

Our third superposition model invokes a L\'evy basis $\La^L$ on $\bbr\times\Phi_L$ such that
\[ L_t:=\La^L((0,t]\times\Phi_L),\quad t\geq0,\quad L_t:=-\La^L((-t,0]\times\Phi_L), \quad t<0, \]
exists for every $t\in\bbr$. 
With $\La^S:=[\La^L,\La^L]^\dd$ in the sense of \eqref{pjqvm}, $\La^S$ is of the form \eqref{subordbasis} and we assume that the predictable compensator of $\mu^{\La^S}$ is given by $\Pi^S(\dd t,\dd y,\dd \vp)=\dd t\,\nu_S(\dd y)\,\pi(\dd\vp)$, where $\pi$ is a probability measure on $\Phi_L$ and $\nu_S$ the L\'evy measure of the following two-sided subordinator:
\begin{equation}\label{eq-Sfrombasis}
 S_t:=\La^S((0,t]\times\Phi_L),\quad t\geq0,\quad S_t:=-\La^S((-t,0]\times\Phi_L), \quad t<0.
\end{equation}
For every $\vp\in\Phi_L$ we denote by $V^\vp$ the two-sided COGARCH process driven by $S$ as in \eqref{cog-2sided}.
The {\sl supCOGARCH 3 volatility process} $\bar V^{(3)}$ is then defined by the integral equation
\beq\label{eq-def-approach3}
 \bar{V}^{(3)}_t = \bar V^{(3)}_0 + \beta t -  \eta \int_{(0,t]}  \bar V^{(3)}_s\,\dd s + \int_{(0,t]}\int_{\Phi_L} \vp V^\vp_{s-} \,\La^S(\dd s,\dd\vp),\quad t\geq0,
\eeq
where $\bar V^{(3)}_0$ is some starting random variable independent of the restriction of $\La^L$ to $\bbr_+\times\Phi_L$
From \eqref{eq-def-approach3} it follows directly that
\begin{equation}\label{approach3jump}
 \Delta \bar{V}^{(3)}_t=\int_{\bbr_+\times\Phi_L} \vp V^\vp_{t-} y\, \mu^{\La^S} (\{t\},\dd\vp, \dd y), \quad t\geq0.
\end{equation}

We present now conditions for stationarity and calculate the second order properties.
The proofs can be found in Section~\ref{s63}.

\bpr \label{prop-app3solution}
The stochastic integral equation \eqref{eq-def-approach3} has a unique solution given by
\beq\label{supcogarch2-explicit}
\bar V^{(3)}_t=\ee^{-\eta t} \left(\bar V^{(3)}_0 + \beta \int_{(0,t]} \ee^{\eta s}\,\dd s + \int_{(0,t]} \ee^{\eta s}\,\dd A_s\right), \quad t\geq0, \eeq
where
\beq\label{processA} A_t:=\int_{(0,t]}\int_{\Phi_L} \vp V^\vp_{s-}\,\La^S(\dd s,\dd \vp),\quad t\geq0, \eeq
is a semimartingale with increasing sample paths, finite at every fixed $t\geq0$.
\epr

\bexam[Ex. \ref{example-vergleich} and \ref{example-vergleich-2} continued]\label{example-vergleich-3}
Let $\pi=p_1\delta_{\vp_1}+ p_2\delta_{\vp_2}$ with $p_1+p_2=1$ and $\vp_1,\vp_2\in\Phi_L$. As opposed to the supCOGARCH 1 process
in Example \ref{example-vergleich} or the supCOGARCH~2 process in Example \ref{example-vergleich-2}, the supCOGARCH 3 process is not
the sum of two (independent or dependent) COGARCH processes. 
In fact, there is a subordinator $S$ driving two COGARCH processes
$V^{\vp_1}$ and $V^{\vp_2}$ and each time when $S$ jumps, a value of $\vp$ is randomly chosen from $\{\vp_1,\vp_2\}$: $\vp$ takes the value $\vp_1$ with probability $p_1$ and the value $\vp_2$ with probability $p_2$. Now the jump size of the supCOGARCH 3 at a particular jump time of $S$ is exactly the jump size of the COGARCH with the chosen parameter $\vp$. If $(T_i)_{i\in\bbn}$
denote the jump times of $S$, we have \[ \Delta \bar V^{(3)}_{T_i} = \Delta V^{\vp_i}_{T_i} = \vp_i V^{\vp_i}_{T_i-}\Delta S_{T_i},\quad i\in\bbn, \]
and $(\vp_i)_{i\in\bbn}$ is an i.i.d. sequence with distribution $\pi$.
Moreover, $(\vp_i)_{i\in\bbn}$ is independent of $S$. This effect is illustrated in Figure \ref{figsup3}.
\eexam
\setlength{\unitlength}{0.8cm}
\begin{figure}[!ht]
\begin{picture}(0,5)
\put(0,0){\includegraphics{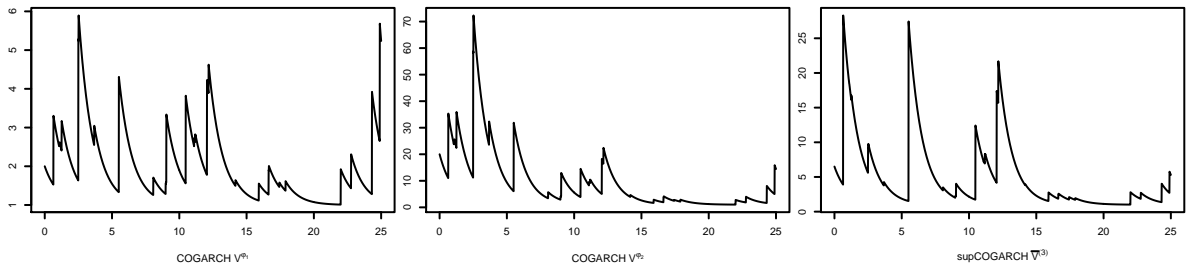}}
\end{picture}
\caption{\scriptsize Two COGARCH processes $V^{\vp_1}$ and $V^{\vp_2}$ driven by the same L\'evy process $L$ and the resulting supCOGARCH $\bar V^{(3)}$. $L$ is a compound Poisson process with rate $1$ and standard normal jumps. The parameters are the same as in Figure~ \ref{figsup1}.}
\label{figsup3}
\end{figure}

The next theorem establishes necessary and sufficient conditions for the existence of a stationary distribution of the supCOGARCH 3 process.

\bthe \label{thm-approach3statsol}
Define the supCOGARCH 3 process $(\bar{V}^{(3)}_t)_{t\geq 0}$ by \eqref{supcogarch2-explicit}. Then a finite random variable $\bar{V_0}^{(3)}$ can be chosen
such that $\bar{V}^{(3)}$ is strictly stationary if and only if
\beq\label{statcond}
\int_{\bbr_+} \int_{\Phi_L} \int_{\bbr_+} 1\wedge (y \vp V_s^\vp \ee^{-\eta s}) \,\dd s\,\pi(\dd\vp)\,\nu_S(\dd y)
< \infty\quad\text{a.s.} \eeq
In the case that a stationary distribution exists, it is uniquely determined by the law of $\frac{\beta}{\eta}+\int_{\bbr_+} \ee^{-\eta s} \,\dd A_s$.
In particular, setting $\bar V^{(3)}_0 := \frac{\beta}{\eta}+\int_{(-\infty,0]}\int_{\Phi_L} \ee^{\eta s}\vp V^\vp_{s-}\,\La^S(\dd s,\dd\vp)$, we obtain the two-sided
stationary supCOGARCH 3 process
\begin{equation} \label{supcogarch3-explicitstationary}
 \bar V^{(3)}_t = \ee^{-\eta t} \left( \beta \int_{(-\infty,t]} \ee^{\eta s}\,\dd s + \int_{(-\infty,t]} \ee^{\eta s}\,\dd A_s \right)
= \frac{\beta}{\eta} + \int_{(-\infty,t]}\int_{\Phi_L} \ee^{-\eta(t-s)}\vp V^\vp_{s-}\,\La^S(\dd s,\dd\vp)
\end{equation}
 for $t\in\RR$.
Moreover, \eqref{statcond} holds in each of the following cases:
\begin{enumerate} \setlength{\itemsep}{-0.2ex}
\setlength{\parsep}{0ex}
\item $\pi([0,\vp_0])=1$ with some $\vp_0<\vp_{\mathrm{max}}$.
\item $\pi(\Phi_L^{(\kappa)}) = 1$ for some $\kappa>0$.
\end{enumerate}
\ethe

The second-order properties of the strictly stationary supCOGARCH 3 process are as follows.

\begin{proposition}\label{secord-supcog3}
Let $\bar{V}^{(3)}$ be the stationary supCOGARCH 3 process given by \eqref{supcogarch3-explicitstationary}.
Recall the notation $\Phi_L^{(\kappa)}$ from Eq.~\eqref{eq-def-Phimoment}.
\begin{enumerate}
 \item
Assume that $\pi(\Phi_L^{(1)})=1$. Then for $t\in\bbr$
\beq\label{mean-supcog3}
\bbe[\bar V^{(3)}_t]=\int_{\Phi_L} \bbe[V^\vp_0]\,\pi(\dd\vp)= \int_{\Phi_L} \frac{\beta}{\eta-\E[S_1]\vp}\,\pi(\dd\vp).
\eeq
\item
Assume that $\pi(\Phi_L^{(2)})=1$. Then
with $\bbe[V_0^\vp V_0^\tvp]$ and $\cov[V_0^\vp,V_0^\tvp]$ as given in Proposition~\ref{prop-corrcogarches},
for $t\in\bbr$ and $h\ge0$ we have
\begin{equation}\label{mean2-supcog3} \bbe[(\bar V_t^{(3)})^2]= \int_{\Phi_L}\int_{\Phi_L} \left(\bbe[V_0^\vp V_0^\tvp]  + \frac{\beta}{\eta} \frac{\var[V_0^\vp]-\cov[V_0^\vp,V_0^\tvp]}{\bbe[V_0^\vp]} \right)\,\pi(\dd\tvp)\,\pi(\dd\vp). \end{equation}
\begin{align}\label{acf-supcog3}
\lefteqn{\cov[\bar V_t^{(3)},\bar V_{t+h}^{(3)}]}\\
=&~ \int_{\Phi_L}\int_{\Phi_L} \left( \ee^{h \Psi(1,\vp)} \cov[V_0^\vp,V_0^\tvp]
+ \ee^{-\eta h}\frac{\beta}{\eta} \frac{\var[V_0^\vp]-\cov[V_0^\vp,V_0^\tvp]}{\bbe[V_0^\vp]}  \right)\,\pi(\dd\tvp)\,\pi(\dd\vp).\nonumber
\end{align}
\end{enumerate}
Note that the quantities in \eqref{mean-supcog3}, \eqref{mean2-supcog3} and \eqref{acf-supcog3} may be infinite.
\end{proposition}

The supCOGARCH 3 process also exhibits Pareto-like tails.

\begin{proposition}\label{tail-supCOG3}
Let $\bar{V}^{(3)}$ be the stationary supCOGARCH 3 process given by \eqref{supcogarch3-explicitstationary}.
Set $\bar \vp:=\inf\{\vp>0\colon \pi((\vp,\infty))=0\}\leq\vp_{\mathrm{max}}<\infty$ and assume that there exists $\bar\kappa>0$ such that \eqref{tailcond} is fulfilled.
Then for $\kappa>0$
\[ \lim_{x\to\infty} x^\kappa \PP[\bar V^{(3)}_0 > x]=\begin{cases} 0 &\text{if } \kappa < \bar\kappa,\\ \infty &\text{if } \kappa>\bar\kappa, \end{cases}\]
and for $\kappa = \bar\kappa$ and $\pi(\{\bar\vp\})=0$ we have
\[ \lim_{x\to\infty} x^{\bar{\kappa}} \PP[\bar V^{(3)}_0 > x]=0, \]
while for $\kappa = \bar\kappa$ and $\pi(\{\bar\vp\})=\bar p>0$
\[ 0<C_\ast := \liminf_{x\to\infty} x^{\bar\kappa}\PP[\bar V^{(3)}_0 > x] \leq \limsup_{x\to\infty} x^{\bar\kappa}\PP[\bar V^{(3)}_0 > x] =: C^\ast <\infty.  \]
\end{proposition}

\brem Just like $\bar V^{(1)}$ and $\bar V^{(2)}$, the process $\bar{V}^{(3)}$ is not a Markov process with respect to its augmented natural filtration
(unless in the case $\pi=\delta_\vp$), but, denoting the augmented natural filtration of $((V^\vp_t)_{\vp\in\Phi_L})_{t\geq0}$ by $\FF^{(3)}=(\cF^{(3)}_t)_{t\geq0}$, we obtain for every measurable function
$f\colon\RR_+\to\RR$ and every $t\geq0$
\[\E\Big[ f\big(\bar{V}^{(3)}_t\big)\big| (\cF^{(3)}_s)_{s\leq t}\Big]
= \E\Big[f\big(\bar{V}^{(3)}_t\big)\big| (V_t^{\vp})_{\vp\in\Phi_L} \Big].\]
\end{remark}

\section{The price processes}\label{s5}

Recall that in the COGARCH model, or its discrete-time analogue, the GARCH model (cf. \cite{ben:1995}),
the driving noises for volatility and price processes are the same \eqref{eq-pricecog}.
In this section, we suggest and investigate price processes corresponding to the supCOGARCH volatility processes. All proofs can be found in Section~\ref{s64}.

\subsection{The integrated supCOGARCH 1 price process}\label{s51}

For the supCOGARCH 1 volatility process $\bar V^{(1)}$ as defined in Section \ref{s31}, there is no canonical choice for a price process, since a whole sequence $(L^{\vp_i})_{i\in\bbn}$ of L\'evy processes is used in its definition.
Hence a priori any function of this sequence is a reasonable candidate for the driver in the price process.
As a simple example we take the L\'evy process $L^{\vp_1}$ as integrator; i.e., we define
\beq\label{eq-def-price1}
G^{(1)}_t := \int_{(0,t]} \sqrt{\bar V^{(1)}_{s-}}\,\dd L^{\vp_1}_s,\quad t\geq0.
\eeq
It is an interesting observation that this process not only allows for common jumps of volatility and price (as it is usual in the standard COGARCH model), but also for jumps only in the volatility and not in the price process.
There is evidence that this happens in real data (cf. \cite{Jacod10}).

It is obvious from the definition that, if $(\bar{V}^{(1)}_t)_{t\geq0}$ is strictly stationary, then $(G^{(1)}_t)_{t\geq 0}$ has stationary increments. Furthermore, its second-order structure is comparable to that of the integrated COGARCH process \cite[Prop. 5.1]{KLM:2004}.

\begin{theorem}\label{price-supCOG1}
Let $\bar{V}^{(1)}=\sum_{i=1}^\infty p_i V^{\vp_i}$, $\vp_i\in\Phi_L$, be a stationary supCOGARCH 1 process as defined in Section \ref{s31},
where each $V^{\vp_i}$ is driven by $S^{\vp_i}=[L^{\vp_i},L^{\vp_i}]^\dd$ and $(L^{\vp_i})_{i\in\bbn}$ are i.i.d. copies of a
L\'evy processes $L$ with zero mean. Define the price process $G^{(1)}$ by \eqref{eq-def-price1} and set
\[\Delta^r G^{(1)}_t:= G^{(1)}_{t+r}-G^{(1)}_t= \int_{(t,t+r]} \sqrt{\bar{V}^{(1)}_{s-}} \,\dd L^{\vp_1}_s, \quad t\geq0, r>0.\]
Recall the notation $\Phi_L^{(\kappa)}$ from Eq.~\eqref{eq-def-Phimoment} and that the support of $\pi$ is countable in this case.
\begin{enumerate}
 \item Assume that $\pi$ has support in $\Phi_L^{(1/2)}$. Then
$$\E[\Delta^r G^{(1)}_t]=0, \quad t\geq0, r>0.$$
\item If further $\bbe[L_1^2]<\infty$ and $\pi$ has support in $\Phi_L^{(1)}$, then for $t\in\RR$, $h\geq r>0$
\begin{align*}
 \bbe[(\Delta^r G^{(1)}_t)^2]&= r \bbe[L_1^2] \bbe[\bar{V}^{(1)}_0] = r \bbe[L_1^2] \int_{\Phi_L^{(1)}} \frac{\beta}{\eta-\vp(\bbe[L_1^2]-\sigma_L^2)}\, \pi(\dd \vp) \quad\mbox{and} \\
\cov[\Delta^r G^{(1)}_t,\Delta^r G^{(1)}_{t+h}]&= 0.
\end{align*}
\item Assume further that $\bbe[L_1^4]<\infty$, $\int_{\bbr} y^3\,\nu_L(\dd y)=0$ and that $\pi\neq\delta_0$ has support in $\Phi_L^{(2)}$. Then for $t\in\RR$, $h\geq r>0$
\begin{align*}
\cov[(\Delta^r G^{(1)}_t)^2,(\Delta^r G^{(1)}_{t+h})^2]
&= \bbe[L_1^2]  \int_{\Phi_L^{(2)}} \frac{\ee^{h\Psi(1,\vp)}-\ee^{(h-r)\Psi(1,\vp)}}{\Psi(1,\vp)} \cov[(\Delta^r G^{(1)}_0)^2, V_r^\vp]\,\pi(\dd \vp)\\
&>0.
\end{align*}
\end{enumerate}
\end{theorem}

\subsection{The integrated supCOGARCH 2 price process}\label{s52}

Let $(L_t)_{t\in\bbr}$ be a two-sided L\'evy process, define the subordinator $S$ by \eqref{eq-def-S2sided}
and let $(\bar{V}^{(2)}_t)_{t\in \RR}$ be the supCOGARCH 2 process driven by $S$ as defined in Section \ref{s32}.
In view of the standard definition of the integrated COGARCH price process \eqref{eq-pricecog}
it makes sense to define the {\sl integrated supCOGARCH 2 price process} by
\begin{equation} \label{eq-def-price2}
 \dd G^{(2)}_t:= \sqrt{\bar{V}^{(2)}_{t-}} \,\dd L_t, \quad G^{(2)}_0=0,\quad t\in \RR.
\end{equation}
Hence, as in the standard COGARCH model, the process $G^{(2)}$ jumps exactly when the volatility $\bar V^{(2)}$ jumps.
Also $(G_t^{(2)})_{t\in\RR}$ has stationary increments if $(\bar{V}^{(2)}_t)_{t\in \RR}$ is strictly stationary. The integrated supCOGARCH 2 process
has the same second-order structure as the integrated supCOGARCH 1 process and, hence, as the integrated COGARCH process as shown in the following.

\begin{theorem} \label{prop-price2}
Suppose that the two-sided L\'evy process $L$ has expectation $0$, define $S$ by \eqref{eq-def-S2sided}, the supCOGARCH volatility $\bar{V}^{(2)}$ as in Section \ref{s32}
with $\pi(\Phi_L)=1$ and the process $G^{(2)}$ by \eqref{eq-def-price2}. 
Set
$$\Delta^r G^{(2)}_t:= G^{(2)}_{t+r}-G^{(2)}_t= \int_{(t,t+r]} \sqrt{\bar{V}^{(2)}_{s-}} \,\dd L_s, \quad t\in\RR, r>0.$$
\begin{enumerate}
 \item Assume that $\pi$ has support in $\Phi_L^{(1/2)}$. Then
$$\E[\Delta^r G^{(2)}_t]=0, \quad t\in\RR, r>0.$$
\item If further $\bbe[L_1^2]<\infty$ and $\pi$ has support in $\Phi_L^{(1)}$, then for $t\in\RR$, $h\geq r>0$
\begin{align*}
 \bbe[(\Delta^r G^{(2)}_t)^2]=&~ r \bbe[L_1^2] \bbe[\bar{V}^{(2)}_0] = r \bbe[L_1^2] \int_{\Phi_L^{(1)}} \frac{\beta}{\eta-\vp(\bbe[L_1^2]-\sigma_L^2)}\, \pi(\dd \vp) , \\
\cov[\Delta^r G^{(2)}_t,\Delta^r G^{(2)}_{t+h}]=&~ 0.
\end{align*}
\item Assume further that $\bbe[L_1^4]<\infty$, $\int_{\bbr} y^3\,\nu_L(\dd y)=0$ and $\pi\neq\delta_0$ has support in $\Phi_L^{(2)}$. Then for $t\in\RR$, $h\geq r>0$
    \hspace*{-1cm}
\begin{align*}
\cov[(\Delta^r G^{(2)}_t)^2,(\Delta^r G^{(2)}_{t+h})^2]=&~ \bbe[L_1^2]  \int_{\Phi_L^{(2)}} \frac{\ee^{h\Psi(1,\vp)}-\ee^{(h-r)\Psi(1,\vp)}}{\Psi(1,\vp)} \cov[(\Delta^r G^{(2)}_0)^2, V_r^\vp] \,\pi(\dd \vp)\\
>&~0.
\end{align*}
\end{enumerate}
\end{theorem}

\subsection{The integrated supCOGARCH 3 price process}\label{s53}

As in the case of the supCOGARCH 2 there is a canonical choice for the driving noise in the price process of the supCOGARCH 3.
With $L$ being a L\'evy process and $V^{(3)}$ the stationary supCOGARCH 3 as defined in \eqref{supcogarch3-explicitstationary},
we define the {\sl integrated supCOGARCH 3 price process} by
\begin{equation} \label{eq-def-price3}
 G^{(3)}_t:= \int_{(0,t]} \sqrt{\bar{V}^{(3)}_{t-}} \,\dd L_t, \quad t\geq0.
\end{equation}
Evidently, $G^{(3)}$ has stationary increments and, if $\pi(\{0\})=0$, it jumps at exactly the times when $\bar V^{(3)}$ jumps. However, whenever $\pi(\{0\})>0$, the supCOGARCH 3 model features price jumps without volatility jumps, a behaviour attested by the empirical findings of \cite{Jacod10}. The second-order structure of $G^{(3)}$ is calculated in the following theorem.

\begin{theorem} \label{prop-price3}
Suppose that $L$ is a L\'evy process with expectation $0$ and that $\pi(\Phi_L)=1$. Define $V^{(3)}$ by \eqref{supcogarch3-explicitstationary} and set
\[\Delta^r G^{(3)}_t:= G^{(3)}_{t+r}-G^{(3)}_t= \int_{(t,t+r]} \sqrt{\bar{V}^{(3)}_{s-}} \,\dd L_s, \quad t\geq0, r>0.\]
\begin{enumerate}
 \item Assume that $\pi$ has support in $\Phi_L^{(1/2)}$. Then
$$\E[\Delta^r G^{(3)}_t]=0, \quad t\geq0, r>0.$$
\item If further $\bbe[L_1^2]<\infty$ and $\pi$ has support in $\Phi_L^{(1)}$, then for $t\geq0$ and $h\geq r>0$
\begin{align*}
 \bbe[(\Delta^r G^{(3)}_t)^2]
= &~r \bbe[L_1^2] \bbe[\bar{V}^{(3)}_0] = r \bbe[L_1^2] \int_{\Phi_L^{(1)}} \frac{\beta}{\eta-\vp(\bbe[L_1^2]-\sigma_L^2)}\, \pi(\dd \vp) , \\
\cov[\Delta^r G^{(3)}_t,\Delta^r G^{(3)}_{t+h}]= &~0.
\end{align*}
\item Assume further that $\bbe[L_1^4]<\infty$, $\int_{\bbr} y^3\,\nu_L(\dd y)=0$ and $\pi\neq\delta_0$ has support in $\Phi_L^{(2)}$. Then for $t\geq0$ and $h\geq r>0$
\begin{align*}
&~\cov[(\Delta^r G^{(3)}_t)^2,(\Delta^r G^{(3)}_{t+h})^2]\\
=&~\bbe[L_1^2] \left[ \frac{\ee^{-\eta(h-r)}-\ee^{-\eta h}}{\eta} \cov[(\Delta^r G^{(3)}_0)^2, \bar V^{(3)}_r] \right.\\
&\quad\quad\quad+\left. \int_{\Phi_L^{(2)}} \left(\frac{\ee^{h\Psi(1,\vp)}-\ee^{(h-r)\Psi(1,\vp)}}{\Psi(1,\vp)}+\frac{\ee^{-\eta h}-\ee^{-\eta(h-r)}}{\eta}\right) \cov[(\Delta^r G^{(3)}_0)^2, V_r^\vp] \,\pi(\dd \vp)\right]\\
>&~0.
\end{align*}
\end{enumerate}
\end{theorem}

\section{Comparison and conclusions}\label{s4}

This section is devoted to highlight the analogies and differences between the three supCOGARCH processes, and to compare them to the standard COGARCH process. 
First note that in all three models, setting $\pi=\delta_\vp$ for $\vp\in\Phi_L$ yields the standard COGARCH process $(V_t^\vp)_{t\geq 0}$ as defined in \eqref{cog}. Hence it seems natural that some features of the COGARCH process are preserved under superpositioning. The next remark summarizes the most important properties.

\brem
(a) \, Comparing the autocovariance functions of the supCOGARCH volatility processes (cf. \eqref{acf-supcog1}, \eqref{acf-supcog2} and \eqref{acf-supcog3}) to those of the COGARCH (cf. \eqref{acf-cog}), we find for large lags $h$ exponential decay in all three supCOGARCH models, but allowing for more flexibility than in the COGARCH model for small and medium lags. \\[2mm]
(b) \, The important property of Pareto-like tails of the stationary distribution of a COGARCH process \cite[Thm.~6]{klm:2006} persists as shown in Propositions~\ref{tail-supCOG1}, \ref{prop-tails-sup2} and~\ref{tail-supCOG3}.\\[2mm]
(c) \, Another similarity is given in the behaviour between jumps, where the COGARCH process exhibits exponential decay \cite[Prop. 2]{klm:2006}.
More precisely, assuming that $\bar V^{(1)}$, $\bar V^{(2)}$ and $\bar V^{(3)}$ only have finitely many jumps on compact intervals, and fixing two consecutive jump times $T_j<T_{j+1}$, we obtain for $i\in\{1,2,3\}$ and $t\in(T_j,T_{j+1})$
\[ \frac{\dd}{\dd t}\bar V^{(i)}_t = \beta - \eta \bar V^{(i)}_t,\quad \bar V^{(i)}_t = \frac{\beta}{\eta} + \left(\bar V^{(i)}_{T_j}-\frac{\beta}{\eta}\right)\ee^{-\eta(t-T_j)}. \]
(d) \, An important difference between the supCOGARCH processes and the COGARCH process is the jump behaviour. This is highlighted in Corollary \ref{jointjumps} and Example \ref{example-vergleich0}. \\[2mm]
(e) \, In general, all supCOGARCH models have common jumps in volatility and price as it is characteristic for the COGARCH model. Additionally, the supCOGARCH 1 model also features volatility jumps without price jumps and the supCOGARCH 3, if $\pi(\{0\})>0$, also price jumps without volatility jumps. Moreover, if we replace $L^{\vp_1}$ in \eqref{eq-def-price1} by a (finite or infinite) linear combination of $(L^{\vp_i})_{i\in\bbn}$, we can control the proportion of common volatility and price jumps to sole volatility jumps in the supCOGARCH 1 model. \\[2mm]
(f) \, Our three models have different degrees of randomness in the following sense. 
The supCOGARCH 1 is defined via a sequence of independent L\'evy processes. So by the adjustment of $\pi$ there is an arbitrary degree of randomness in the model.
The supCOGARCH 2 model has only one single source of randomness, namely the driving L\'evy process. 
Finally, the supCOGARCH 3 incorporates two sources of randomness: one originating from the L\'evy process $L=\La^L((0,\cdot]\times\Phi_L)$ and one from the sequence $(\vp_i)_{i\in\bbn}$ chosen at the jump times of $L$.  
\erem

One of the motivations for this study was the observation made in \cite{JKMc} that for a COGARCH process $(V^\vp,G^\vp)$ there is always a deterministic relationship between volatility jumps and price jumps given by
\[ q^{\vp}_T :=\frac{\phi(V^\vp_{T-},V^\vp_{T})}{\psi(G^\vp_{T-},G^\vp_{T})} \equiv \vp \]
for deterministic functions 
\beq\label{eq-jumprelations} 
\psi(x,y) = (y-x)^2,\quad \phi(x,y)=y-x.  
\eeq
and every jump time $T$ of the driving L\'evy process. 

From the following corollary, which is a direct consequence of the respective definitions, we see immediately that for all three supCOGARCH models such a deterministic functional relationship between volatility and price jumps is no longer present.

\bco\label{jointjumps} Let $T$ be a jump time of $L^{\vp_1}$ for the supCOGARCH~1, and a jump time of $L$ for the supCOGARCH 2 and 3. Furthermore, define $\bar \vp:=\inf\{\vp>0\colon \pi((\vp,\infty))=0\}$ and $\underline{\vp}:=\sup\{\vp>0\colon \pi((0,\vp))=0\}$ (using the convention $\sup \emptyset := 0$, $\inf \emptyset:=\infty$).
\ben 
\item We have
\begin{align}
\Delta \bar V^{(1)}_T &= p_1 \vp_1 V^{\vp_1}_{T-} (\Delta L_T^{\vp_1})^2,&\quad \Delta G^{(1)}_T &= \sqrt{\sum_{i=1}^\infty p_i V^{\vp_i}_{T-}} \Delta L^{\vp_1}_T\label{rel1}\\
\Delta \bar V^{(2)}_T &= \int_{\Phi_L} \vp V^\vp_{T-}\,\pi(\dd\vp) (\Delta L_T)^2,&\quad \Delta G^{(2)}_T &= \sqrt{\int_{\Phi_L} V^\vp_{T-} \,\pi(\dd\vp)} \Delta L_T\label{rel2}\\
\Delta \bar V^{(3)}_T &= \vp_T V^{\vp_T}_{T-}(\Delta L_T)^2,&\quad \Delta G^{(3)}_T &= \sqrt{\bar V^{(3)}_{T-}}\Delta L_T,\label{rel3}
\end{align}
where in the last line $\vp_T$ is a random variable which has distribution $\pi$ and is independent of $L$.
\item Define
\beq\label{quot} q^{(i)}_T := \frac{\phi(\bar V^{(i)}_{T-},\bar V^{(i)}_{T})}{\psi(G^{(i)}_{T-},G^{(i)}_{T})} \eeq
for $i=1,2,3$ with $\phi$ and $\psi$ given in \eqref{eq-jumprelations}. Then we have
\[ q^{(1)}_T \leq \bar\vp \quad\text{and}\quad \underline\vp \leq q^{(2)}_T \leq \bar\vp; \]
moreover, if $\vp_T = \bar\vp$ (resp. $\vp_T=\underline\vp$), we have 
\[ q^{(3)}_j\geq \bar\vp \quad (\text{resp. } q^{(3)}_j\leq\underline\vp). \]
\een
\eco

\bexam[Ex. \ref{example-vergleich}, \ref{example-vergleich-2} and \ref{example-vergleich-3} continued]\label{example-vergleich0}
Let us compare the jumps in the supCOGARCH volatility processes for $\pi=p_1\delta_{\vp_1}+ p_2\delta_{\vp_2}$ with $p_1+p_2=1$ and $\vp_1,\vp_2\in\Phi_L$: 
We see from \eqref{rel1} that in the supCOGARCH 1 model a squared jump of $L^{\vp_i}$ is always scaled with $p_i \vp_i V_{t-}^{\vp_i}$ and, hence, the parameter $\vp_i$ as well as the weight $p_i$ take part in the scaling. In contrast, defining $S^{\vp_i}=\La^S((0,\cdot]\times\{\vp_i\})$ for $i=1,2$ in the case of the supCOGARCH 3 process, each jump of $S=S^{\vp_1}+S^{\vp_2}=[L,L]^\dd$ is scaled with $\vp_1 V_{t-}^{\vp_1}$ or $\vp_2 V_{t-}^{\vp_2}$, depending on whether $S^1$ or $S^2$ actually jumps. 
Here the probabilities $p_i$ do not influence the scaling of the jump, but the intensity of the driving processes $S^{\vp_i}$, in other words, the $p_i$ determine the probability for the value $\vp_i$ to be chosen at a specific jump time. 
Finally, for the supCOGARCH 2 process, the jump size of the subordinator $S=[L,L]^\dd$ is always scaled with $p_1\vp_1 V^{\vp_1}_{t-} + p_2\vp_2 V^{\vp_2}_{t-}$, so all weights and parameters are involved here.
\eexam

\subsubsection*{Simulation results}
To illustrate the theoretical findings above, we present simulations of the different supCOGARCH volatility processes as well as the different price processes in Figures~\ref{figvol} and \ref{figprice} below.
As L\'evy process $L$ we choose a variance gamma process arising through time changing a standard Brownian motion by an independent gamma process 
with mean and variance $1$. 

Note that we have chosen different parameters for the simulations presented in Figures \ref{figvol} and \ref{figprice}, respectively, in order to better visualize the differences between the three volatility and the three price processes.

To illustrate the profound difference between the COGARCH and the three supCOGARCH models with reference to \eqref{eq-jumprelations}, we also compute $q^{(1)}, q^{(2)}$ and $q^{(3)}$ as defined in \eqref{quot} for the jump times of the simulation in Figure~\ref{figprice}.
The histograms of $\log q^{(i)}$ are given in Figure~\ref{fighist}. We see that both the supCOGARCH~1 and 2 exhibit a certain interval of values for $\log q^{(1)}$ and $\log q^{(2)}$. As we would expect from Corollary \ref{jointjumps}, both $\log q^{(1)}$ and $\log q^{(2)}$ are bounded from above by $\log \vp_2$, but only $\log q^{(2)}$ is bounded from below by $\log \vp_1$ whereas the $\log q^{(1)}$ has a relatively long tail on the negative side. Also, in general, the values of $q^{(1)}$ tend to be smaller than those of $q^{(2)}$. This is due to the fact that at a common jump time of volatility and price, the volatility jump size is the sum of two terms for the supCOGARCH 2 but only a single term for the supCOGARCH~1 (see \eqref{rel2} and \eqref{rel1}).
As a result, the nominator in \eqref{quot} is usually smaller for the supCOGARCH 1 than for the supCOGARCH 2. Finally, again in coincidence with Corollary~\ref{jointjumps}, the supCOGARCH~3 shows two disjoint intervals for the values of $q^{(3)}$, corresponding to the two different values of $\vp$ chosen for the superposition.

\setlength{\unitlength}{0.8cm}
\begin{figure}[!ht]
\begin{picture}(0,4)
\put(0,0){\includegraphics{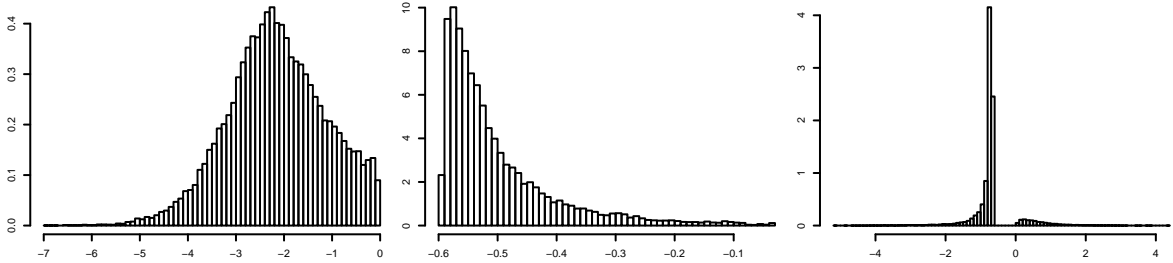}}
\end{picture}
\caption{\scriptsize The pictures (from left to right) show the histograms for $\log(q^{(1)})$, $\log(q^{(2)})$ and $\log(q^{(3)})$.}
\label{fighist}
\end{figure}

\newpage
\setlength{\unitlength}{0.8cm}
\begin{figure}[!ht]
\begin{picture}(6,25)
\put(0,24){a) $L$}
\put(0,19){b) $V^{\vp_2}$}
\put(0,14){c) $\bar V^{(1)}$}
\put(0,9){d) $\bar V^{(2)}$}
\put(0,4){e) $\bar V^{(3)}$}
\put(1.5,20){\includegraphics{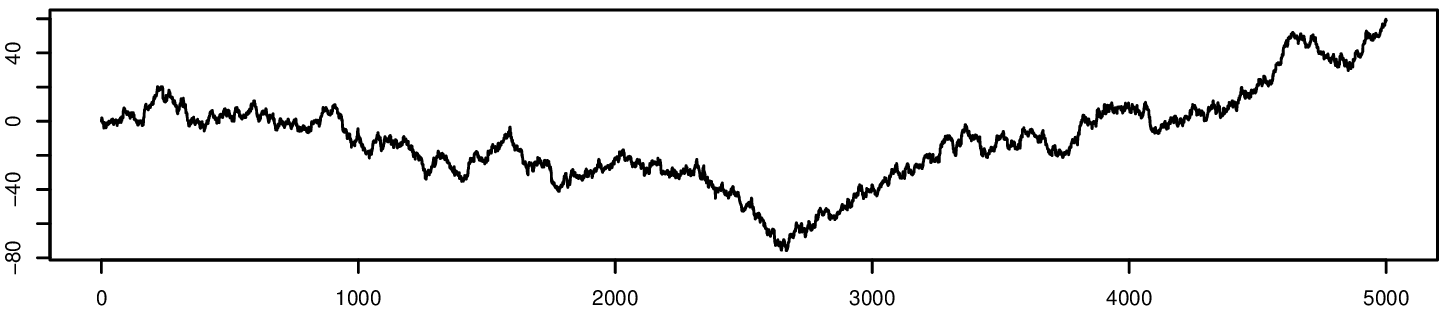}}
\put(1.5,15){\includegraphics{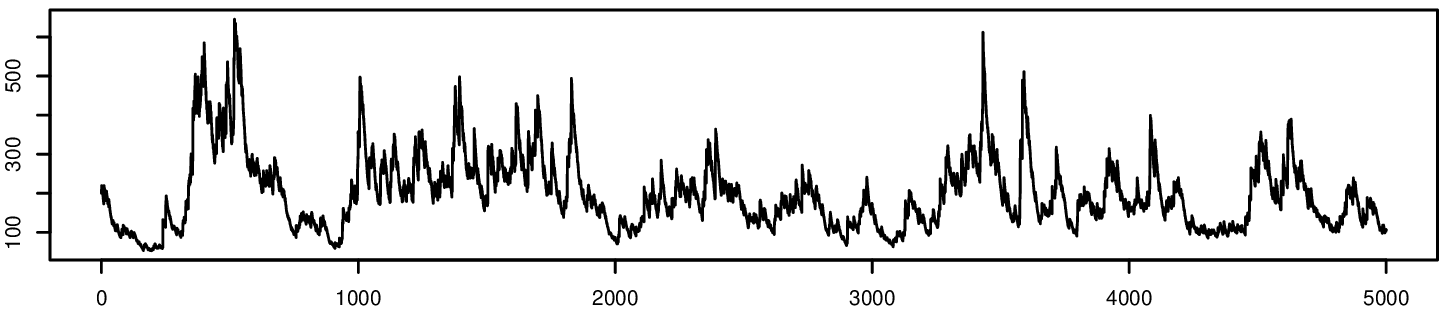}}
\put(1.5,10){\includegraphics{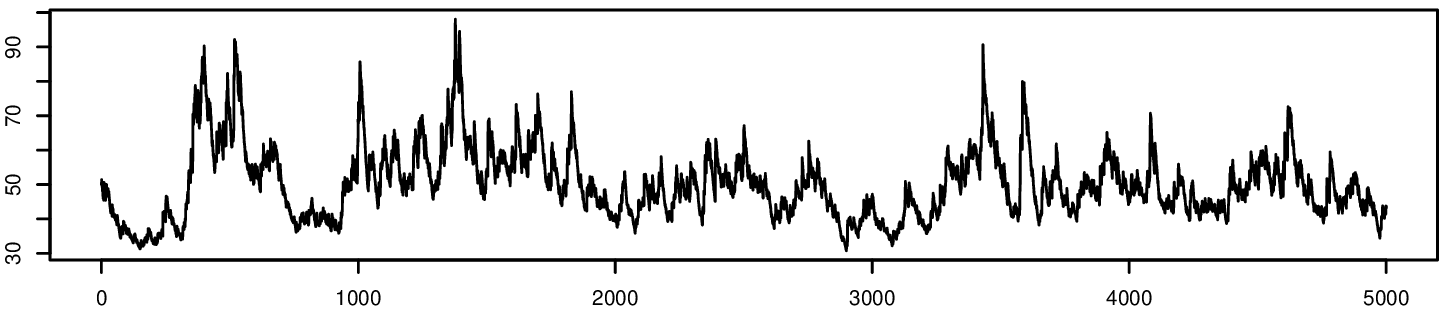}}
\put(1.5,5){\includegraphics{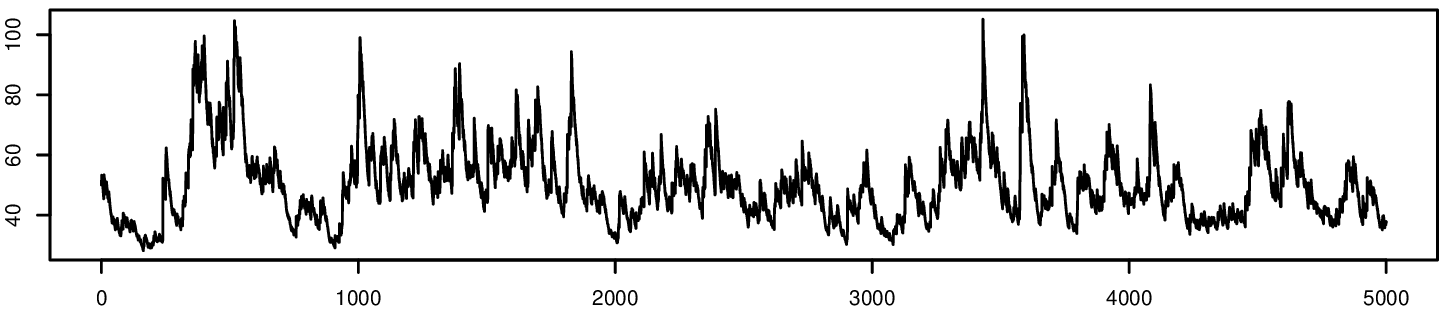}}
\put(1.5,0){\includegraphics{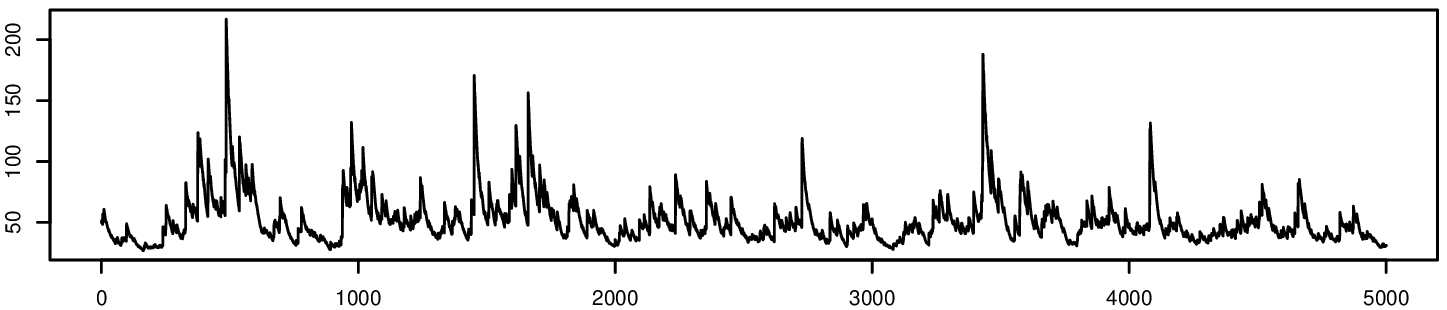}}
\end{picture}
\caption{\scriptsize The parameters are: $\beta=1$, $\eta=0.05$, $\vp_1=0.02$, $\vp_2=0.045$, $\pi=0.9\delta_{\vp_1}+0.1\delta_{\vp_2}$, starting value is the mean; a) $L$ is a variance gamma process with mean $0$ and variance $1$; b) COGARCH process driven by $L$ with parameter $\vp_2$; c) supCOGARCH process $\bar V^{(1)}$ where $V^{\vp_2}$ is driven by $L$ and $V^{\vp_1}$ is driven by an independent copy of $L$; d) supCOGARCH process $\bar V^{(2)}$ driven by $L$; e) supCOGARCH process $\bar V^{(3)}$ driven by $L$.}\label{figvol}
\end{figure}

\newpage
\setlength{\unitlength}{0.8cm}
\begin{figure}[!ht]
\begin{picture}(6,25)
\put(0,24){a) $L$}
\put(0,19){b) $G^{\vp_2}$}
\put(0,14){c) $G^{(1)}$}
\put(0,9){d) $G^{(2)}$}
\put(0,4){e) $G^{(3)}$}
\put(1.5,20){\includegraphics{VG-levy.eps}}
\put(1.5,15){\includegraphics{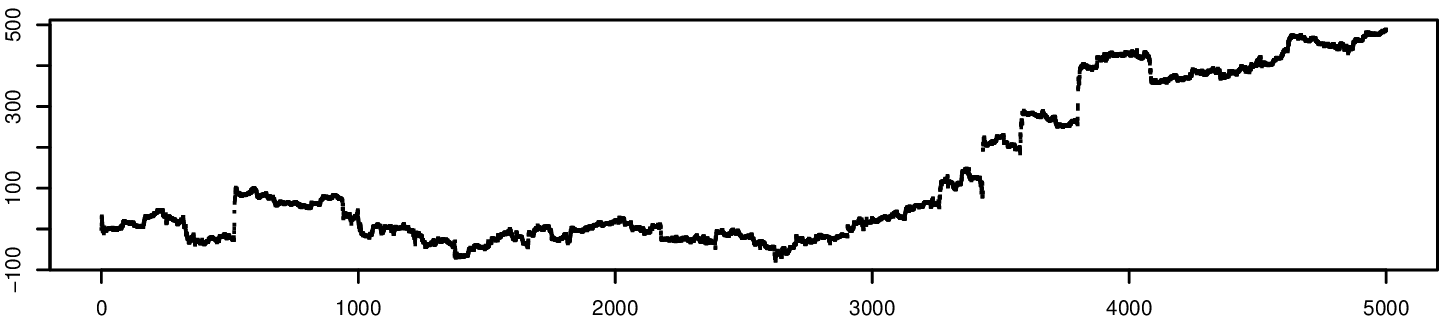}}
\put(1.5,10){\includegraphics{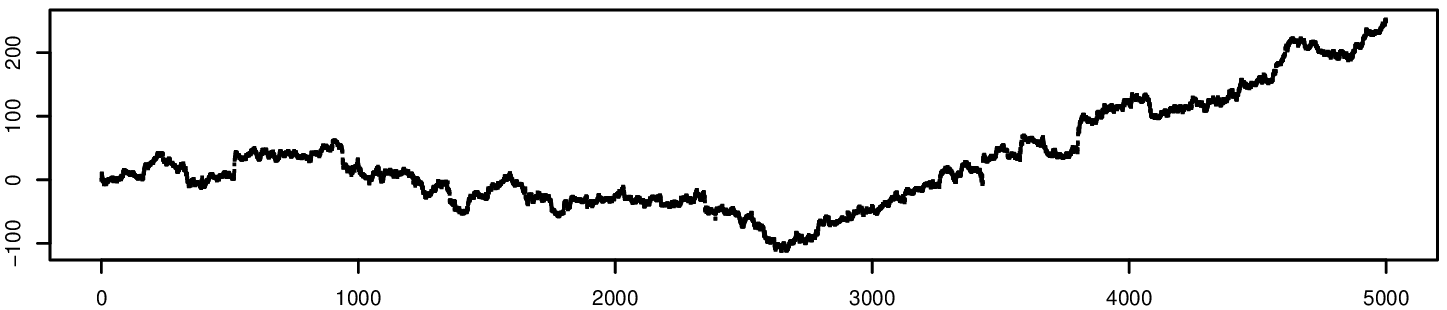}}
\put(1.5,5){\includegraphics{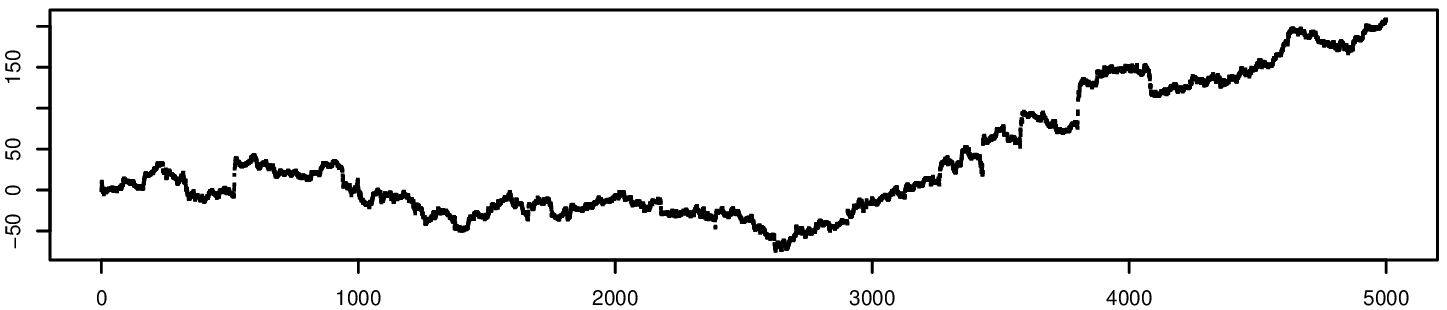}}
\put(1.5,0){\includegraphics{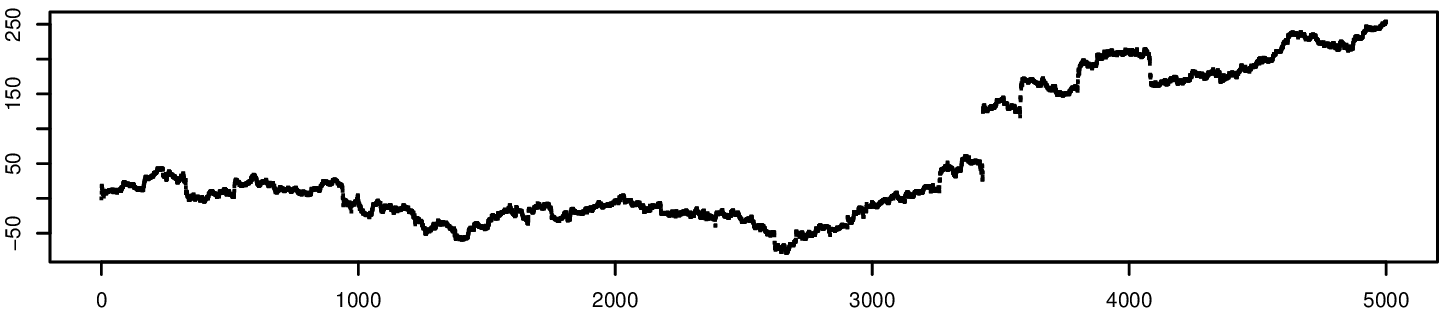}}
\end{picture}
\caption{\scriptsize The parameters are: $\beta=1$, $\eta=1$, $\vp_1=0.5$, $\vp_2=0.995$, $\pi=0.9\delta_{\vp_1}+0.1\delta_{\vp_2}$; a) $L$ is the same L\'evy process as in Figure \ref{figvol}; b) COGARCH price process driven by $L$ with parameter $\vp_2$; c), d) and f) supCOGARCH price processes $G^{(1)}$, $G^{(2)}$ and $G^{(3)}$ driven by $L$.}\label{figprice}
\end{figure}
\newpage

\subsubsection*{Estimation}

 A thorough investigation of the statistical analysis of the supCOGARCH processes via parameter estimation goes far beyond the scope of this paper. Nevertheless let us shortly comment on the main task, namely the estimation of the superposition measure $\pi$, for which no standard estimation procedure is available as it is typical for multifactor models.
 
In the case of the supOU stochastic volatility model, several attempts have been made to infer the underlying superposition measure. For example, assuming the form $\pi=\sum_{i=1}^K p_i\delta_{\vp_i}$ for some known $K\in\bbn$, in \cite{Barn:Shep:2001,Barn:Shep:2012} a least-square fit of the autocovariance function is employed. In \cite{STW:13} a generalized method of moments is used to estimate the supOU model under the hypothesis of a gamma distribution for $\pi$. Further, in \cite{Griffin:2011, Griffin:Steel:2010} a Bayesian nonparametric approach is proposed in the case that $\pi$ is a discrete or continuous measure, respectively. Whether and how these approaches, or the estimation procedures for the COGARCH model mentioned in the Introduction can be adapted to the supCOGARCH case, is open.

\section{Proofs and auxiliary results}\label{s6}

\subsection{Proofs for Section \ref{s31}}\label{s61}

\begin{proof}[Proof of Theorem \ref{thm-app1-stat}]
First assume that \eqref{approach1statiff} holds. Then we know that each COGARCH process $V^{\vp_i}$ in the
representation $\bar{V}^{(1)}=\sum_{i=1}^\infty p_i V^{\vp_i}$
admits a unique stationary distribution given by the law of
$V_\infty^{\vp_i}=\beta\int_{\bbr_+} \ee^{-X_{t}^{\vp_i}}\,\dd t$ and that by choosing
$V^{\vp_i}_0\eqd V^{\vp_i}_\infty$
independently of $S^{\vp_i}$, the corresponding COGARCH process $V^{\vp_i}$ is strictly stationary. Thus setting
$\bar{V}_0^{(1)} := \sum_{i\in\NN} p_i V^{\vp_i}_0$, $\bar{V}^{(1)}$ becomes strictly stationary as shown in the following.

Assume for a moment that $\pi$ has finite support. Then for every
$0\leq t_1< t_2<\ldots <t_n$, $n\in \N$, $h>0$ we can use the independence of $(V^{\vp_i})_{i\in\bbn}$ to obtain
\begin{align*}
 (\bar{V}^{(1)}_{t_1}, \ldots , \bar{V}^{(1)}_{t_n})
=&~ \left(\sum_{i=1}^m p_i V_{t_1}^{\vp_i}, \ldots, \sum_{i=1}^m p_i V_{t_n}^{\vp_i}\right)= \sum_{i=1}^m p_i (V_{t_1}^{\vp_i}, \ldots, V_{t_n}^{\vp_i})\\
\eqd&~  \sum_{i=1}^m p_i (V_{t_1+h}^{\vp_i}, \ldots, V_{t_n+h}^{\vp_i})= (\bar{V}^{(1)}_{t_1+h}, \ldots , \bar{V}^{(1)}_{t_n+h}).
\end{align*}
Due to the fact that $\sum_{i=1}^m p_i V_{t}^{\vp_i}$ is strictly increasing in $m$, the case for $\pi$ having countable support
follows now by a standard monotonicity argument.

Conversely, assume that \eqref{approach1statiff} is violated, i.e. there exists a $\vp_j$ with $\pi(\{\vp_j\})>0$ such that $V^{\vp_j}$ has no stationary distribution.
Then by  \cite[Thm. 3.1]{KLM:2004}  $V_t^{\vp_j}$ converges in probability to $\infty$ as $t\to \infty$.
This yields that also $\bar{V}_t^{(1)}= p_jV_t^{\vp_j}+ \sum_{i=1, i\neq j}^\infty p_i V_{t}^{\vp_i}$ converges in probability
to $\infty$ as $t\to \infty$ since $\sum_{i=1, i\neq j}^\infty p_i V_{t}^{\vp_i}$ is nonnegative. Hence $\bar{V}_t^{(1)}$ cannot be strictly stationary.
\end{proof}

\begin{proof}[Proof of Proposition \ref{propapproach1moments}]
 The moment conditions as well as the formulas for expectation and covariance follow directly from \eqref{approach1statsol} together with
the corresponding results for the COGARCH process \eqref{momentcondcogarch}, \eqref{mean-cog} and \eqref{acf-cog} observing that all appearing processes
are strictly positive.
\end{proof}

\begin{proof}[Proof of Proposition \ref{tail-supCOG1}]
Throughout this proof we slightly change our notation as follows.
Given i.i.d. subordinators $(S^i)_{i\in\bbn}$, we denote the COGARCH process driven by $S^i$ with parameter $\vp>0$ by $V^{i,\vp}$
such that we have $\bar V^{(1)}=\sum_{i=1}^\infty p_i V^{i,\vp_i}$. If $\kappa<\bar\kappa$,
then we know by the definition of $\Psi$ in \eqref{eq-defpsi} and \cite[Lemma 4.1(d)]{KLM:2004} that for every $\vp\in(0,\bar\vp]$ there
exists a unique constant $\kappa(\vp)>0$ which satisfies \eqref{tailcond} with $\bar\kappa$ replaced by $\kappa(\vp)$ and such that $\kappa(\vp)$
is strictly decreasing in $\vp$. Moreover, as shown in \cite[Thm. 6]{klm:2006}, for each $i\in\bbn$ the tail of $V^{i,\vp}$ is asymptotically equivalent to $C(\vp)x^{-\kappa(\vp)}$ with some specific constant $C(\vp)>0$. So by \cite[Lemma A3.26]{EKM} we have
\[ x^\kappa \PP[\bar V^{(1)}_0>x] \leq x^{\kappa-\bar\kappa}x^{\bar\kappa}\PP\left[\sum_{i=1}^\infty p_i V^{i,\bar\vp}_0 > x\right]\to 0 \]
as $x\to\infty$ for all $\kappa<\bar\kappa$. Conversely, if $\kappa> \kappa(\vp_i)$ for some $i\in\bbn$, then
\[ x^\kappa \PP[\bar V^{(1)}_0 > x] \geq x^\kappa \PP[p_i V^{i,\vp_i}_0> x]
= x^{\kappa-\kappa(\vp_i)}p_i^{\kappa(\vp_i)} (x/p_i)^{\kappa(\vp_i)}\PP[V_0^{i,\vp_i}>x/p_i]\to\infty. \]
Recalling that $\kappa(\vp)$ is defined via the equation $\Psi(\kappa(\vp),\vp)=0$, this result is still valid for $\kappa>\bar\kappa$
since we have $\inf_{i\in\bbn} \kappa(\vp_i) =\bar\kappa$ by the implicit function theorem.

Finally, it remains to consider the case $\kappa=\bar\kappa$. If $\pi(\{\bar\vp\})=0$, then using \cite[Lemma 2]{klm:2006} and again
\cite[Lemma A3.26]{EKM}, we obtain
\[ x^{\bar\kappa}\PP[\bar V^{(1)}_0>x]\leq x^{\bar\kappa} \PP\left[\sum_{\vp_i\leq\vp} p_i V^{i,\vp}_0 + \sum_{\vp_i>\vp} p_i V^{i,\bar\vp}_0>x\right]
\sim x^{\bar\kappa} \PP\left[\sum_{\vp_i>\vp} p_i V^{i,\bar\vp}_0>x\right]\to\sum_{\vp_i>\vp} p_i^{\bar\kappa} C(\bar\vp) \]
as $x\to\infty$ for every $\vp\in(0,\bar\vp)$. Letting $\vp\to\bar\vp$, the assertion follows.
The case $\pi(\{\bar\vp\})=:\bar p>0$ now follows from the results above and ($\bar i$ is the index corresponding to $\bar\vp$)
\begin{align*} x^{\bar\kappa}\PP[\bar V^{(1)}_0>x] &= x^{\bar\kappa}\PP\left[\sum_{\vp_i<\bar\vp} p_i V^{i,\vp_i}_0 + \bar p V^{\bar i, \bar\vp}_0 > x\right]\\
&\leq x^{\bar\kappa}\PP[\bar p V^{\bar i, \bar\vp}_0 > x(1-\eps)] + x^{\bar\kappa}\PP\left[\sum_{\vp_i<\vp} p_i V^{i,\vp_i}_0>\eps x\right]
\to \left(\frac{\bar p}{1-\eps}\right)^{\bar\kappa}C(\bar\vp).  \end{align*}
Letting $\eps\to0$, we may set $C:=\bar p^{\bar\kappa}C(\bar\vp)$.
\end{proof}

\subsection{Proofs for Section \ref{s32}}\label{s62}

For the proof of Theorem \ref{thm-app2-stat} we need the following lemma.

\begin{lemma} \label{lem-costationaritaet}
Let $(S_t)_{t\in \RR}$ be a subordinator without drift and define the double-indexed processes
$(X^\vp_t)_{t\in\RR, \vp\in\Phi_L}$ and  $(V^{\vp}_t)_{t\in\RR,\vp\in\Phi_L}$ according to \eqref{eq-defX2sided}
and \eqref{cog-2sided}. Then for all $n\in\NN$, $-\infty<t_1<t_2<\ldots< t_n<\infty$, $h>0$
$$((V_{t_1}^\vp)_{\vp\in\Phi_L},(V_{t_2}^\vp)_{\vp\in\Phi_L}, \ldots, (V_{t_n}^\vp)_{\vp\in\Phi_L})
\eqd ((V_{t_1+h}^\vp)_{\vp\in\Phi_L},(V_{t_2+h}^\vp)_{\vp\in\Phi_L}, \ldots, (V_{t_n+h}^\vp)_{\vp\in\Phi_L}),$$
i.e. the $\bbr^{\Phi_L}$-valued stochastic process $((V_t^{\vp})_{\vp\in\Phi_L})_{t\in\RR}$ is strictly stationary.
In particular, every finite-dimensional process $(V^{\vp_1}_t, \ldots , V^{\vp_m}_t)_{t\in\RR}$, $m\in\NN$, is strictly stationary.
\end{lemma}
\begin{proof}
 Imitating the proof of \cite[Thm. 3.2]{KLM:2004} for the finite-dimensional process
$(V^{\vp_1}_t, \ldots , V^{\vp_m}_t)_{t\in\RR}$, $m\in\NN$, one readily sees that
$$\left((V^{\vp_1}_{t_1}, \ldots , V^{\vp_m}_{t_1}), \ldots, (V^{\vp_1}_{t_n}, \ldots , V^{\vp_m}_{t_n})\right)
\eqd \left((V^{\vp_1}_{t_1+h}, \ldots , V^{\vp_m}_{t_1+h}), \ldots, (V^{\vp_1}_{t_n+h}, \ldots , V^{\vp_m}_{t_n+h})\right).$$
As stochastic processes with the same index space are equal in distribution, whenever their finite-dimensional
distributions are equal (e.g. \cite[Prop. 2.2]{Kallenberg02}), this already yields the assertion.
\end{proof}

\begin{proof}[Proof of Theorem \ref{thm-app2-stat}]
The result follows from the definition of $\bar{V}^{(2)}$ and Lemma~\ref{lem-costationaritaet}.
\end{proof}

To prove Proposition \ref{prop-corrcogarches}, another auxiliary lemma will be established.

\begin{lemma} \label{lem-processX}
 Let $(S_t)_{t\in \RR}$ be a subordinator without drift, let $\vp, \tilde{\vp} \geq 0$ be fixed and define the processes
$(X^\vp_t)_{t\in\RR}$ and  $(X^{\tilde{\vp}}_t)_{t\in\RR}$ according to \eqref{eq-defX2sided}. Set $X_t:=X^\vp_t+X^{\tilde\vp}_t$, $t\in\RR$.
\begin{enumerate}
 \item The process $(X_t)_{t\in\RR}$ is a L\'evy process with characteristic triplet $(2\eta, 0, \nu_X)$ where
$\nu_X=\nu_S\circ T^{-1}$ for $T\colon\RR_+\to\RR_-, y\mapsto -\log(1+(\vp+\tilde\vp)y+\vp\tilde\vp y^2)$.
\item Let $\vp,\tvp>0$. Then $\bbe[\ee^{-\kappa X_t}]$ is finite at $\kappa>0$ for some $t>0$, or, equivalently, for all $t>0$,
if and only if $\bbe[S^{2\kappa}_1]<\infty$. In this case we have $\bbe[\ee^{-\kappa X_t}]=\ee^{t h_\kappa(\vp,\tilde\vp)}$, where
\begin{equation*}  h_{\kappa}(\vp,\tilde\vp)=-2\eta \kappa+\int_{\bbr_+} \big(((1+\vp y)(1+\tilde\vp y))^\kappa-1\big)\,\nu_S(\dd y). \end{equation*}
For $\kappa=1$ we have 
\begin{equation} \label{eq-h-explicit}
 h(\vp,\tilde\vp):=  h_1(\vp,\tilde\vp) = -2\eta +(\vp+\tvp)\E[S_1] + \vp\tvp \var[S_1].
\end{equation}
\end{enumerate}
\end{lemma}

\begin{proof}
(a) Observe that by definition
\[ X_t=2\eta t-\sum_{0<s\leq t}\log\big[(1+\vp\Delta S_s)(1+\tvp\Delta S_s)\big]
= 2\eta t-\sum_{0<s\leq t} \log(1+(\vp+\tvp)\Delta S_s + \vp\tvp(\Delta S_s)^2) \]
for $t\geq0$, which directly yields the assertion.\\
(b) By \cite[Thm. 25.17]{sato} $\bbe[\ee^{-\kappa X_t}]$ is finite for some, or, equivalently, for every
$t>0$ if and only if
$$\int_{|y|>1}\ee^{-\kappa y}\,\nu_X(\dd y)= \int_{|y|>1}\ee^{-\kappa y}\,\nu_S(T^{-1}(\dd y))=
\int_{y\in D^c} (1+(\vp + \tvp)y + \vp\tvp y^2)^\kappa \,\nu_S(\dd y)<\infty $$
where $D=\left[\frac{-(\vp+\tvp)-\sqrt{(\vp-\tvp)^2+4\vp\tvp\ee}}{2\vp\tvp}, \frac{-(\vp+\tvp)+\sqrt{(\vp-\tvp)^2+4\vp\tvp\ee}}{2\vp\tvp}\right]$. This
directly yields (b).
\end{proof}

\bproof[Proof of Proposition \ref{prop-corrcogarches}]
Due to Lemma \ref{lem-costationaritaet} $(V_t^\vp,V_t^\tvp)_{t\in\RR}$ is strictly stationary such that it suffices to consider $t>0$. 
Assume  w.l.o.g. that $0<\vp\leq \tvp$. Then it follows from the
definition of the COGARCH process that $V^\vp\leq V^\tvp$. Hence $\bbe[V^\vp_tV_t^{\tilde\vp}]\leq \bbe[V^\tvp_t V_t^{\tvp}]$ and
similarly $\bbe[V^\vp_t V_{t+h}^{\tvp}]\leq \bbe[V^\tvp_t V_{t+h}^{\tvp}]$, which are both finite as \eqref{momentcondcogarch} is
given for $\kappa=2$.
We start with the computation of $\bbe[V^\vp_tV_t^{\tilde\vp}]$ and use \eqref{cog} to obtain
\begin{align}
\bbe[V^\vp_tV_t^{\tilde\vp}]
&=\beta^2\bbe\left[\int_{(0,t]} \ee^{X_{s}^\vp-X_t^\vp}\,\dd s\int_{(0,t]} \ee^{X_{r}^{\tilde\vp}-X_t^{\tilde\vp}}\,\dd r\right]
+\beta\bbe[V_0^{\tilde\vp}]\bbe\left[\int_{(0,t]} \ee^{X_{s}^\vp-X_t^\vp-X_t^\tvp}\,\dd s\right]+\nonumber\\
&+\beta\bbe[V_0^\vp]\bbe\left[\int_{(0,t]} \ee^{X_{r}^\tvp-X_t^\tvp-X_t^\vp}\,\dd r\right]
+\bbe[V_0^\vp V_0^\tvp]\bbe[\ee^{-X_t^\vp-X_t^\tvp}]\nonumber\\
&=: \beta^2 I_1+\beta \bbe[V_0^\tvp] I_2 + \beta\bbe[V_0^\vp] I_3+\bbe[V_0^\vp V_0^\tvp] I_4.\label{aux2}
\end{align}
Recall the L\'evy process $X$ defined in Lemma~\ref{lem-processX} and observe that the increments of $X$
and $X^\vp$ on disjoint intervals are mutually independent. Thus we have by \eqref{eq-defpsi} and Lemma~\ref{lem-processX}(b)
\begin{align*}
I_1&=\bbe\left[\int_{(0,t]}\int_{(0,r]} \ee^{X_{s}^\vp-X_{r}^\vp+X_{r}^\vp-X_t^\vp+X_{r}^\tvp-X_t^\tvp}\,\dd s\,\dd r\right]
+\bbe\left[\int_{(0,t]}\int_{(r,t]} \ee^{X_{r}^\tvp-X_{s}^\tvp+X_{s}^\tvp-X_t^\tvp+X_{s}^\vp-X_t^\vp}\,\dd s\,\dd r\right]\\
&=\int_{(0,t]}\int_{(0,r]} \ee^{(r-s)\Psi(1,\vp)+(t-r)h(\vp,\tvp)}\,\dd s\,\dd r + \int_{(0,t]}\int_{(r,t]} \ee^{(s-r)\Psi(1,\tvp)+(t-s)h(\vp,\tvp)}\,\dd s\,\dd r\\
&=\frac{-a\ee^{ct}+c\ee^{at}+a-c}{a^2c-ac^2}+\frac{-b\ee^{ct}+c\ee^{bt}+b-c}{b^2c-bc^2},
\end{align*}
where $a:=\Psi(1,\vp)$, $b:=\Psi(1,\tvp)$ and $c:=h(\vp,\tvp)$. Very similar calculations lead to
\[ I_2=\frac{\ee^{bt}-\ee^{ct}}{b-c},\quad I_3=\frac{\ee^{at}-\ee^{ct}}{a-c}\]
while we know from Lemma~\ref{lem-processX}(b) that $I_4= \ee^{ct}$.

According to \eqref{mean-cog} we have $\bbe[V_0^\vp]=-\beta/a$ and $\bbe[V_0^\tvp]=-\beta/b$.
Furthermore, we have $\bbe[V^\vp_0 V^\tvp_0]=\bbe[V^\vp_t V^\tvp_t]$ due to stationarity. Putting all this into \eqref{aux2}, we obtain
\[ (1-\ee^{ct})\bbe[V^\vp_t V^\tvp_t] = \beta^2(1-\ee^{ct})\left(\frac{1}{ac}+\frac{1}{bc}\right). \]
Since $t>0$ we have $1-\ee^{ct}\neq0$, so dividing the last equation by this term yields
$$ \bbe[V^\vp_t V^\tvp_t]=\frac{\beta^2}{\Psi(1,\tvp)h(\vp,\tvp)}+\frac{\beta^2}{\Psi(1,\vp)h(\vp,\tvp)}$$
from which \eqref{eq-crosscor1} and \eqref{eq-crosscor1b} follow immediately.

To obtain the formula for $\cov[V^\vp_t, V_{t+h}^{\tvp}]$ observe first that
\beq\label{eq-splittV} V_{t+h}^\tvp= A_{t,t+h}^\tvp V_t^\tvp + B_{t,t+h}^\tvp, \eeq
where
$$A_{t,t+h}^\tvp= \ee^{-(X^\tvp_{t+h} - X^\tvp_t)} \quad \mbox{and} \quad B_{t,t+h}^\tvp = \beta \int_{(t,t+h]} \ee^{-(X^\tvp_{t+h} - X^\tvp_{s})} \,\dd s.$$
In particular, we see that $A_{t,t+h}^\tvp$ and $B_{t,t+h}^\tvp$ are independent of $(V^\vp_t, V_{t}^{\tvp})$ such that
\begin{equation} \label{eq-crosscoraux}
 \bbe[V^\vp_t V_{t+h}^{\tvp}] = \bbe[V^\vp_t( A_{t,t+h}^\tvp V_t^\tvp + B_{t,t+h}^\tvp)]
= \E[ A_{t,t+h}^\tvp]\bbe[V^\vp_t V_t^\tvp]  + \E[V^\vp_t]\E[B_{t,t+h}^\tvp].
\end{equation}
Now since
$$\E[ A_{t,t+h}^\tvp]= \E[\ee^{-(X^\tvp_{t+h} - X^\tvp_t)} ] = \E[\ee^{-X^\tvp_{h} } ]= \ee^{h \Psi(1,\tvp)}$$
and
\begin{align*}
 \E[B_{t,t+h}^\tvp]
= \beta \int_{(t,t+h]} \ee^{(t+h-s)\Psi(1,\tvp)}\,\dd s = \frac{\beta}{\Psi(1,\tvp)} \left(\ee^{h \Psi(1,\tvp)}-1\right)
= \E[V_0^\tvp]\left(1-\ee^{h \Psi(1,\tvp) }\right),
\end{align*}
Eq.~\eqref{eq-crosscoraux} directly yields
\begin{equation*}
 \bbe[V^\vp_t V^\tvp_{t+h}]
=\ee^{h \Psi(1, \tvp)} \bbe[V^\vp_0 V^\tvp_0] + \left( 1-\ee^{h \Psi(1, \tvp)} \right) \E[V_0^\vp]\E[V_0^\tvp],
\end{equation*}
which gives \eqref{eq-crosscor2}.
\eproof

\begin{proof}[Proof of Proposition \ref{propapproach2moments}]
 Due to the fact that all appearing processes are nonnegative we can use Tonelli's Theorem to determine the given formulas directly from the definition of $\bar{V}^{(2)}$.
\end{proof}

\begin{proof}[Proof of Proposition \ref{prop-tails-sup2}.]
The proof is mainly the same as the proof of Proposition \ref{tail-supCOG1}, so we only indicate the differences.
For $\kappa<\bar\kappa$ use the estimation $\PP[\bar V^{(2)}_0>x]\leq \PP[V^{\bar\vp}_0>x]$.
For $\kappa>\bar\kappa$ and $\pi(\{\bar\vp\})=0$, it suffices to consider $\kappa>\kappa(\vp_i)$ after having chosen
sequences $(\vp_i)_{i\in\bbn}$ and $(\eps_i)_{i\in\bbn}$ with $\pi((\vp_i-\eps_i,\vp_i])>0$ for each $i\in\bbn$.
Using $\PP[\bar V^{(2)}_0>x]\geq \PP\left[\pi((\vp_i-\eps_i,\vp_i])V^{\vp_i}_0>x\right]$ gives the result.
Similarly, use $\PP[\bar V^{(2)}_0>x]\leq \PP\left[\pi((0,\vp])V^\vp_0+\pi((\vp,\bar\vp])V^{\bar\vp}_0>x\right]$ for $\kappa=\bar\kappa$ and $\pi(\{\bar\vp\})=0$.
For $\kappa=\bar\kappa$ and $\pi(\{\bar\vp\})=:\bar p>0$, we may use $\bar V^{2}_0=\int_{(0,\bar\vp)} V^{\vp}_0\,\pi(\dd\vp)+\bar p V^{\bar\vp}_0$. Finally,
the case $\kappa>\bar\kappa$ and $\pi(\{\bar\vp\})>0$ follows from the preceding arguments.
\end{proof}


\subsection{Proofs for Section \ref{s33}}\label{s63}

\bproof[Proof of Proposition \ref{prop-app3solution}]
By \eqref{eq-defX2sided} and \eqref{cog-2sided}, the function $\vp\mapsto V^\vp_s$ is increasing in $\vp$ for every $s\in\bbr$. As a consequence, we have for all $t\geq0$
\[ A_t\leq\int_{(0,t]}\int_{\Phi_L} \vp_\mathrm{max} V^{\vp_\mathrm{max}}_s\,\La^S(\dd s,\dd \vp) = \vp_\mathrm{max} \int_{(0,t]} V^{\vp_\mathrm{max}}_s \,\dd S_s <\infty. \]
Since $A$ is by definition c\`adl\`ag, $\bbg^{(3)}$-adapted and increasing, $A$ is a semimartingale 
\cite[e.g. Def. I.4.21]{JS2} such that
uniqueness of the solution of \eqref{eq-def-approach3} follows from \cite[Thm.~V.7]{Protter04}.
It remains to show that \eqref{supcogarch2-explicit} solves \eqref{eq-def-approach3}. Using integration by parts \cite[Def. I.4.45]{JS2} and \cite[Prop. I.4.49d]{JS2}, we obtain
\begin{align*}
\dd \bar V^{(3)}_t &= \left(\bar V^{(3)}_0+\int_{(0,t]} \ee^{\eta s}\,\dd A_s + \beta \int_{(0,t]} \ee^{\eta s}\,\dd s\right)\,\dd \!\left(\ee^{-\eta t}\right) + \ee^{-\eta t}(\ee^{\eta t}\,\dd A_t+\beta \ee^{\eta t}\,\dd t)\\
&=-\eta \bar V^{(3)}_t\,\dd t + \dd A_t + \beta\,\dd t = (\beta-\eta \bar V^{(3)}_t)\,\dd t + \dd A_t.
\end{align*}\\[-1cm]
\eproof

In order to show that the supCOGARCH 3 process $\bar V^{(3)}$ from \eqref{eq-def-approach3} has a stationary solution we need a series of lemmata.

\ble\label{lem1}
Let $n,m\in\bbn$. For $-\infty<t_1<\ldots<t_{m+1}<\infty$, $0<\vp_1<\ldots<\vp_{n+1}<\vp_\mathrm{max}$ and $h>0$ we have
\begin{align}\label{aux1}
 \lefteqn{(V^{\vp_i}_{t_j},\La^S((t_j,t_{j+1}]\times(\vp_i,\vp_{i+1}])\colon i\leq n, j\leq m)}\\
\eqd&~(V^{\vp_i}_{t_j+h},\La^S((t_j+h,t_{j+1}+h]\times(\vp_i,\vp_{i+1}])\colon i\leq n, j\leq m). \nonumber
\end{align}
\ele
\bproof
For $1\leq i\leq n$ and $1\leq j\leq m$ write $\La^{i}_{j}:=\La^S((t_j,t_{j+1}]\times(\vp_i,\vp_{i+1}])$ and $\La^{i}_{j,h}:=\La^S((t_j+h,t_{j+1}+h]\times(\vp_i,\vp_{i+1}])$
and let $Z^m$ and $Z^m_h$ denote the left- and right-hand side of \eqref{aux1}, respectively.
We first consider $m=1$.
On the one hand, we obtain from Lemma \ref{lem-costationaritaet}
that $(V^{\vp_1}_{t_1},\ldots,V^{\vp_n}_{t_1}) \eqd(V^{\vp_1}_{t_1+h},\ldots,V^{\vp_n}_{t_1+h})$.
On the other hand, due to the independence
of their single components, the vectors $(\La^1_1,\ldots,\La^n_1)$ and $(\La^1_{1,h},\ldots,\La^n_{1,h})$ have the same distribution.
Since additionally the $V$-vector is independent of the $\La^S$-vector, the assertion in the case $m=1$ follows. 
For $m\geq 2$, using induction
and the independence of $\La^i_m$ and $Z^{m-1}$, it suffices to show that the conditional distribution of $(V^{\vp_i}_{t_m}\colon i=1,\ldots,n)$ given
$Z^{m-1}$ does not change if shifted by $h$. By Markovianity (see \cite[Thm. 3.2]{KLM:2004}) this distribution only depends on
$(V^{\vp_i}_{t_{m-1}},\La^i_{m-1}\colon i=1,\ldots,n)$ such that by \eqref{eq-splittV} and using the notation there, we only need to consider
the distribution of $(A^{\vp_i}_{t_{m-1},t_{m}},B^{\vp_i}_{t_{m-1},t_{m}}\colon i=1,\ldots,n)$ given $(\La^i_{m-1}\colon i=1,\ldots,n)$.
Since the former vector is a measurable transformation of $(\Delta S_s\colon t_{m-1}\leq s\leq t_m)$, it is evident that this distribution
is invariant under a shift by $h$, which finishes the proof.
\eproof

In connection to \eqref{processA}, we show a further auxiliary result. To this end define
\begin{equation} \label{processA2sided}
 A_t:=\int_{(0,t]}\int_{\Phi_L} \vp V^\vp_{s-}\,\La^S(\dd s,\dd\vp), \quad t\geq0,\quad A_t:=-\int_{(t,0]} \int_{\Phi_L} \vp V^\vp_{s-}\,\La^S(\dd s,\dd \vp), \quad t<0.
\end{equation}

\ble\label{lem2}
The process $(A_t)_{t\in\bbr}$ defined in \eqref{processA2sided} has stationary increments,
 i.e. for every $n\in\bbn$, $-\infty<t_1<\ldots<t_{n+1}<\infty$ and $h>0$ we have
\beq\label{statincrA}(A_{t_2}-A_{t_1},\ldots,A_{t_{n+1}}-A_{t_n})\eqd(A_{t_2+h}-A_{t_1+h},\ldots,A_{t_{n+1}+h}-A_{t_n+h}).\eeq
\ele

\bproof
By an approximation via Riemann sums (note that $\vp\mapsto V^\vp_s$ is continuous in $\vp$ for all $s$), cf. \cite[Prop. I.4.44]{JS2},
we may use Lemma \ref{lem1} to obtain
\begin{align*}
(A_{t_2}-A_{t_1},\ldots,A_{t_{n+1}}-A_{t_n})&=\left(\int_{t_1}^{t_{2}} \int_{\Phi_L} \vp V^\vp_{s-}\,\La^S(\dd s,\dd\vp),\ldots, \int_{t_n}^{t_{n+1}} \int_{\Phi_L} \vp V^\vp_{s-}\,\La^S(\dd s,\dd\vp)\right)\\
&\eqd\left(\int_{t_1+h}^{t_{2}+h} \int_{\Phi_L} \vp V^\vp_{s-}\,\La^S(\dd s,\dd\vp) ,\ldots,\int_{t_n+h}^{t_{n+1}+h} \int_{\Phi_L} \vp V^\vp_{s-}\,\La^S(\dd s,\dd\vp)\right)\\
&=(A_{t_2+h}-A_{t_1+h},\ldots,A_{t_{n+1}+h}-A_{t_n+h}).
\end{align*}\\[-1cm]
\eproof

\begin{proof}[Proof of Theorem \ref{thm-approach3statsol}]
Since $\ee^{-\eta t}\int_{(0,t]} \ee^{\eta s}\,\dd s \to \eta^{-1}$ as $t\to\infty$ the process $(V_t^{(3)})_{t\geq 0}$ converges
in distribution to a finite random variable as $t\to\infty$ if and only if
\[ \ee^{-\eta t} \int_{(0,t]} \ee^{\eta s}\,\dd A_s = \int_{(0,t]} \ee^{\eta(s-t)}\,\dd A_s = \int_{(-t,0]} \ee^{\eta s}\,\dd A_{s+t} \eqd \int_{(-t,0]} \ee^{\eta s}\,\dd A_s\eqd \int_{(0,t]} \ee^{-\eta s}\,\dd A_s \]
converges to a finite random variable in distribution as $t\to\infty$, where we used Lemma \ref{lem2} for the distributional equalities.
By monotonicity this is equivalent to the existence of
\[ \int_{\bbr_+} \ee^{-\eta s}\,\dd A_s = \int_{\bbr_+} \int_{\Phi_L} \ee^{-\eta s}\vp V_{s-}^\vp\,\La^S(\dd s,\dd\vp) \]
in probability. As shown in \cite[Thm. 3.1]{Chong13} and the following remark, this holds if and only if \eqref{statcond} is valid.

Hence in case that \eqref{statcond} is violated, no stationary distribution can exist.
On the other hand, given \eqref{statcond}, following the above computations, the process $(V_t^{(3)})_{t\geq 0}$ converges in distribution
to $\bar{V}^{(3)}_\infty:= \frac{\beta}{\eta} + \int_{0}^\infty \ee^{-\eta s}\,\dd A_s$, which is thus the unique possible stationary distribution.

To show that $(V_t^{(3)})_{t\geq 0}$ is actually strictly stationary when started in a random variable $V_0^{(3)}\eqd\bar{V}^{(3)}_\infty$ which is independent of $\La^L$ on $\bbr_+\times\Phi_L$, we set $\bar V^{(3)}_0:= \frac{\beta}{\eta}+\int_{(-\infty,0]} \ee^{\eta s}\,\dd A_s$.
Then using Lemma~\ref{lem2} we obtain for all  $0\leq t_1<\ldots<t_n$ and $h>0$
\begin{align*} &~(\bar V^{(3)}_{t_1},\ldots,\bar V^{(3)}_{t_n})\\
\eqd &~\left(\int_{(-\infty,t_1]} \ee^{-\eta(t_1-s)}\,\dd A_s+\beta\int_{(-\infty,t_1]} \ee^{-\eta(t_1-s)}\,\dd s,\ldots,\int_{(-\infty,t_n]} \ee^{-\eta(t_n-s)}\,\dd A_s+\beta\int_{(-\infty,t_n]} \ee^{-\eta(t_n-s)}\,\dd s\right)\\
= &~\left(\int_{(-\infty,0]} \ee^{\eta s}\,\dd A_{s+t_1}+\beta\int_{\bbr_+} \ee^{-\eta s}\,\dd s,\ldots,\int_{(-\infty,0]} \ee^{\eta s}\,\dd A_{s+t_n}+\beta\int_{\bbr_+} \ee^{-\eta s}\,\dd s\right)\\
\eqd &~\left(\int_{(-\infty,0]} \ee^{\eta s}\,\dd A_{s+t_1+h}+\beta\int_{\bbr_+} \ee^{-\eta s}\,\dd s,\ldots,\int_{(-\infty,0]} \ee^{\eta s}\,\dd A_{s+t_n+h}+\beta\int_{\bbr_+} \ee^{-\eta s}\,\dd s\right)\\
\eqd &~(\bar V^{(3)}_{t_1+h},\ldots,\bar V^{(3)}_{t_n+h})
\end{align*}
and hence the process $(V_t^{(3)})_{t\geq 0}$ is strictly stationary.

It remains to show that (a) and (b) imply \eqref{statcond}.
First observe from \eqref{eq-defX2sided} and \eqref{cog-2sided} that for fix $s$ the function $\vp \mapsto V^\vp_s$ is increasing in $\vp$. 
So if (a) holds, we have
\[\int_{\bbr_+} \int_{\Phi_L} \int_{\bbr_+} 1\wedge (y \vp V_s^\vp \ee^{-\eta s}) \,\dd s\,\pi(\dd\vp)\,\nu_S(\dd y) \leq \int_{\bbr_+} \int_{\bbr_+} 1\wedge (y \vp_0 V_s^{\vp_0} \ee^{-\eta s}) \,\dd s\,\nu_S(\dd y)<\infty\]
because \eqref{statcond} holds for $\pi=\delta_{\vp_0}$ (in this case $\bar V^{(3)}$ is just the COGARCH process $V^{\vp_0}$).

Finally,  (b) follows from (a) together with the fact that $\vp_\mathrm{max}^{(\kappa)}< \vp_\mathrm{max}$. 
\end{proof}

For the proof of Proposition \ref{secord-supcog3} we need the following Lemma.

\ble\label{quadvar}
Let $(A_t)_{t\in\bbr}$, $V^\vp$ and $\bar{V}^{(3)}$ be defined as in \eqref{processA2sided}, \eqref{cog-2sided} and \eqref{supcogarch3-explicitstationary}, respectively. Then, under the assumptions of Proposition~\ref{secord-supcog3}, we have for $t\geq0$
\begin{align*} [A,A]_t&=\int_{(0,t]}\int_{\Phi_L}\int_{\bbr_+} \vp^2 (V^\vp_{s-})^2 y^2\,\mu^{\La^S}(\dd s,\dd\vp,\dd y)\quad \mbox{and}\\
[\bar{V}^{(3)},V^\vp]_t &=[A,V^\vp]_t =\vp\int_{(0,t]}\int_{\Phi_L}\int_{\bbr_+} \tvp V^\vp_{s-} V^\tvp_{s-} y^2\,\mu^{\La^S}(\dd s,\dd\tvp,\dd y),
\end{align*}
with $\mu^{\La^S}$ as defined in \eqref{jm}. For $t<0$, let the expressions on the left-hand side denote the respective quadratic (co-)variation on $(t,0]$. Then the integrals have to be computed on $(-t,0]$ instead of $(0,t]$.
\ele
  
\bproof
Obviously it suffices to consider $t\geq0$. Since $A$ is an increasing pure-jump process, 
\[
[A,A]_t=\sum_{0<s\leq t} (\Delta A_s)^2 = \sum_{0<s\leq t} \big(\Delta (\vp V^\vp_{\cdot -}y\ast\mu^{\La^S})_s\big)^2=\sum_{0<s\leq t}\left(\sum_{\vp\in\Phi_L} \vp V^\vp_{s-}\La^S(\{s\}\times\{\vp\})\right)^2
\]
Noting that for almost every $\om$ there is at most one $\vp\in\Phi_L$ at time $s$ with $\La^S(\{s\}\times\{\vp\})(\om)\neq0$, we obtain
\[ [A,A]_t=\sum_{0<s\leq t}\sum_{\vp\in\Phi_L} \vp^2 (V^\vp_{s-})^2\La^S(\{s\}\times\{\vp\})^2, \]
as desired. Similarly,
\[ [A,V^\vp]_t = \sum_{0<s\leq t} \Delta A_s \Delta V^\vp_s= \sum_{0<s\leq t}\left(\sum_{\tvp\in\Phi_L} \tvp V^\tvp_{s-}\La^S(\{s\}\times\{\tvp\})\right) \vp V^\vp_{s-}\Delta S_s \]
according to \eqref{cog-sde}.
Now observe that for all $s\in\bbr$, $\Delta S_s=\La^S(\{s\}\times\bbr_+)=\sum_{\vp\in\Phi_L} \La^S(\{s\}\times\{\vp\})$ where again for almost every $\om$ there is at most one $\vp\in\Phi_L$ at time $s$ with $\La^S(\{s\}\times\{\vp\})(\om)\neq0$.
As a result,
\[ [A,V^\vp]_t=\sum_{0<s\leq t} \vp V^\vp_{s-} \sum_{\tvp\in\Phi_L} \tvp V^\tvp_{s-}\La^S(\{s\}\times\{\tvp\})^2 = \vp(\tvp V^\vp_{\cdot -}V^\tvp_{\cdot -} y^2\ast\mu^{\La^S}_t).\]
Finally, we have $[\bar V^{(3)}, V^\vp]=[A,V^\vp]$ by \eqref{eq-def-approach3}.
\eproof

\begin{proof}[Proof of Proposition \ref{secord-supcog3}]
 First observe that Theorem \ref{thm-approach3statsol}(c) ensures the existence of the given stationary version of
$\bar V^{(3)}$ under the assumptions of the present theorem. \\
We set $m_1:=\int_{\bbr_+} y\,\nu_S(\dd y)=\bbe[S_1]$, $m_2:=\int_{\bbr_+} y^2\,\nu_S(\dd y)=\var[S_1]$ and assume w.l.o.g. $\pi(\{0\})=0$.
For the mean we use \eqref{mean-cog} and obtain
\begin{align*}
\bbe[\bar V_t^{(3)}]
&= \bbe[\bar V_0^{(3)}] = \frac{\beta}{\eta}+\bbe\left[\int_{(-\infty,0]} \ee^{\eta s}\,\dd A_s\right]
= \frac{\beta}{\eta}+m_1\int_{(-\infty,0]} \ee^{\eta s}\,\dd s \int_{\Phi_L} \vp \bbe[V^\vp_0]  \, \pi(\dd\vp)\\
&= \frac{\beta}{\eta}-\frac{\beta}{\eta}\int_{\Phi_L} \left(1+\frac{\eta}{m_1\vp-\eta}\right)\, \pi(\dd\vp)
= -\int_{\Phi_L} \frac{\beta}{m_1\vp-\eta}\,\pi(\dd\vp) = \int_{\Phi_L} \bbe[V^\vp_0] \,\pi(\dd\vp).
\end{align*}
To compute the autocovariance function of $\bar V^{(3)}$ observe that for $t\geq0$, $h\geq0$
we have from \eqref{supcogarch3-explicitstationary}
\begin{align}
\cov[\bar V^{(3)}_t,\bar V^{(3)}_{t+h}]
=&\,\ee^{-2\eta t}\ee^{-\eta h}\bbe\left[\int_{(-\infty,t]} \ee^{\eta s}\,\dd A_s \int_{(-\infty,t+h]} \ee^{\eta s}\,\dd A_s \right] \nonumber\\
&\quad -\bbe\left[\int_{(-\infty,t]} \ee^{-\eta(t-s)}\,\dd A_s\right]\bbe\left[\int_{(-\infty,t+h]} \ee^{-\eta(t+h-s)}\,\dd A_s\right]\nonumber\\
=&\,\ee^{-2\eta t}\ee^{-\eta h}\left(\bbe\left[\left(\int_{(-\infty,t]} \ee^{\eta s}\,\dd A_s\right)^2\right]
+\bbe\left[\int_{(-\infty,t]} \ee^{\eta s}\,\dd A_s \int_{t}^{t+h} \ee^{\eta s}\,\dd A_s \right]\right) \nonumber \\
&\quad -\frac{m_1^2}{\eta^2}\left(\int_{\Phi_L} \vp\bbe[V^\vp_0]\,\pi(\dd\vp)\right)^2\nonumber\\
=:&\,\ee^{-2\eta t}\ee^{-\eta h}(E_1+E_2)-\frac{m_1^2}{\eta^2}\left(\int_{\Phi_L} \vp\bbe[V^\vp_0]\,\pi(\dd\vp)\right)^2.\label{cov-aux}
\end{align}
For $E_1$ we can use integration by parts (see \cite[Eq. I.4.45]{JS2}) together with \cite[Thms. II.19 and VI.29]{Protter04} and Lemma \ref{quadvar} to obtain
\begin{align} E_1 &= 2 \bbe\left[\int_{(-\infty,t]} \left(\int_{(-\infty,s)} \ee^{\eta r} \,\dd A_r\right) \ee^{\eta s}\,\dd A_s\right]
+ \bbe\left[\int_{(-\infty,t]} \ee^{2\eta s}\,\dd [A,A]_s\right]\nonumber\\
&= 2m_1\int_{\Phi_L} \int_{(-\infty,t]} \bbe\left[\left(\int_{(-\infty,s]} \ee^{\eta r}\,\dd A_r\right)V^\vp_s\right]\ee^{\eta s}\vp \,\dd s\,\pi(\dd\vp) \nonumber \\
&\quad +m_2\int_{\Phi_L} \int_{(-\infty,t]} \ee^{2\eta s}\vp^2\bbe[(V^\vp_s)^2]\,\dd s\,\pi(\dd\vp) \nonumber\\
&= 2m_1\int_{\Phi_L} \int_{(-\infty,t]} g(s,\vp) \ee^{\eta s}\vp \,\dd s\,\pi(\dd\vp)
+\frac{m_2}{2\eta}\ee^{2\eta t} \int_{\Phi_L} \vp^2\bbe[(V^\vp_0)^2]\,\pi(\dd\vp),\label{E1}
\end{align}
where $g(s,\vp):=\bbe\left[V^\vp_s\int_{(-\infty,s]} \ee^{\eta r}\,\dd A_r\right]$.
Then again using integration by parts, Lemmas \ref{lem1} and \ref{quadvar} and Eqs.~\eqref{eq-psi-explicit}, \eqref{cog-sde}, \eqref{expect} and \eqref{mean-cog},
we obtain
\begin{align*}
g(s,\vp)
=&~\bbe\left[\int_{(-\infty,s]}\int_{(-\infty,r)} \ee^{\eta u}\,\dd A_u\,\dd V^\vp_r\right]
+\bbe\left[\int_{(-\infty,s]} V_{r-}^\vp\ee^{\eta r}\,\dd A_r\right] + \bbe\left[\int_{(-\infty,s]} \ee^{\eta r}\,\dd [A,V^\vp]_r\right]\\
=&~\bbe\left[\int_{(-\infty,s]} \left(\int_{(-\infty,r]} \ee^{\eta u}\,\dd A_u\right) (\beta-\eta V_r^\vp)\,\dd r\right]
+\bbe\left[\int_{(-\infty,s]}\left(\int_{(-\infty,r)} \ee^{\eta u}\,\dd A_u\right)\vp V^\vp_{r-}\,\dd S_r\right]\\
&\quad +\bbe\left[\int_{(-\infty,s]} V_{r-}^\vp\ee^{\eta r}\,\dd A_r\right]
+ \bbe\left[\int_{(-\infty,s]} \ee^{\eta r}\,\dd [A,V^\vp]_r\right]\\
=&~\beta m_1 \int_{(-\infty,s]} \int_{(-\infty,r]} \ee^{\eta u} \,\dd u\,\dd r \int_{\Phi_L} \tvp\bbe[V^\tvp_0]\,\pi(\dd\tvp)
+(m_1\vp-\eta)\int_{(-\infty,s]} g(r,\vp)\,\dd r \\
&\quad +m_1\int_{(-\infty,s]} \ee^{\eta r}\,\dd r \int_{\Phi_L} \tilde\vp \bbe[V^\vp_0 V^\tvp_0]\,\pi(\dd\tvp)
+m_2\vp\int_{(-\infty,s]}\ee^{\eta r}\,\dd r \int_{\Phi_L} \tvp\bbe[V^\vp_0 V^\tvp_0]\,\pi(\dd\tvp)\\
=&~\frac{\ee^{\eta s}}{\eta}\left(\frac{m_1\beta}{\eta} \int_{\Phi_L} \tvp \bbe[V^\tvp_0]\,\pi(\dd\tvp)
+ (m_1+m_2\vp)\int_{\Phi_L} \tvp \bbe[V^\vp_0 V^\tvp_0]\,\pi(\dd\tvp)\right) \\
&\quad + \Psi(1,\vp)\int_{(-\infty,s]} g(r,\vp)\,\dd r\\
=&~\frac{\ee^{\eta s}}{\eta} C(\vp)+\Psi(1,\vp)\int_{(-\infty,s]} g(r,\vp)\,\dd r,
\end{align*}
with
\[ C(\vp) := \int_{\Phi_L} C(\vp,\tvp)\,\pi(\dd\tvp),\quad C(\vp,\tvp)
:= -\frac{m_1}{\eta}\Psi(1,\vp)\tvp\bbe[V_0^\vp]\bbe[V_0^\tvp]+(m_1+m_2\vp)\tvp\bbe[V_0^\vp V_0^\tvp]. \]
Solving this integral equation yields $g(s,\vp)=\frac{C(\vp)\ee^{\eta s}}{\eta-\Psi(1,\vp)}$. Inserting this result in \eqref{E1} gives
\[ E_1 = \frac{m_1}{\eta} \ee^{2\eta t} \int_{\Phi_L} \frac{\vp C(\vp)}{\eta-\Psi(1,\vp)} \,\pi(\dd\vp)
+\frac{m_2}{2\eta}\ee^{2\eta t} \int_{\Phi_L} \vp^2\bbe[(V^\vp_0)^2]\,\pi(\dd\vp). \]
Let us turn to $E_2$ and denote the augmented natural filtration of $\La^L$ by $\bbg^{(3)}=(\calg_t^{(3)})_{t\in\bbr}$. Now taking conditional expectation w.r.t. $\cG^{(3)}_t$ and observing that $V^\vp$, $\bar V^{(3)}$ as well as $A$ are all adapted to $\GG^{(3)}$, we obtain
\[E_2=\bbe\left[\left(\int_{(-\infty,t]} \ee^{\eta s}\,\dd A_s\right) \bbe\left[\left.\int_{t}^{t+h}\int_{\Phi_L}
\ee^{\eta s}\vp V_{s-}^\vp\, \La^S(\dd s,\dd \vp)\right| \cG^{(3)}_t\right] \right].\]
Observing that the restriction of $\La^S$ on $(t,t+h]$ is independent of $\calf_t$, we have
\[E_2=\bbe\left[\left(\int_{(-\infty,t]} \ee^{\eta s}\,\dd A_s\right) m_1\int_{\Phi_L}\int_{(t,t+h]}
 \ee^{\eta s}\vp \bbe[V_{s-}^\vp|\cG^{(3)}_t]\,\dd s\,\pi(\dd \vp)\right].\]
According to \cite[Eq. (4.5)]{KLM:2004} we have
$\bbe[V_{s-}^\vp|\cG^{(3)}_t] = (V^\vp_t - \bbe[V^\vp_0])\ee^{(s-t)\Psi(1,\vp)}+\bbe[V^\vp_0]$ for $s>t$. So we get
\begin{align*}
E_2
=&~ m_1\bbe\left[\left(\int_{(-\infty,t]}\ee^{\eta s}\,\dd A_s\right)
\int_{\Phi_L}\int_{(t,t+h]} \ee^{\eta s} \vp\big((V^\vp_t-\bbe[V^\vp_0])\ee^{(s-t)\Psi(1,\vp)}+\bbe[V^\vp_0]\big)\,\dd s\,\pi(\dd\vp)\right]\\
=&~ m_1\int_{\Phi_L} \vp \bbe\left[V_t^\vp\int_{(-\infty,t]} \ee^{\eta s}\,\dd A_s\right] \int_{(t,t+h]}
\ee^{\eta s}\ee^{(s-t)\Psi(1,\vp)}\,\dd s\,\pi(\dd\vp)\\
&\quad +m_1\bbe\left[\int_{(-\infty,t]} \ee^{\eta s}\,\dd A_s\right] \int_{\Phi_L}
\vp\bbe[V^\vp_0] \int_{(t,t+h]} \ee^{\eta s}(1-\ee^{(s-t)\Psi(1,\vp)})\,\dd s\,\pi(\dd\vp)\\
=&~ \int_{\Phi_L} g(t,\vp) \ee^{\eta t}(\ee^{m_1\vp h}-1)\,\pi(\dd\vp)\\
&\quad +m_1^2\int_{(-\infty,t]}\ee^{\eta s}\,\dd s\int_{\Phi_L} \vp \bbe[V_0^\vp]\,\pi(\dd\vp)
\int_{\Phi_L} \vp \bbe[V_0^\vp] \ee^{\eta t}\left(\frac{\ee^{\eta h}-1}{\eta}-\frac{\ee^{m_1\vp h}-1}{m_1\vp}\right)\,\pi(\dd\vp)\\
=&~\ee^{2\eta t} \left( \int_{\Phi_L} \frac{C(\vp)}{\eta-\Psi(1,\vp)}
(\ee^{m_1\vp h}-1)\,\pi(\dd\vp)+\frac{m_1^2}{\eta^2} (\ee^{\eta h}-1) \left(\int_{\Phi_L} \vp \bbe[V_0^\vp]\,\pi(\dd\vp)\right)^2 \right.\\
&\quad \left.-\frac{m_1}{\eta}\int_{\Phi_L} \vp\bbe[V^\vp_0]\,\pi(\dd\vp)
\int_{\Phi_L} \bbe[V_0^\vp](\ee^{m_1\vp h}-1)\,\pi(\dd\vp) \right).
\end{align*}

Now inserting the results for $E_1$ and $E_2$ in \eqref{cov-aux}, we obtain
\begin{align}
\cov[\bar V_t^{(3)},\bar V_{t+h}^{(3)}]
=&~\ee^{-\eta h}\left(\frac{m_1}{\eta} \int_{\Phi_L} \frac{\vp C(\vp)}{\eta-\Psi(1,\vp)} \,\pi(\dd\vp)+\frac{m_2}{2\eta} \int_{\Phi_L} \vp^2\bbe[(V^\vp_0)^2]\,\pi(\dd\vp)\right)\nonumber \\
&\quad +\int_{\Phi_L} \frac{C(\vp)}{\eta-\Psi(1,\vp)} (\ee^{\Psi(1,\vp) h}-\ee^{-\eta h})\,\pi(\dd\vp)-\frac{m_1^2}{\eta^2} \ee^{-\eta h} \left(\int_{\Phi_L} \vp \bbe[V_0^\vp]\,\pi(\dd\vp)\right)^2 \nonumber \\
&\quad -\frac{m_1}{\eta}\int_{\Phi_L} \tvp\bbe[V^\tvp_0]\,\pi(\dd\tvp) \int_{\Phi_L} \bbe[V_0^\vp](\ee^{\Psi(1,\vp) h}-\ee^{-\eta h})\,\pi(\dd\vp) \nonumber \\
=&~\int_{\Phi_L} \int_{\Phi_L}  \left(\frac{C(\vp,\tvp)}{\eta-\Psi(1,\vp)} - \frac{m_1}{\eta}\tvp\bbe[V^\tvp_0]\bbe[V_0^\vp]\right)  \ee^{\Psi(1,\vp) h} \,\pi(\dd\vp)\,\pi(\dd\tvp) \nonumber \\
&\quad +  \ee^{-\eta h} \int_{\Phi_L} \int_{\Phi_L} \left(  \frac{m_1\vp C(\vp,\tvp)}{\eta(\eta-\Psi(1,\vp))} + \frac{m_2 \vp^2 \bbe[(V_0^\vp)^2]}{2\eta} - \frac{C(\vp,\tvp)}{\eta-\Psi(1,\vp)}\right. \nonumber \\
&\quad \left. - \frac{m_1^2}{\eta^2}\vp\tvp\bbe[V_0^\vp]\bbe[V_0^\tvp]+\frac{m_1}{\eta}\tvp\bbe[V_0^\vp]\bbe[V_0^\tvp] \right) \,\pi(\dd\vp)\,\pi(\dd\tvp), \label{eq-aux5}
\end{align}
where using Proposition \ref{prop-corrcogarches} together with Eqs.~\eqref{eq-psi-explicit} and \eqref{eq-h-explicit} gives
\begin{align}
 \lefteqn{\frac{C(\vp,\tvp)}{\eta-\Psi(1,\vp)}-\frac{m_1}{\eta}\tvp\bbe[V_0^\vp]\bbe[V_0^\tvp]} \nonumber \\
=&~-\frac{m_1\Psi(1,\vp)\tvp\bbe[V_0^\vp]\bbe[V_0^\tvp]}{\eta(\eta-\Psi(1,\vp))}
+ \frac{(m_1+m_2\vp)\tvp\bbe[V_0^\vp V_0^\tvp]}{\eta-\Psi(1,\vp)} - \frac{m_1}{\eta}\tvp\bbe[V_0^\vp]\bbe[V_0^\tvp]\nonumber \\
=&~ \frac{(m_1+m_2\vp)\tvp}{\eta-\Psi(1,\vp)}\bbe[V_0^\vp V_0^\tvp]-\frac{m_1\tvp}{\eta(\eta-\Psi(1,\vp))}\bbe[V_0^\vp]\bbe[V_0^\tvp](\Psi(1,\vp)+\eta-\Psi(1,\vp))\nonumber \\
=&~ \frac{\eta+\Psi(1,\tvp)+ h(\vp,\tvp)-\Psi(1,\vp)-\Psi(1,\tvp)}{\eta-\Psi(1,\vp)}\bbe[V_0^\vp V_0^\tvp]-\frac{\eta+\Psi(1,\tvp)}{\eta-\Psi(1,\vp)}\bbe[V_0^\vp]\bbe[V_0^\tvp]\nonumber \\
=&~ \left( 1+ \frac{h(\vp,\tvp)}{\eta-\Psi(1,\vp)}\right)\bbe[V_0^\vp V_0^\tvp] -  \left( 1+ \frac{\Psi(1,\vp)+\Psi(1,\tvp)}{\eta-\Psi(1,\vp)}\right)\bbe[V_0^\vp]\bbe[V_0^\tvp]\nonumber \\
=&~ \cov[V_0^\vp, V_0^\tvp],\label{eq-aux4}
\end{align}
while for the second part of \eqref{eq-aux5} by Eqs.~\eqref{eq-psi-explicit}, \eqref{mean-cog} and \eqref{acf-cog}
\begin{align}
&~\frac{m_1\vp C(\vp,\tvp)}{\eta(\eta-\Psi(1,\vp))} + \frac{m_2 \vp^2 \bbe[(V_0^\vp)^2]}{2\eta} - \frac{C(\vp,\tvp)}{\eta-\Psi(1,\vp)}- \frac{m_1^2}{\eta^2}\vp\tvp\bbe[V_0^\vp]\bbe[V_0^\tvp]+\frac{m_1}{\eta}\tvp\bbe[V_0^\vp]\bbe[V_0^\tvp] \nonumber \\
=&~\frac{\Psi(1,\vp)C(\vp,\tvp)}{\eta(\eta-\Psi(1,\vp))} + \frac{\big(\Psi(2,\vp)-2\Psi(1,\vp)\big)}{2\eta}\bbe[(V_0^\vp)^2] - \frac{m_1\tvp\Psi(1,\vp)}{\eta^2}\bbe[V_0^\vp]\bbe[V_0^\tvp]\nonumber \\
=&\!:~ F_1+ F_2 + F_3. \label{eq-aux3}
\end{align}
Now observe that by \eqref{squaremean-cog} and \eqref{mean-cog}
\begin{equation*}
 F_2= \frac{\Psi(2,\vp)}{2\eta}\bbe[(V_0^\vp)^2]-\frac{\Psi(1,\vp)}{\eta}\bbe[(V_0^\vp)^2] = -\frac{\beta}{\eta}\bbe[V_0^\vp] + \frac{\beta}{\eta}\frac{\bbe[(V_0^\vp)^2]}{\bbe[V_0^\vp]},
\end{equation*}
while
\begin{align*}
 F_3 =&~ - \frac{(\Psi(1,\tvp)+\eta)\Psi(1,\vp)}{\eta^2}\bbe[V_0^\vp]\bbe[V_0^\tvp]= -\frac{\beta^2}{\eta^2} + \frac{\Psi(1,\vp)}{\eta} \cov[V_0^\vp, V_0^\tvp] - \frac{\Psi(1,\vp)}{\eta}\bbe[V_0^\vp V_0^\tvp]\\
=&~ -\frac{\beta^2}{\eta^2} - \frac{\beta}{\eta} \frac{\cov[V_0^\vp, V_0^\tvp]}{\bbe[V_0^\vp]} + \frac{\beta}{\eta}\frac{\bbe[V_0^\vp V_0^\tvp]}{\bbe[V_0^\vp]}.
\end{align*}
On the other hand we obtain by similar means
\begin{align*}
 \Psi(1,\vp)C(\vp,\tvp)
=&~ \frac{(m_1+m_2\vp)\tvp\beta^2 (\Psi(1,\vp)+ \Psi(1,\tvp))}{h(\vp,\tvp)\Psi(1,\tvp)} -\frac{m_1\beta^2\tvp\Psi(1,\vp)}{\eta\Psi(1,\tvp)} \\
=&~ \frac{\beta^2}{\eta\Psi(1,\tvp)h(\vp,\tvp)}\left( \eta(m_1+m_2\vp)\tvp (\Psi(1,\vp)+ \Psi(1,\tvp)) -m_1 \tvp\Psi(1,\vp)h(\vp,\tvp) \right)\\
=&~ \frac{\beta^2(\eta-\Psi(1,\vp))}{\eta\Psi(1,\tvp)h(\vp,\tvp)}\big(h(\vp,\tvp)\Psi(1,\tvp) + \eta(\Psi(1,\vp)+\Psi(1,\tvp))\big)
\end{align*}
such that by Proposition \ref{prop-corrcogarches}
$$F_1= \frac{\beta^2}{\eta^2} - \frac{\beta}{\eta}\frac{\bbe[V_0^\vp V_0^\tvp]}{\bbe[V_0^\vp]}.$$
Finally inserting \eqref{eq-aux4} and \eqref{eq-aux3} with the obtained formulas for $F_1$, $F_2$ and $F_3$ in \eqref{eq-aux5} gives
\begin{align*}
\cov[\bar V_t^{(3)},\bar V_{t+h}^{(3)}]
=&~ \int_{\Phi_L} \int_{\Phi_L}  \left( \cov[V_0^\vp, V_0^\tvp] \ee^{\Psi(1,\vp) h} \right. \\
&\quad +  \left. \ee^{-\eta h}  \left(-\frac{\beta}{\eta}\bbe[V_0^\vp] + \frac{\beta}{\eta}\frac{\bbe[(V_0^\vp)^2]}{\bbe[V_0^\vp]}
- \frac{\beta}{\eta} \frac{\cov[V_0^\vp, V_0^\tvp]}{\bbe[V_0^\vp]}  \right) \right) \,\pi(\dd\vp)\,\pi(\dd\tvp),
\end{align*}
which yields the result.
\end{proof}

\begin{proof}[Proof of Proposition \ref{tail-supCOG3}]
To show the assertion for $\kappa<\bar\kappa$ we use $\PP[\bar V^{(3)}_0>x]\leq \PP[V^{\bar\vp}_0>x]$ and proceed as in the proof of
Proposition~\ref{tail-supCOG1}. For the other cases, observe that
\[ \bar V^{(3)}_0 \eqd \frac{\beta}{\eta} + \int_{\RR_+} \int_{\Phi_L} \ee^{-\eta t} \vp V^\vp_{t-}\,\La^S(\dd t,\dd\vp) = \sum_{i=1}^\infty \ee^{-\eta T_i} \vp_i V^{\vp_i}_{T_i-}\Delta S_{T_i}, \]
where $(T_i)_{i\in\bbn}$ are the jump times of $S$ and $(\vp_i)_{i\in\bbn}$ is an i.i.d. sequence with common distribution
$\pi$ which is also independent of $S$. We start by proving that, if $I$ is a measurable subset of $\Phi_L$ with $\pi(I)=:p>0$ and $\vp\in\Phi_L$,
 then there are constants $0<C_\ast(\vp,p),C^\ast(\vp,p)<\infty$, only dependent on $I$ via $p$, with
\begin{align}\label{tail-aux-eq0} C_\ast(\vp,p)
&=\liminf_{x\to\infty} x^{\kappa(\vp)}\PP\big[\sum_{\vp_i\in I} \ee^{-\eta T_i}\vp V^\vp_{T_i-}\Delta S_{T_i} > x \big]\\
&\leq\limsup_{x\to\infty} x^{\kappa(\vp)}\PP\big[\sum_{\vp_i\in I} \ee^{-\eta T_i}\vp V^\vp_{T_i-}\Delta S_{T_i} > x \big]
=C^\ast(\vp,p)  \nonumber \end{align}
and moreover, if $p\to0$, then $C_\ast(\vp,p),C^\ast(\vp,p)\to0$.

We abbreviate the sum in \eqref{tail-aux-eq0} by $V(\vp,I)$ or $V(I)$. Since the sequence $(\vp_i)_{i\in\bbn}$ is independent of everything else,
 the distribution of $V(I)$ only depends on $p$, which means that the constants $C_\ast(\vp,p)=:C_\ast(p)$ and $C^\ast(\vp,p)=:C^\ast(p)$
only depend on $p$. Also, they are obviously decreasing in $p$.
 Hence, for the claimed convergence to $0$, it suffices to show $C^\ast(2^{-n})\leq((1+2^{-\kappa(\vp)})/2)^n C(\vp)$ for all $n\in\bbn_0$, where $C(\vp)$
is the tail constant of $V^\vp_0$ as in the proof of Proposition~\ref{tail-supCOG1}. The case $n=0$ corresponds to $V(I)\eqd V^\vp_0$ and the statement is clear.
For $n\geq1$, find a set $I^\prime$ disjoint with $I$ such that $\pi(I^\prime)=\pi(I)=2^{-n}$ and therefore $\pi(J)=2^{-(n-1)}$ for $J=I\cup I^\prime$. Since
\begin{align*} \PP[V(J)>x]
 &=\PP[V(I)+V(I^\prime) >x]\geq 2\PP[V(I)>x]-\PP[V(I)>x,V(I^\prime)>x]\\
&\geq2\PP[V(I)>x]-\PP[V(J)>2x],
\end{align*}
we have by induction
\[ C^\ast(2^{-n})=\limsup_{x\to\infty} x^{\kappa(\vp)}\PP[V(I)>x]\leq \frac{1+2^{-\kappa(\vp)}}{2} C^\ast(2^{-(n-1)}). \]
It remains to show that $C^\ast(p)<\infty$ and $C_\ast(p)>0$ for all $p>0$.
Again by monotonicity, the first inequality is obvious and in the second inequality we only need to consider $p=1/n$. To this end, partition $\Phi_L$ into $n$ disjoint sets $(I_k)_{k=1,\ldots,n}$, each with $\pi(I_k)=1/n$. Then observe that
\[ \PP[V^\vp_0>x] \leq \PP[V(I_1)>x/n\text{ or } \ldots \text{ or } V(I_n)>x/n] \leq n \PP[V(I_1)>x/n], \]
which implies
\[ C_\ast(1/n) = \liminf_{x\to\infty} x^{\kappa(\vp)}\PP[V(I_1)>x] \geq Cn^{-(\kappa(\vp)+1)} > 0. \]

Let us come back to the main line of the proof of Proposition \ref{tail-supCOG3}. If $\vp<\bar\vp$, then we have by the above
\[ \liminf_{x\to\infty} x^\kappa \PP[\bar V^{(3)}_0>x]\geq \liminf_{x\to\infty} x^\kappa \PP\big[V(\vp,[\vp,\bar\vp])>x\big] \to \infty  \]
for all $\kappa>\kappa(\vp)$ and therefore, by the same argument as in the proof of Proposition \ref{tail-supCOG1}, for all $\kappa>\bar\kappa$.

Next, consider the case $\kappa=\bar\kappa$ and $\bar p=0$. Then, again by the above and the proof of \cite[Lemma 2]{klm:2006}
\begin{align*} \limsup_{x\to\infty} x^{\bar\kappa}\PP[\bar V^{(3)}_0>x] &\leq \limsup_{x\to\infty} x^{\bar\kappa} \PP\big[ V(\vp,(0,\vp]) + V(\bar\vp,(\vp,\bar\vp])>x\big]\\
&= \limsup_{x\to\infty} x^{\bar\kappa}\PP\big[V(\bar\vp,(\vp,\bar\vp])>x\big] = C^\ast(\bar\vp,\pi((\vp,\bar\vp])), \end{align*}
which converges to $0$ as $\vp\to\bar\vp$. For the case $\bar p>0$ first decompose
\[ \bar V^{(3)}_0 = \frac{\beta}{\eta} + \sum_{\vp_i\neq\bar\vp} \ee^{-T_i}\vp_i V^{\vp_i}_{T_i-}\Delta S_{T_i} + V(\bar\vp,\{\bar\vp\}) =: \frac{\beta}{\eta}+Z+V(\bar\vp,\{\bar\vp\})\]
and observe that $\limsup_{x\to\infty} x^{\bar\kappa}\PP[Z>x] = 0$ by the results so far.
Reading along the lines of the proof of \cite[Lemma 2]{klm:2006}, we obtain
\begin{align*}
\liminf_{x\to\infty} x^{\bar\kappa}\PP[\bar V^{(3)}_0>x] &= \liminf_{x\to\infty} x^{\bar\kappa}\PP[V(\bar\vp,\{\bar\vp\})>x]=C_\ast(\bar\vp,\bar p),\\
\limsup_{x\to\infty} x^{\bar\kappa}\PP[\bar V^{(3)}_0>x] &= \limsup_{x\to\infty} x^{\bar\kappa}\PP[V(\bar\vp,\{\bar\vp\})>x]=C^\ast(\bar\vp,\bar p),
\end{align*}
which finishes the proof.
\end{proof}

\subsection{Proofs for Section \ref{s5}}\label{s64} 

\begin{proof}[Proof of Theorem \ref{price-supCOG1}]
First observe that the assumption that $\pi$ has support in $\Phi_L^{(\kappa)}$ implies $\E\big[(\bar{V}^{(1)}_t)^\kappa\big]<\infty$. Therefore, $\bbe[L^{\vp_1}_1]=0$ implies
$$\E[\Delta^r G^{(1)}_t]= \E\left[\int_{(t,t+r]} \sqrt{\bar{V}^{(1)}_s} \,\dd L^{\vp_1}_s \right]= 0.$$
Next assume $\bbe[L_1^2]<\infty$. 
Using integration by parts and the fact that $G^{(1)}$ has stationary increments, we have
\begin{align*}
  \E[(\Delta^r G^{(1)}_t)^2]&= \E[(G^{(1)}_r)^2] = 2 \E\left[ \int_{(0,r]} G^{(1)}_{s-}\sqrt{\bar{V}^{(1)}_{s-}} \,\dd L^{\vp_1}_s  \right] + \E\left[\int_{(0,r]} \bar{V}^{(1)}_{s-} \,\dd[L^{\vp_1},L^{\vp_1}]_s  \right]\\
&= 0+ \var[L_1] \E[\bar{V}^{(1)}_0] r,
\end{align*}
which, together with Proposition \ref{propapproach2moments} and the relation between $S$ and $L$ in \eqref{eq-def-S}, gives the stated formula.
Furthermore, for $h\geq r>0$ we have, in view of the above computations and again using integration by parts,
\begin{align*}
 \cov[\Delta^r G^{(1)}_t,\Delta^r G^{(1)}_{t+h}]
&= \E\left[\Delta^r G^{(1)}_t \Delta^r G^{(1)}_{t+h} \right]\\
&= \E\left[ \int_{(0,t+h+r]} \mathds{1}_{(t, t+r]}(s) \sqrt{\bar{V}^{(1)}_{s-}} \,\dd L^{\vp_1}_s \int_{(0,t+h+r]} \mathds{1}_{(t+h, t+h+r]}(u) \sqrt{\bar{V}^{(1)}_{u-}} \,\dd L^{\vp_1}_u  \right]\\
&= \E\left[ \int_{(0,t+h+r]} \mathds{1}_{(t, t+r]}(s)\mathds{1}_{(t+h, t+h+r]}(s)  \bar{V}^{(1)}_{s-} \,\dd [L^{\vp_1},L^{\vp_1}]_s \right]\\
&\quad + \E\left[ \int_{(0,t+h+r]}\left( \int_{(0,u]} \mathds{1}_{(t, t+r]}(s) \sqrt{\bar{V}^{(1)}_{s-}} \,\dd L^{\vp_1}_s\right)\mathds{1}_{(t+h, t+h+r]}(u) \sqrt{\bar{V}^{(1)}_{u-}} \,\dd L^{\vp_1}_u  \right]\\
&\quad + \E\left[ \int_{(0,t+h+r]}\left( \int_{(0,u]} \mathds{1}_{(t+h, t+h+r]}(s) \sqrt{\bar{V}^{(1)}_{s-}} \,\dd L^{\vp_1}_s\right)\mathds{1}_{(t, t+r]}(u) \sqrt{\bar{V}^{(1)}_{u-}} \,\dd L^{\vp_1}_u  \right]\\
&=0.
\end{align*}
To compute the covariance of the squared increments, let $\bbg^{(1)}=(\calg_t^{(1)})_{t\geq0}$ denote the augmented natural filtration of $(L^{\vp_i})_{i\in\bbn}$ and observe that
\begin{align*}
\E \left[(\Delta^r G^{(1)}_0)^2(\Delta^r G^{(1)}_{h})^2  \right]
&= \E \left[\E \left[(\Delta^r G^{(1)}_0)^2(\Delta^r G^{(1)}_{h})^2| \cG^{(1)}_r  \right] \right]
= \E \left[ (\Delta^r G^{(1)}_0)^2 \E \left[(\Delta^r G^{(1)}_{h})^2| \cG^{(1)}_r  \right] \right],
\end{align*}
where again by integration by parts
\begin{align*}
\lefteqn{ \E \left[(\Delta^r G^{(1)}_{h})^2)| \cG^{(1)}_r  \right]
= \E \left[\left. \left( \int_{(h,h+r]} \sqrt{\bar{V}^{(1)}_s} \,\dd L^{\vp_1}_s \right)^2 \right| \cG^{(1)}_r  \right]}\\
&= 2 \E \left[\left. \int_{(h,h+r]} \left(\int_{(0,s]} \sqrt{\bar{V}^{(1)}_{u-}} \,\dd L^{\vp_1}_u \right) \sqrt{\bar{V}^{(1)}_{s-}} \dd L^{\vp_1}_s \right| \cG^{(1)}_r  \right] + \E \left[\left. \int_{(h,h+r]} \bar{V}^{(1)}_{s-} \,\dd [L^{\vp_1},L^{\vp_1}]_s  \right| \cG^{(1)}_r  \right]\\
&= 0 + \E[L_1^2] \int_{(h,h+r]} \E[ \bar{V}^{(1)}_{s-}| \cG^{(1)}_r ] \,\dd s.
\end{align*}
Next, for $s>r$ we obtain, using the notation as in the proof of Proposition~\ref{prop-corrcogarches},
\begin{align*}
\E\left[ \bar{V}^{(1)}_{s} | \cG^{(1)}_r \right]
&= \int_{\Phi_L^{(1)}} \E\left[V_s^\vp|\cG^{(1)}_r  \right]\, \pi(\dd \vp) 
= \int_{\Phi_L^{(1)}} \E\left[(A_{r,s}^\vp V_r^\vp + B_{r,s}^\vp)| \cG^{(1)}_r  \right] \, \pi(\dd \vp)\\
&= \int_{\Phi_L^{(1)}} \left( \E[ A_{r,s}^\vp] V_r^\vp + \E[B_{r,s}^\vp]\right)\, \pi(\dd \vp)\\
&= \int_{\Phi_L^{(1)}} \left(\ee^{(s-r)\Psi(1,\vp)} V_r^\vp + \E[V_0^\vp]\left(1 - \ee^{(s-r)\Psi(1,\vp)} \right) \right) \,\pi(\dd \vp).
\end{align*}
Together with the preceding computations, this yields
\begin{align*}
 \lefteqn{\cov[(\Delta^r G^{(1)}_0)^2,(\Delta^r G^{(1)}_{h})^2]}\\
&= \E \left[ (\Delta^r G^{(1)}_0)^2 \E[L_1^2] \int_{(h,h+r]} \E[ \bar{V}^{(1)}_{s-}| \cG^{(1)}_r ] \,\dd s \right] - \E\left[(\Delta^r G^{(1)}_0)^2\right]\E\left[(\Delta^r G^{(1)}_{h})^2)\right]\\
&= \E[L_1^2]  \E \left[ (\Delta^r G^{(1)}_0)^2 \int_{(h,h+r]} \int_{\Phi_L^{(1)}} \left(\ee^{(s-r)\Psi(1,\vp)} V_r^\vp + \E[V_0^\vp]\left(1 - \ee^{(s-r)\Psi(1,\vp)} \right) \right) \,\pi(\dd \vp) \,\dd s \right]\\
&\quad -\left(\E\left[(\Delta^r G^{(1)}_0)^2\right]\right)^2 \\
&= \E[L_1^2]  \E \left[ (\Delta^r G^{(1)}_0)^2 \int_{\Phi_L^{(1)}} \frac{1}{\Psi(1,\vp)} \left(\ee^{h\Psi(1,\vp)}-\ee^{(h-r)\Psi(1,\vp)}\right)(V_r^\vp - \E[V_0^\vp]) + r \E[V_0^\vp] \,\pi(\dd \vp)\right]\\
&\quad -\left(\E\left[(\Delta^r G^{(1)}_0)^2\right]\right)^2 \\
&= \E[L_1^2]  \int_{\Phi_L^{(1)}} \frac{1}{\Psi(1,\vp)} \left(\ee^{h\Psi(1,\vp)}-\ee^{(h-r)\Psi(1,\vp)}\right)\left(\E[(\Delta^r G^{(1)}_0)^2 V_r^\vp] - \E[(\Delta^r G^{(1)}_0)^2] \E[V_r^\vp]\right)\,\pi(\dd \vp)\\
&\quad +  \E[(\Delta^r G^{(1)}_0)^2] r \E[L_1^2]  \int_{\Phi_L^{(1)}}  \E[V_0^\vp] \,\pi(\dd \vp) - \left(\E\left[(\Delta^r G^{(1)}_0)^2\right]\right)^2\\
&= \E[L_1^2]  \int_{\Phi_L^{(1)}} \frac{1}{\Psi(1,\vp)} \left(\ee^{h\Psi(1,\vp)}-\ee^{(h-r)\Psi(1,\vp)}\right)\cov[(\Delta^r G^{(1)}_0)^2, V_r^\vp] \,\pi(\dd \vp).
\end{align*}
It remains to prove $\cov[(\Delta^r G^{(1)}_0)^2, V_r^\vp]\geq0$ with strict inequality if $\pi(\{\vp\})>0$ in order to
obtain the claimed positivity of the covariance of the squared increments. Again using integration by parts, we get
\[ (\Delta^r G^{(1)}_0)^2 = \left(\int_{(0,r]} \sqrt{\bar V^{(1)}_{s-}}\,\dd L^{\vp_1}_s\right)^2 = 2M_r + \int_{(0,r]} \bar V^{(1)}_{s-}\,\dd [L^{\vp_1},L^{\vp_1}]_s, \]
where
\[ M_r:=\int_{(0,r]} \sqrt{\bar V^{(1)}_{s-}} \left(\int_{(0,s)} \sqrt{\bar V^{(1)}_{u-}} \,\dd L^{\vp_1}_u\right)\,\dd L^{\vp_1}_s  \]
satisfies $\bbe[M_r]=0$ due to $\bbe[L_1]=0$ and
\begin{align}
\bbe[M_r V_r^\vp] &= \bbe\left[ \int_{(0,r]} M_s(\beta-\eta V^\vp_s)\,\dd s \right] + \bbe\left[ \int_{(0,r]} M_{s-} \vp V^\vp_{s-}\,\dd S^{\vp}_s\right] + \bbe\left[ \int_{(0,r]} V^\vp_{s-}\,\dd M_s \right]\nonumber\\ &\quad + \bbe\big[ [V^\vp,M]_r \big] \nonumber\\
&=\Psi(1,\vp)\int_{(0,r]} \bbe[M_s V^\vp_s]\,\dd s + \bbe\big[ [V^\vp,M]_r\big]. \label{price-supCOG1-aux}
\end{align}
Applying $\int_\bbr y^3\,\nu_L(\dd y)=0$ and the independence of $L^\vp$ and $L^{\vp_1}$, if $\vp\neq\vp_1$, we have
\begin{align}\label{eq-price-cases} \bbe\big[[V^{\vp}, M]_r\big] &= \vp\bbe\left[\int_{(0,r]} V^\vp_{s-} \sqrt{\bar V^{(1)}_{s-}} \left(\int_{(0,s)} \sqrt{\bar V^{(1)}_{u-}} \,\dd L^{\vp_1}_u\right)\,\dd [L^{\vp_1},S^\vp]_s\right]\\
&=\begin{cases} 0 &\text{if } \vp\neq\vp_1,\\ \vp\displaystyle\int_\bbr y^3\,\nu_L(\dd y) \displaystyle\int_{(0,r]} \bbe\left[V^\vp_{s-} \sqrt{\bar V^{(1)}_{s-}} \left(\displaystyle\int_{(0,s)} \sqrt{\bar V^{(1)}_{u-}}\,\dd L^{\vp_1}_u\right) \right]\,\dd s = 0 &\text{if } \vp=\vp_1. \end{cases} \nonumber
\end{align}
Therefore, \eqref{price-supCOG1-aux} together with the fact that $\bbe[M_0 V_0^\vp]=0$ implies that $\bbe[M_r V_r^\vp]=0$ for all $r\geq0$.
As a consequence, we have
\begin{align*} \cov[(\Delta^r G^{(1)}_0)^2,V^\vp_r] &= \cov\left[2M_r+\int_{(0,r]} \bar V^{(1)}_{s-}\,\dd [L^{\vp_1},L^{\vp_1}]_s, V^\vp_r\right]\\
&=\bbe\left[ V^\vp_r \int_{(0,r]} \bar V^{(1)}_{s-}\,\dd [L^{\vp_1},L^{\vp_1}]_s\right]-\bbe[V^\vp_1]\bbe\left[\int_{(0,r]} \bar V^{(1)}_{s-}\,\dd [L^{\vp_1},L^{\vp_1}]_s\right]\\
&=\bbe\left[ V^\vp_r \int_{(0,r]} \bar V^{(1)}_{s-}\,\dd [L^{\vp_1},L^{\vp_1}]_s\right]-r \bbe[L_1^2]\bbe[\bar V^{(1)}_0]\bbe[V_0^\vp],
\end{align*}
where an application of the integration by parts formula yields
\begin{align*}
f(r):=&~\bbe\left[ V^\vp_r \int_{(0,r]} \bar V^{(1)}_{s-}\,\dd [L^{\vp_1},L^{\vp_1}]_s\right]\\
=&~\bbe[L_1^2]\int_{(0,r]} \bbe[V_s^\vp\bar V^{(1)}_s]\,\dd s + \beta\int_{(0,r]} \bbe\left[\int_{(0,s]} \bar V^{(1)}_{u-}\,\dd [L^{\vp_1},L^{\vp_1}]_u\right]\,\dd s \\
&\quad\quad\quad+\Psi(1,\vp)\int_{(0,r]} \bbe\left[V_s^\vp \int_{(0,s]} \bar V^{(1)}_{u-}\,\dd [L^{\vp_1},L^{\vp_1}]_u\right]\,\dd s + \bbe\left[\int_{(0,r]} \bar V^{(1)}_{s-}\,\dd \big[ S^{\vp_1},V^\vp\big]_s\right]\\
=&~\bbe[L_1^2]\bbe[V_0^\vp \bar V_0^{(1)}]r + \beta\bbe[L_1^2]\bbe[\bar V^{(1)}_0]\frac{r^2}{2} + \Psi(1,\vp)\int_{(0,r]} f(s)\,\dd s\\ &\quad + \mathds{1}_{\{\vp=\vp_1\}}\vp \int_\bbr y^2 \,\nu_S(\dd y) \bbe[V_0^\vp \bar V_0^{(1)}]r,\\
f(0)=&~0.
\end{align*}
Solving this integral equation yields ($m_2:=\int_\bbr y^2\,\nu_S(\dd y)$)
\begin{align*} f(r)&=\frac{(\bbe[L_1^2]+\mathds{1}_{\{\vp=\vp_1\}}\vp m_2)\bbe[V_0^\vp \bar V^{(1)}_0]\Psi(1,\vp)(\ee^{\Psi(1,\vp)r}-1)}{\Psi(1,\vp)^2}\\
&\quad\quad\quad+\frac{\beta\bbe[L_1^2]\bbe[\bar V^{(1)}_0](-\Psi(1,\vp)r+\ee^{\Psi(1,\vp)r}-1)}{\Psi(1,\vp)^2}, \end{align*}
which by \eqref{mean-cog} yields the claimed positive correlation, since
\begin{align}
\lefteqn{\cov[(\Delta^r G^{(1)}_0)^2,V^\vp_r]=f(r)-\bbe[L_1^2]\bbe[V_0^\vp]\bbe[\bar V^{(1)}_0]r} \label{analog}\\
=&~\frac{(\bbe[L_1^2]+\mathds{1}_{\{\vp=\vp_1\}}\vp m_2)\bbe[V_0^\vp \bar V^{(1)}_0]\Psi(1,\vp)(\ee^{\Psi(1,\vp)r}-1) + \beta\bbe[L_1^2]\bbe[\bar V^{(1)}_0](\ee^{\Psi(1,\vp)r}-1)}{\Psi(1,\vp)^2} \nonumber\\
=&~\frac{\ee^{\Psi(1,\vp)r}-1}{\Psi(1,\vp)}\left(\bbe[L_1^2]\cov[V_0^\vp,\bar V_0^{(1)}] + \mathds{1}_{\{\vp=\vp_1\}}\vp \int_\bbr y^2\,\nu_S(\dd y) \bbe[V_0^\vp \bar V^{(1)}_0]\right)\geq0 \nonumber
\end{align}
with  $\cov[V_0^\vp,\bar V_0^{(1)}]=\pi(\{\vp\})\var[V^\vp_1]$.
\end{proof}

\begin{proof}[Proof of Theorem \ref{prop-price2}]
The proof works similarly to the proof 
of Theorem \ref{price-supCOG1} with the obvious changes,
when independence of the single COGARCH processes was used (e.g. \eqref{eq-price-cases}). Also replace $\bbg^{(1)}$ by $\bbg^{(2)}=(\calg_t^{(2)})_{t\in\bbr}$, the augmented natural filtration of $L$, and notice that $\cov[V_0^\vp,\bar V^{(2)}_0]=\int_{\Phi_L^{(2)}} \cov[V_0^\vp,V_0^{\tvp}]\,\pi(\dd\tvp)>0$
by Proposition~\ref{eq-crosscor2}.
\end{proof}

\begin{proof}[Proof of Theorem \ref{prop-price3}]
Analogously to the proof of Theorem \ref{price-supCOG1}, one can show that (a) and (b) hold and that for (c) we have
\beq\label{aux-price3}
\E \left[(\Delta^r G^{(3)}_0)^2(\Delta^r G^{(3)}_{h})^2  \right] = \E[L_1^2] \E \left[ (\Delta^r G^{(3)}_0)^2 \int_{(h,h+r]} \E[ \bar{V}^{(3)}_{s-}| \cG^{(3)}_r ] \,\dd s \right],\eeq
where from \eqref{supcogarch3-explicitstationary} and \cite[Eq. (4.5)]{KLM:2004} we have
\begin{align*}
\bbe[\bar V^{(3)}_{s-}| \cG^{(3)}_r] &= \ee^{-\eta(s-r)}\bar V^{(3)}_r + \beta \ee^{-\eta s}\int_{(r,s)}\ee^{\eta u}\dd u + \bbe\left[\left.\int_{(r,s)} \int_{\Phi_L^{(2)}} \ee^{-\eta(s-u)}\vp V^\vp_{u-}\,\La^S(\dd u,\dd\vp) \right| \cG^{(3)}_r\right]\\
&=\ee^{-\eta(s-r)}\bar V^{(3)}_r + \frac{\beta}{\eta} (1-\ee^{-\eta(s-r)}) + \bbe[S_1]\int_{(r,s]} \int_{\Phi_L^{(2)}} \ee^{-\eta(s-u)}\vp\bbe[V^\vp_{u-} | \cG^{(3)}_r] \,\pi(\dd\vp)\,\dd u\\
&=\ee^{-\eta(s-r)}\bar V^{(3)}_r + \frac{\beta}{\eta} (1-\ee^{-\eta(s-r)})\\
&\quad +\bbe[S_1]\int_{(r,s]}\int_{\Phi_L^{(2)}} \ee^{-\eta(s-u)}\vp\big((V^\vp_r-\bbe[V_0^\vp])\ee^{(u-r)\Psi(1,\vp)}+\bbe[V_0^\vp]\big)\,\pi(\dd\vp)\,\dd u.
\end{align*}
Applying \eqref{mean-cog} we obtain
\begin{align*}
\lefteqn{ \bbe[S_1]\int_{(r,s]}\int_{\Phi_L^{(2)}} \ee^{-\eta(s-u)}\vp\big((V^\vp_r-\bbe[V_0^\vp])\ee^{(u-r)\Psi(1,\vp)}+\bbe[V_0^\vp]\big)\,\pi(\dd\vp)\,\dd u}\\
=&~\bbe[S_1]\int_{\Phi_L^{(2)}} \left( \frac{\vp(V^\vp_r-\bbe[V_0^\vp])}{\vp \E[S_1]} \left(\ee^{\Psi(1,\vp)(s-r)}-\ee^{-\eta(s-r)} \right) + \frac{\vp \E[V_0^\vp]}{\eta} \left(1- \ee^{-\eta(s-r)}\right)\right) \,\pi(\dd\vp) \\
=&~\int_{\Phi_L^{(2)}} \ee^{\Psi(1,\vp)(s-r)}(V_r^\vp-\bbe[V_0^\vp])\,\pi(\dd\vp)- \ee^{-\eta(s-r)}\left(\int_{\Phi_L^{(2)}} V_r^\vp\,\pi(\dd\vp) - \E[\bar V^{(3)}_0] \right)\\
&\quad +  \left(1- \ee^{-\eta(s-r)}\right) \int_{\Phi_L^{(2)}} \frac{\bbe[S_1] \vp }{\eta}\frac{-\beta}{ \Psi(1,\vp)}  \,\pi(\dd\vp)\\
=&~\int_{\Phi_L^{(2)}} \ee^{\Psi(1,\vp)(s-r)}(V_r^\vp-\bbe[V_0^\vp])\,\pi(\dd\vp)- \ee^{-\eta(s-r)}\left(\int_{\Phi_L^{(2)}} V_r^\vp\,\pi(\dd\vp) - \E[\bar V^{(3)}_0] \right)\\
&\quad - \frac{\beta}{\eta}\left(1- \ee^{-\eta(s-r)}\right) \int_{\Phi_L^{(2)}} \left(1+ \frac{\eta}{ \Psi(1,\vp)}\right)  \,\pi(\dd\vp)\\
=&~\int_{\Phi_L^{(2)}} \ee^{\Psi(1,\vp)(s-r)}(V_r^\vp-\bbe[V_0^\vp])\,\pi(\dd\vp)- \ee^{-\eta(s-r)} \int_{\Phi_L^{(2)}} V_r^\vp\,\pi(\dd\vp)
- \frac{\beta}{\eta}\left(1- \ee^{-\eta(s-r)}\right)  + \E[\bar V^{(3)}_0]
\end{align*}
such that
\begin{align*}
\bbe[\bar V^{(3)}_{s-}| \cG^{(3)}_r] &=\ee^{-\eta(s-r)}\left(\bar V^{(3)}_r - \int_{\Phi_L^{(2)}} V_r^\vp\,\pi(\dd\vp)\right)  +\int_{\Phi_L^{(2)}} \ee^{\Psi(1,\vp)(s-r)}(V_r^\vp-\bbe[V_0^\vp])\,\pi(\dd\vp)
+ \E[\bar V^{(3)}_0].
\end{align*}
Inserting this into \eqref{aux-price3} yields 
\begin{align*}
&~\cov[(\Delta^r G_0^{(3)})^2,(\Delta^r G_h^{(3)})^2] \\
=&~\bbe[L_1^2]\bbe\left[(\Delta^r G_0^{(3)})^2 \int_{(h,h+r]} \E[ \bar{V}^{(3)}_{s-}| \cG^{(3)}_r ] \,\dd s \right] - \bbe[(\Delta^r G_0^{(3)})^2]^2\\
=&~\bbe[L^2_1]\bbe\left[(\Delta^r G_0^{(3)})^2 \left( \frac{\ee^{-\eta h} - \ee^{-\eta(h-r)}}{-\eta} \left(\bar V^{(3)}_r-\int_{\Phi_L^{(2)}} V_r^\vp\,\pi(\dd\vp)\right) \right)\right.\\
&\quad\quad\quad+\left.\int_{\Phi_L^{(2)}} \frac{\ee^{\Psi(1,\vp)h}-\ee^{\Psi(1,\vp)(h-r)}}{\Psi(1,\vp)}(V_r^\vp-\bbe[V^\vp_0])\,\pi(\dd\vp) \right]\\
=&~\bbe[L^2_1]\left[\frac{\ee^{-\eta h}-\ee^{-\eta(h-r)}}{-\eta}\cov[(\Delta^r G_0^{(3)})^2,\bar V^{(3)}_r]  \right.\\
&\quad\quad\quad+\left.\int_{\Phi_L^{(2)}} \left(\frac{\ee^{\Psi(1,\vp)h}-\ee^{\Psi(1,\vp)(h-r)}}{\Psi(1,\vp)}-\frac{\ee^{-\eta h}-\ee^{-\eta(h-r)}}{-\eta}\right)\cov[(\Delta^r G_0^{(3)})^2,V^\vp_r]\,\pi(\dd\vp)\right].
\end{align*}
Since $\Psi(1,\vp)>-\eta$ and the function $x\mapsto (\ee^{hx}-\ee^{(h-r)x})/x$ is increasing in $x$ for $x<0$, it remains to prove $\cov[(\Delta^r G_0^{(3)})^2,\bar V^{(3)}_r]>0$ and $\cov[(\Delta^r G_0^{(3)})^2,V^\vp_r]>0$. For the latter one, proceed as in the proof of Theorem \ref{price-supCOG1} and note that $\cov[V_0^\vp,\bar V_0^{(3)}]>0$. Indeed, using integration by parts, 
\begin{align*}
V^\vp_r\bar V^{(3)}_r&=V^\vp_0\bar V^{(3)}_0 + \int_{(0,r]} \bar V^{(3)}_{s-}\,\dd V^\vp_s + \int_{(0,r]} V^\vp_{s-}\,\dd \bar V^{(3)}_s + [V^\vp,\bar V^{(3)}]_r\\
&=V^\vp_0\bar V^{(3)}_0 + \int_{(0,r]} \bar V^{(3)}_{s}(\beta-\eta V^\vp_s)\,\dd s + \int_{(0,r]} \bar V^{(3)}_{s-} \vp V^\vp_{s-}\,\dd S_s+ \int_{(0,r]} V^\vp_{s-}(\beta-\eta \bar V^{(3)}_s)\,\dd s \\
&\quad\quad\quad + \int_{(0,r]} \int_{\Phi_L^{(2)}} V^\vp_{s-}\tvp V^\tvp_{s-}\,\La^S(\dd s,\dd\tvp) + \vp\int_{(0,r]} \int_{\Phi_L^{(2)}} \int_{\bbr_+} V^\vp_{s-}\tvp V^\tvp_{s-} y^2\,\mu^{\La^S}(\dd s,\dd\tvp,\dd y),
\end{align*}
with $[V^\vp,\bar V^{(3)}]_r$ as given in Lemma \ref{quadvar}. Taking expectations, differentiating w.r.t. $r$ and using the stationarity of $V^\vp\bar V^{(3)}$, which is a consequence of Lemma \ref{lem1}, we find that ($m_1:=\int_{\bbr_+} y\,\nu_S(\dd y)$ and $m_2:=\int_{\bbr_+} y^2\,\nu_S(\dd y)$)
\[ \beta(\bbe[\bar V_0^{(3)}]+\bbe[V_0^\vp]) + (\vp m_1-2\eta)\bbe[V_0^\vp \bar V^{(3)}_0] + (m_1+\vp m_2)\int_{\Phi_L^{(2)}} \tvp\bbe[V_0^\vp V_0^\tvp]\,\pi(\dd\tvp)=0, \]
which implies that
\begin{align*}
\lefteqn{ \cov[\bar V^{(3)}_0,V^\vp_0]}\\
=&~\frac{\beta(\bbe[\bar V_0^{(3)}]+\bbe[V_0^\vp]) + (m_1+\vp m_2)\int_{\Phi_L^{(2)}} \tvp\bbe[V_0^\vp V_0^\tvp]\,\pi(\dd\tvp)-(\eta-\Psi(1,\vp))\bbe[\bar V^{(3)}_0]\bbe[V_0^\vp]}{\eta-\Psi(1,\vp)}.
\end{align*}
To show the positivity of this term, we only have to consider the numerator, which by \eqref{mean-cog}, \eqref{mean-supcog3} and \eqref{eq-h-explicit} can be simplified to
\begin{align*}
&~\int_{\Phi_L^{(2)}} (m_1+\vp m_2)\tvp \cov[V_0^\vp,V_0^\tvp]\,\pi(\dd\tvp) + \beta(\bbe[\bar V_0^{(3)}]+\bbe[V_0^\vp]) + h(\vp,\tvp)\bbe[V_0^\vp]\bbe[\bar V_0^{(3)}]\\
=&~\int_{\Phi_L^{(2)}} (m_1+\vp m_2)\tvp \cov[V_0^\vp,V_0^\tvp]\,\pi(\dd\tvp) + \beta^2 m_2\int_{\Phi_L^{(2)}} \frac{\vp\tvp}{\Psi(1,\vp)\Psi(1,\tvp)}\,\pi(\dd\tvp) > 0.
\end{align*}
Finally, 
using the same methods as in the proof of Theorem \ref{price-supCOG1}, one can derive the following analogue of Eq.~\eqref{analog}:
\[ \cov[(\Delta^r G^{(3)}_0)^2,\bar V^{(3)}_0] = g(r)-\bbe[L^2_1]\bbe[\bar V_0]^2 r,\]
where
\begin{align*} g(r)&=\ee^{-\eta r} \left(\int_{(0,r]} \ee^{\eta s}\left(a+bs+\int_{\Phi_L^{(2)}} m_1\vp f(\vp,s)\,\pi(\dd\vp)\right)\,\dd s\right),\quad r\geq0,\\
 a &= \bbe[L_1^2]\bbe[(\bar V^{(3)}_0)^2] + \int_{\bbr_+} y^2\,\nu_S(\dd y) \int_{\Phi_L^{(2)}} \vp\bbe[V_0^\vp \bar V_0]\,\pi(\dd\vp),\\
 b &= \beta\bbe[L_1^2]\bbe[\bar V^{(3)}_0]\quad\text{and} \quad f(\vp,r) = \bbe\left[V_r^\vp \int_{(0,r]} \bar V^{(3)}_{u-}\,\dd [L,L]_u\right].
\end{align*}
The positivity now follows from
\[ \cov[(\Delta^r G^{(3)}_0)^2,\bar V^{(3)}_0] \geq \ee^{\eta r}\int_{(0,r]} \ee^{-\eta s}\,\dd s \bbe[L_1^2]\bbe[(\bar V^{(3)}_0)^2]-\bbe[L^2_1]\bbe[\bar V^{(3)}_0]^2 r \]
and the fact that $\ee^{\eta r}\int_{(0,r]} \ee^{-\eta s}\,\dd s = (\ee^{\eta r}-1)/\eta > r$ for all $r>0$.
\end{proof}

\section*{Acknowledgements}
The authors take pleasure in thanking Jean Jacod for inspiring and clarifying discussions. The second author acknowledges support from the graduate program TopMath at Technische Universit\"at M\"unchen.

\bibliography{bib-supCOGARCH}

\begin{thebibliography}{10}

\bibitem{AJ}
Y.~A{\"i}t-Sahalia and J.~Jacod.
\newblock {\em High-Frequency Financial Econometrics}.
\newblock Princeton University Press, Princeton, 2014.

\bibitem{BNsuper}
O.E. Barndorff-Nielsen.
\newblock Superposition of {O}rnstein-{U}hlenbeck type processes.
\newblock {\em Theory Probab. Appl.}, 45(2):175--194, 2001.

\bibitem{BN-S:2004}
O.E. Barndorff-Nielsen and J.~Schmiegel.
\newblock L{\'e}vy-based spatial-temporal modelling, with applications to
  turbulence.
\newblock {\em Russ. Math. Surv.}, 59(1):65--90, 2004.

\bibitem{Barn:Shep:2001}
O.E. Barndorff-Nielsen and N.~Shephard.
\newblock Non-{G}aussian {O}rnstein-{U}hlenbeck based models and some of their
  uses in financial economics.
\newblock {\em J. R. Stat. Soc. Ser. B Stat. Methodol.}, 63(2):167--241, 2001.

\bibitem{Barn:Shep:2012}
O.E. Barndorff-Nielsen and N.~Shephard.
\newblock Econometric analysis of realized volatility and its use in estimating
  stochastic volatility models.
\newblock {\em J. R. Stat. Soc. Ser. B Stat. Methodol.}, 64(2):253--280, 2012.

\bibitem{BN-S:2011}
O.E. Barndorff-Nielsen and R.~Stelzer.
\newblock Multivariate sup{OU} processes.
\newblock {\em Ann. Appl. Probab.}, 21(1):140--182, 2011.

\bibitem{behme}
A.~Behme.
\newblock Distributional properties of solutions of
  $\mathrm{d}{V}_t={V}_{t-}\mathrm{d}{U}_t+\mathrm{d}{L}_t$ with {L}{\'e}vy
  noise.
\newblock {\em Adv. Appl. Probab.}, 43(3):688--711, 2011.

\bibitem{behme_diss}
A.~Behme.
\newblock {\em Generalized Ornstein-Uhlenbeck Processes and Extensions}.
\newblock PhD thesis, Technische Universit{\"a}t Braunschweig, 2011.

\bibitem{ben:1995}
T.~Bollerslev, R.F. Engle, and D.B. Nelson.
\newblock {ARCH} models.
\newblock In R.F. Engle and D.~McFadden, editors, {\em Handbook of
  Econometrics}, volume~4, pages 2959--3038. North-Holland, Amsterdam, 1994.

\bibitem{brockwell5}
P.J. Brockwell.
\newblock L{\'e}vy-driven {CARMA} processes.
\newblock {\em Ann. Inst. Stat. Math.}, 53(1):113--124, 2001.

\bibitem{BCL:2006}
P.J. Brockwell, E.~Chadraa, and A.~Lindner.
\newblock Continuous-time {G}{A}{R}{C}{H} processes.
\newblock {\em Ann. Appl. Probab.}, 16(2):790--826, 2006.

\bibitem{Chong13}
C.~Chong and C.~Kl{\"u}ppelberg.
\newblock Integrability conditions for space-time stochastic integrals: theory
  and applications.
\newblock {\em Bernoulli}, 2014.
\newblock Accepted for publication.

\bibitem{EKM}
P.~Embrechts, C.~Kl{\"u}ppelberg, and T.~Mikosch.
\newblock {\em Modelling Extremal Events for Insurance and Finance}.
\newblock Springer, Berlin, 1997.

\bibitem{Fasen}
V.~Fasen.
\newblock Extremes of continuous-time processes.
\newblock In T.G. Andersen, R.A. Davis, J.-P. Kreiss, and T.~Mikosch, editors,
  {\em Handbook of Financial Time Series}, pages 653--667. Springer,
  Heidelberg, 2009.

\bibitem{FK}
V.~Fasen and C.~Kl{\"u}ppelberg.
\newblock Extremes of sup{OU} processes.
\newblock In F.E. Benth, G.~{Di Nunno}, T.~Lindstr{\o}m, B.~{\O}ksendal, and
  T.~Zhang, editors, {\em Stochastic Analysis and Applications: The Abel
  Symposium 2005}, pages 340--359. Springer, Heidelberg, 2007.

\bibitem{FZ}
C.~Francq and J.-M. Zakoian.
\newblock {\em {GARCH} {M}odels}.
\newblock Wiley, Chichester, 2010.

\bibitem{Griffin:2011}
J.E. Griffin.
\newblock Inference in infinite superpositions of non-{G}aussian
  {O}rnstein-{U}hlenbeck processes using {B}ayesian nonparametic methods.
\newblock {\em J. Finan. Econom.}, 9(3):519--549, 2011.

\bibitem{Griffin:Steel:2010}
J.E. Griffin and M.F.J. Steel.
\newblock Bayesian inference with stochastic volatility models using continuous
  superpositions of non-{G}aussian {O}rnstein-{U}hlenbeck processes.
\newblock {\em Comput. Stat. Data Anal.}, 54(11):2594--2608, 2010.

\bibitem{JKMc}
J.~Jacod, C.~Kl{\"u}ppelberg, and G.~M{\"u}ller.
\newblock Functional relationships between price and volatility jumps and its
  consequences for discretely observed data.
\newblock {\em J. Appl. Probab.}, 49(4):901--914, 2012.

\bibitem{JS2}
J.~Jacod and A.N. Shiryaev.
\newblock {\em Limit Theorems for Stochastic Processes}.
\newblock Springer, Berlin, 2nd edition, 2003.

\bibitem{Jacod10}
J.~Jacod and V.~Todorov.
\newblock Do price and volatility jump together?
\newblock {\em Ann. Appl. Probab.}, 20(4):1425--1469, 2010.

\bibitem{Kallenberg02}
O.~Kallenberg.
\newblock {\em Foundations of Modern Probability}.
\newblock Springer, New York, 2nd edition, 2002.

\bibitem{KLM:2004}
C.~Kl{\"u}ppelberg, A.~Lindner, and R.~Maller.
\newblock A continuous-time {GARCH} process driven by a {L}{\'e}vy process:
  stationarity and second-order behaviour.
\newblock {\em J. Appl. Probab.}, 41(3):601--622, 2004.

\bibitem{klm:2006}
C.~Kl{\"u}ppelberg, A.~Lindner, and R.~Maller.
\newblock Continuous time volatility modelling: {COGARCH} versus
  {O}rnstein-{U}hlenbeck models.
\newblock In Y.~Kabanov, R.~Liptser, and J.~Stoyanov, editors, {\em From
  Stochastic Calculus to Mathematical Finance. The Shiryaev Festschrift}, pages
  393--419, Berlin, 2006. Springer.

\bibitem{lindner:maller:2005}
A.~Lindner and R.~Maller.
\newblock L{\'e}vy integrals and the stationarity of generalised
  {O}rnstein-{U}hlenbeck processes.
\newblock {\em Stoch. Process. Appl.}, 115(10):1701--1722, 2005.

\bibitem{Protter04}
P.E. Protter.
\newblock {\em Stochastic Integration and Differential Equations.}
\newblock Springer, Berlin, 2nd edition, 2004.

\bibitem{Rajput:1989}
B.S. Rajput and J.~Rosi{\'n}ski.
\newblock Spectral representations of infinitely divisible processes.
\newblock {\em Prob. Theory Relat. Fields}, 82(3):451--487, 1989.

\bibitem{sato}
K.-I. Sato.
\newblock {\em L{\'e}vy Processes and Infinitely Divisible Distributions.}
\newblock Cambridge University Press, Cambridge, 1999.

\bibitem{STW:13}
R.~Stelzer, T.~Tosstorff, and M.~Wittlinger.
\newblock Moment based estimation of sup{O}{U} processes and a related
  stochastic volatility model.
\newblock Submitted, preprint available under arXiv:1305.1470 [math.PR], 2013.

\bibitem{steutel}
F.W. Steutel and K.~van Harn.
\newblock {\em Infinite Divisibility of Probability Distributions on the Real
  Line}.
\newblock Marcel Dekker, New York, 2004.

\bibitem{todorov:tauchen:2006}
V.~Todorov and G.~Tauchen.
\newblock Simulation methods for {L}{\'e}vy-driven continuous-time
  autoregressive moving average {(CARMA)} stochastic volatility models.
\newblock {\em J. Bus. Econ. Stat.}, 24(4):455--469, 2006.

\bibitem{Walsh}
J.B. Walsh.
\newblock An introduction to stochastic partial differential equations.
\newblock In P.L. Hennequin, editor, {\em {\'E}cole d'{\'E}t{\'e} de
  Probabilit{\'e}s de Saint Flour XIV - 1984}, pages 265--439. Springer,
  Berlin, 1986.

\end{thebibliography}
\bibliographystyle{plain}

\end{document}